\documentclass{amsart}

\usepackage{amsfonts}
\usepackage{graphicx}
\usepackage{tabularx}
\usepackage{array}
\usepackage[usenames,dvipsnames]{color}
\usepackage{comment}
\usepackage{amsmath}
\usepackage{amsthm}
\usepackage{amssymb}
\usepackage{fullpage}
\usepackage{xcolor}
\usepackage[shortlabels]{enumitem}
\usepackage{tikz}
\usepackage{circuitikz}
\usepackage{ifthen}

\renewcommand{\phi}{\varphi}
\usetikzlibrary{decorations.pathreplacing}

\tikzstyle{legend_general}=[rectangle, rounded corners, thin,
                          top color= white,bottom color=lavander!25,
                          minimum width=2.5cm, minimum height=0.8cm,
                          violet]

\DeclareMathOperator{\mc}{c}
\DeclareMathOperator{\id}{id}

\renewcommand{\epsilon}{\varepsilon}

\newtheorem{theorem}{Theorem}[section]

\newtheorem{observation}[theorem]{Observation}
\newtheorem{question}[theorem]{Question}

\newtheorem{claim}[theorem]{Claim}
\newtheorem{lemma}[theorem]{Lemma}
\newtheorem{corollary}[theorem]{Corollary}
\theoremstyle{definition}
\newtheorem{definition}[theorem]{Definition}

\usepackage[utf8]{inputenc}

\title{Hamiltonicity of covering graphs of trees}
\author{Peter Bradshaw}
\address{Department of Mathematics, University of Illinois Urbana-Champaign, Urbana, IL, USA}
\email{pb38@illinois.edu}

\author{Zhilin Ge}
\address{Department of Mathematics, Simon Fraser University, Burnaby, BC, Canada}
\email{zge@sfu.ca}

\author{Ladislav Stacho}
\address{Department of Mathematics, Simon Fraser University, Burnaby, BC, Canada}
\email{ladislav\_stacho@sfu.ca}

\thanks{Peter Bradshaw received support from NSF RTG grant DMS-1937241 }

\begin{document}

\maketitle
\section*{Abstract}
%The Hamiltonicity of highly symmetric graphs has been a popular subject of research, with much focus on vertex-transitive graphs and in particular on Cayley graphs.
%In this paper, we focus on the highly symmetric class of covering graphs, specifically those obtained by lifting a tree, as a voltage graph, over a cyclic group. Batagelj and Pisanski \cite{Batagelj} were the first to provide sufficient conditions for the Hamiltonicity of such graphs, in particular the Cartesian product of a tree and a cycle.  Later, Hell, Nishiyama, and Stacho \cite{HellVoltage} proposed a generalization of the result. 

In this paper, we consider covering graphs obtained by lifting a tree with a loop at each vertex as a voltage graph over a cyclic group. We generalize a tool of Hell, Nishiyama, and Stacho \cite{HellVoltage}, known as the billiard strategy, for constructing Hamiltonian cycles in the covering graphs of paths. We show that our extended tool can be used to provide new sufficient conditions for the Hamiltonicity of covering graphs of trees that are similar to those  of Batagelj and Pisanski \cite{Batagelj} and of Hell, Nishiyama, and Stacho \cite{HellVoltage}.
Next, we focus specifically on covering graphs obtained from trees lifted as voltage graphs over cyclic groups $\mathbb Z_p$ of large prime order $p$. We prove that for a given reflexive tree $T$ 
whose edge labels are assigned uniformly at random from a finite set, the corresponding lift is almost surely Hamiltonian for a large enough prime-ordered cyclic group $\mathbb Z_p$. Finally, we show that if a reflexive tree $T$ is lifted over a group $\mathbb Z_p$ of a large prime order, then for any assignment of nonzero elements of $\mathbb Z_p$ to the edges of $T$, the corresponding cover of $T$ has a large circumference.

\section{Introduction}
Given a graph $H$, an \emph{automorphism} on $H$ is a bijective function $\phi:V(H) \rightarrow V(H)$ such that for each vertex pair $u,v \in V(H)$, $[u,v] \in E(H)$ if and only if $[\phi(u),\phi(v)] \in E(H)$. (We use $[u,v]$ to denote an edge and $(u,v)$ to denote a directed arc.) If, for every vertex pair $u,v \in V(H)$, there exists an automorphism on $H$ satisfying $\phi(u) = v$, then we say that $H$ is \emph{vertex-transitive}.
In 1969, Lov\'asz asked whether every connected vertex-transitive graph has a Hamiltonian path \cite{Lovasz1969combinatorial}, and since then, this question has led to considerable interest in questions about Hamiltonian paths and cycles in graphs with a high degree of symmetry. %, such as vertex-transitive graphs. 
While Lov\'asz's question has not yet been answered affirmatively, no vertex-transitive graph without a Hamiltonian path has yet been found, and there are only four known vertex-transitive graphs (on at least three vertices) with no Hamiltonian cycle \cite{kutnar2009hamilton}, including the Petersen graph.

One particular type of vertex-transitive graph is a \emph{Cayley graph}, which is defined as follows. Given a group $G$ and its generating set $S$, the Cayley graph $\Gamma$ generated by $G$ and $S$ is defined as a %vertex-transitive 
graph for which $V(\Gamma) = G$ and $E(\Gamma) = \{[v,vs] : v\in G \textrm{ and }s\in S\}$.
The simple fact that none of the four known vertex-transitive non-Hamiltonian graphs is a Cayley graph has led to a folklore conjecture that every Cayley graph is Hamiltonian.
This conjecture is mentioned, for example, by Kutnar and Maru\v{s}i\v{c}~\cite{kutnar2009hamilton}. 
Babai \cite{babai1979long} 
gave a partial answer to Lov\'asz's question by proving
that every connected vertex‐transitive graph on $n\geq 4$ vertices has a circumference greater than $\sqrt{3n}$. 
On the other hand, Babai in \cite{babai1979problem} 
conjectured that there exists a positive constant $\epsilon>0$ for which infinitely many connected vertex-transitive graphs $G$ have a circumference of at most $(1-\epsilon)|V(G)|$. 
Further research has focused on answering special cases of Lov\'asz's question about Hamiltonian paths. 
For example, Witte \cite{Witte} proved that every directed Cayley graph on a $p$-group has a directed Hamiltonian cycle, 
and Keating and Witte \cite{Keating} proved that a Cayley graph over a group $G$ is Hamiltonian whenever the commutator subgroup of $G$ is cyclic and of prime-power order.
In this paper, we focus on the Hamiltonicity problem in \emph{covering graphs} (also called \emph{lifts}) of \emph{voltage graphs}. 
Covering graphs are not vertex-transitive graphs in general, but they are still highly symmetric, and hence they share many properties with vertex-transitive graphs.

Informally, a \emph{covering graph} is a graph generated from a group $G$, a \emph{base graph} $\Gamma$, and an assignment
% {\LS Here we are using $\phi$ but later $\sigma$.} 
$\sigma:E(\Gamma) \rightarrow G$, by making one copy of each vertex of $\Gamma$ for each element of $G$ and then by adding edges according to certain rules based on the labels given by $\sigma$.
The term ``covering graph" comes from the fact that when a base graph $\Gamma$ is considered as a $1$-dimensional simplicial complex with the standard topology, a covering graph of $\Gamma$ gives a covering space for $\Gamma$. 
Covering graphs were originally introduced by Gross to describe graph embeddings on surfaces \cite{Gross}. Covering graphs have since gained wider attention and have been used, for instance, to generate graphs of large girth on a small number of vertices \cite{Exoo}, as well as counterexamples to a conjecture of Greenwell and Kronk about edge colorings and Hamiltonicity in cubic graphs \cite{Thomason}.

We formally define a covering graph as follows. Let $\Gamma$ be an undirected graph with possible loops. For each edge $[u,v] \in E(\Gamma)$, we say that $\Gamma$ has corresponding arcs $(u,v)$ and $(v,u)$, which correspond to the two opposite directions in which the edge $[u,v]$ can be traversed. 
Accordingly, we define the \emph{arc set} of $\Gamma$ as $A(\Gamma) = \{(u,v), (v,u): [u,v] \in E(\Gamma)\}.$ If $[v,v]$ is a loop of $\Gamma$, then we let $A(\Gamma)$ contain two elements $(v,v)$ corresponding to $[v,v]$. Then, given a group $G$, we say that a \emph{voltage assignment} on $\Gamma$ is a function $\sigma:A(\Gamma) \rightarrow G$ that satisfies $\sigma(u,v) = \sigma(v,u)^{-1}$ for every edge $[u,v] \in E(\Gamma)$, and such that $\sigma$ assigns inverse elements to each pair of loops $(v,v)$.
We often refer to the values assigned by $\sigma$ as \emph{labels}.
 For a voltage assignment $\sigma$, we say that the pair $(\Gamma, \sigma)$ is a \emph{voltage graph}. 
 Given a voltage graph $(\Gamma, \sigma)$, where $\sigma:A(\Gamma) \rightarrow G$ is a voltage assignment to a group $G$,
%with group operation $"+"$\footnote{Besides addition, other group operations are also possible. We only define it with operation $"+"$ because we focus on the case group $G$ is cyclic and use $\mathbb{Z}_n$ as the example cyclic group throughout the paper.}
we define the \emph{covering graph} of $(\Gamma,\sigma)$, written $\Gamma^{\sigma}$, as follows:

\begin{itemize}
    \item $V(\Gamma^{\sigma}) = V(\Gamma) \times G$.
    \item For any vertex pair $u,v \in V(\Gamma)$ and element pair $a,b \in G$, $(u,a)$ and $(v,b)$ are adjacent in $\Gamma^{\sigma}$ if and only if there exists an arc $e = (u,v) \in A(\Gamma)$ satisfying $\sigma(e) = a^{-1}b $.% or there exists an arc $e = vu \in E(\Gamma)$ with $\sigma(e) = a - b$.
\end{itemize}
We observe that every Cayley graph $H$ is a covering graph of a graph on a single vertex, with one loop for each generator of $H$.

Whenever we have a voltage graph $(\Gamma,\sigma)$ containing
a loop $[v,v]$ in its edge set, $\Gamma$ then contains 
two directed loops $(v,v)$ in its arc set. We say that one of these loops is \emph{primary} and the other is \emph{secondary}. When we write $\sigma(v,v)$, we refer to the voltage assigned to the primary loop $(v,v)$. This removes ambiguity from the notation $\sigma(v,v)$. Also, we often abbreviate $\sigma(v,v)$ as $\sigma(v)$.
%
%For each vertex $v \in V(\Gamma)$ and element $a \in G$, there is a vertex $(v,a)$ in $\Gamma^{\sigma}$.
%, and we will often write $v_a$ instead of $(v,a)$. 
For a vertex $v \in V(\Gamma)$, we write $v^{\sigma}$ for the subgraph of $\Gamma^{\sigma}$ induced by the vertex set $\{(v,g) : g \in G\}$, 
and we say that $v^{\sigma}$ is the \emph{fiber} over $v$. Note that if $v$ has no loop in ${\Gamma}$, then $v^{\sigma}$ is an independent set of $|G|$ vertices.

In all voltage graphs that we consider, $G$ is  a cyclic group with a binary operator $+$. For each element $a \in G$, we
use $-a$ to denote the group inverse of $a$. We use $\mathbb{Z}_n$ to represent the cyclic additive group $\mathbb{Z}/n\mathbb{Z}$ of $n$ elements, and we use $\mathbb{Z}_p$ to represent the cyclic additive group $\mathbb{Z}/p\mathbb{Z}$ of $p$ elements, where $p$ is some prime number.

We observe that for a graph $\Gamma$ with a voltage assignment $\sigma: A(\Gamma) \rightarrow G$ for some cyclic group $G$, the covering graph $\Gamma^{\sigma}$ has a large number of automorphisms. 
In particular, for each element $a \in G$, the function $f_a:V(\Gamma^{\sigma}) \rightarrow V(\Gamma^{\sigma})$ mapping $(v,g) \mapsto (v,g+a)$ 
is a graph automorphism. Hence, for each subgraph of $H \subseteq \Gamma^{\sigma}$, the automorphism $f_a$ transforms $H$ into a graph isomorphic to $H$.

We say that a graph $\Gamma$ is \emph{reflexive} if there exists exactly one loop at every vertex of $\Gamma$. (Our definition of a reflexive graph is slightly stronger than the traditional definition of a reflexive graph, which requires \emph{at least} one loop at every vertex of the graph.)
 In \cite{Batagelj}, Batagelj and Pisanski ask when the Cartesian product of a tree $T$ and a cycle $C_n$ contains a Hamiltonian cycle. The graph $T \times C_n$ can be seen as the lift of a reflexive tree $\Gamma$, isomorphic to $T$ (after deleting all loops from $\Gamma$), with a voltage assignment $\sigma: A(\Gamma) \rightarrow \mathbb{Z}_n$, where each arc $a$ of $\Gamma$ corresponding to a cut-edge is assigned $\sigma(a) = 0$, and each loop pair $\ell, \ell'$ of $\Gamma$ is assigned $\sigma(\ell) = 1$ and $\sigma(\ell') = -1$.
 With $\sigma$ defined this way, the authors give a necessary and sufficient condition for the lift $\Gamma^{\sigma}$ to be Hamiltonian. Later, Hell et al.~\cite{HellVoltage} considered lifts of reflexive trees $\Gamma$ with more general voltage assignments $\sigma$ on $\mathbb Z_n$. They allowed every loop of $\Gamma$ be assigned a value coprime to $n$, and gave a necessary and sufficient conditions for the lift $\Gamma^{\sigma}$ to be Hamiltonian. 
 
 In this paper, we aim to develop new tools for studying the covering graphs of reflexive trees and to find further conditions under which the lift of a reflexive tree with a voltage assignment on a cyclic group
is Hamiltonian. 
Throughout the paper, we consider the following question: 
\begin{quote}
    Given a reflexive tree $\Gamma$, a cyclic group $G$, and a voltage assignment $\sigma:A(\Gamma) \rightarrow G$, under what conditions does the covering graph $\Gamma^{\sigma}$
    contain a Hamiltonian cycle?
\end{quote}

The results of the paper are organized as follows. In Section \ref{sec:tools}, we
extend a method of Hell, Nishiyama, and Stacho \cite{HellVoltage} in order to develop new sufficient conditions
for the Hamiltonicity of the lift of a tree based on a decomposition of the tree into paths similar to the one in \cite{HellVoltage}. 
In Section \ref{sec:large}, we consider a reflexive tree lifted over a cyclic group $\mathbb Z_p$ of prime order, and we find a more relaxed sufficient condition for the Hamiltonicity of such a covering graph. 
In particular, we prove that for a given reflexive tree $T$ 
whose edge labels are assigned uniformly at random from a finite set,
the corresponding lift is almost surely Hamiltonian for a large enough prime-ordered cyclic group $\mathbb Z_p$.
Finally, in Section \ref{sec:circumference}, we show that if the group $\mathbb Z_p$ is of a large enough prime order, then for any assignment of nonzero 
elements in $\mathbb Z_p$ to the edges of $\Gamma$, the lift $\Gamma^{\sigma}$ has a large circumference.

\section{Extending the billiard strategy}
\label{sec:tools}
In this section, we generalize a tool of Hell, Nishiyama, and Stacho \cite{HellVoltage}, known as the \emph{billiard strategy}, which gives a sufficient condition for when the covering graph of a voltage path is Hamiltonian. Our generalized tool can be used to give many new sufficient conditions for the Hamiltonicity of lifts of voltage graphs.
We give one such new sufficient condition for when the covering graph of a reflexive path is Hamiltonian, and we give a second  sufficient condition for Hamiltonicity based on path decomposition, which is of a similar flavor to results from \cite{Batagelj} and \cite{HellVoltage}.

%In this section, we first investigate under what conditions the covering graphs of reflexive paths are Hamiltonian.
%We also show that reflexive trees that can be decomposed into paths satisfying appropriate conditions have Hamiltonian covering graphs. 
%Our results are of similar flavour to those in \cite{Batagelj} and \cite{HellVoltage}.

Throughout this section, we use the following lemma, which was proved in \cite{HellVoltage} and allows us to make some general assumptions on certain labels in voltage graphs of reflexive trees.

\begin{lemma}
\label{lemZero}
Let $\Gamma$ be a graph, and let $(u,v) \in A(\Gamma)$ be an arc corresponding to a cut-edge $[u,v] \in E(\Gamma)$. Let $G$ be a group, and for a pair $g,h \in G$, let $\sigma_g:A(\Gamma) \rightarrow G$ and $\sigma_h:A(\Gamma) \rightarrow G$ be voltage assignments. Suppose that $\sigma_g$ and $\sigma_h$ satisfy the following properties:
\begin{itemize}
    \item $\sigma_g(u,v) = g$, and $\sigma_h(u,v) = h$;
    \item For each arc $e \in A(\Gamma)$ satisfying $e \not \in \{(u,v), (v,u)\}$, $\sigma_g(e) = \sigma_h(e)$.
\end{itemize}
Then $\Gamma^{\sigma_g} \cong \Gamma^{\sigma_h}$.
\end{lemma}
In particular, Lemma \ref{lemZero} tells us that whenever we consider a reflexive tree $\Gamma$ with a voltage assignment $\sigma$, we may assume that $\sigma(e) = 0$ for each cut-edge $e \in A(\Gamma)$.
This assumption makes our analysis considerably simpler.

The billiard strategy from \cite{HellVoltage} is a tool for finding Hamiltonian cycles in the covering graphs of reflexive paths. We extend this technique into a more general method in the following lemma. 
Roughly speaking, given a covering graph of a reflexive path $\Gamma$ over a cyclic group, the following lemma guarantees the existence of a family $\mathcal P$ of paths in the lift of $\Gamma$ such that the paths in $\mathcal P$ include all vertices in the fibers over internal vertices of $\Gamma$, and such that the endpoints of paths in $\mathcal P$ appear 
%consecutively {\PB rethink}
%\footnote{Given the lift $\Gamma^{\sigma}$ of a voltage graphs $(\Gamma, \sigma)$ over a cyclic group $\mathbb Z_n$, two vertices $(v,a), (v,a') \in V(\Gamma^{\sigma})$ are \textit{consecutive} if $a-a' \in \{1,-1\}$. Note that the \textit{consecutive edges} will be defined differently from consecutive vertices later in the paper.} 
in the fiber over each endpoint of $\Gamma$ at voltages forming an arithmetic progression in $\mathbb Z_n$. The original billiard strategy in \cite{HellVoltage} can be obtained from our following lemma by requiring that every voltage assignment on the path $\Gamma$ be coprime to our group size $n$ and then setting $|\mathcal P| = 2$.

For a path $P$ with endpoints $u$ and $v$, 
we often 
give a direction to $P$ and
say that $P$ \emph{begins} at $u$ and \emph{ends} at $v$, 
or that $P$ begins at $v$ and ends at $u$. 
For a path $P$ that begins at $u$ and ends at $v$, if $S$ is a vertex set satisfying $S \cap V(P) \neq \emptyset$, 
then we say that $P$ \emph{arrives} in $S$ at $w$ if $w \in V(P) \cap S$ and $w$ is at a minimum distance from $u$ along $P$ among all vertices in $V(P) \cap S$.

\begin{lemma}
\label{lemma:billiardPaths}
For an integer $m \geq 1$, let $\Gamma = (v_1, \dots, v_m)$ be a reflexive path, and let $\sigma:A(\Gamma) \rightarrow \mathbb{Z}_n$ be a voltage assignment. 
Let $l, r \in [0 ,n-1]$ and $d \in [1,n/\gcd(r,n)]$ be constants. 
Then, there exists a family of $d$ vertex-disjoint paths $\mathcal P = \{P_0, \dots, P_{d-1}\}$ in $\Gamma^{\sigma}$ satisfying the following properties:
\begin{itemize}
    \item
    \label{item:first_bullet}
    The paths $P_0, \dots, P_{d-1}$ begin 
    at the vertices $(v_1, l), (v_1, l+r), \dots, (v_1, l+(d-1)r)$, respectively.
    \item For each $2 \leq t \leq m$, the paths of $\mathcal P$ arrive in the fiber $v_t^{\sigma}$ at a set of $d$ vertices $\{(v_t, i_t), (v_t, i_t+r), \dots, (v_t, i_t+(d-1)r)\}$, for some value $i_t \in \mathbb{Z}_n$, where addition is calculated in $\mathbb Z_n$.
    \item For each $2 \leq t \leq m-1$, if a path $P\in \mathcal P$ visits a component $K$ of a fiber $v_t^\sigma$, then every vertex of $K$ is visited by a path from $\mathcal P $.
    \item For each $1 \leq t \leq m-1$, after a path $P\in \mathcal P $ leaves a fiber $v_t^\sigma$, $P$ never returns to $v_t^\sigma$.
\end{itemize}
\end{lemma}

\begin{proof}
By Lemma \ref{lemZero}, we may assume that $\sigma(e) = 0$ for each arc $e \in A(\Gamma)$ that is not a loop. By applying an appropriate automorphism to $\Gamma^{\sigma}$, we may also assume without loss of generality that $l=0$. 

We induct on $m$, the number of vertices of $\Gamma$.
When $m=1$, for $j \in \{0, \dots, d-1\}$, we let $P_j$ be a path of length $0$ containing the single vertex $(v_1, jr)$. Since $d \leq n/\gcd(r,n)$, our paths $P_j$ are all distinct. Therefore,
the first statement of the lemma holds, and the other three statements hold vacuously.

Now, suppose $m\geq 2$. We construct our family of paths as follows.
By the induction hypothesis, there exists a vertex-disjoint family $\mathcal P$ of paths
$P_0, \dots, P_{d-1}$
beginning at $(v_1, 0), (v_1, r), \dots, (v_1, (d-1)r)$, respectively, 
ending at the vertex set $\{(v_{m-1}, i_{m-1}), (v_{m-1}, i_{m-1}+r), \dots, (v_{m-1}, i_{m-1}+(d-1)r)\}$ for some $i_{m-1} \in \mathbb Z_n$, and satisfying the last three conditions of the lemma after replacing $m$ with $m-1$.

Let $a = \sigma(v_{m-1})$ be the voltage of the primary loop at $v_{m-1}$. Consider a path $P \in \mathcal P$
which arrives at the fiber $v_{m-1}^{\sigma}$ at the vertex $(v_{m-1}, i_{m-1}+jr) \in V(v_{m-1}^{\sigma})$  for some $j \in \{0,\dots,d-1\}$.
If $a \neq 0$, then we perform the following steps.
We extend $P$ by adding the vertices 
\[(v_{m-1}, i_{m-1}+jr+a), (v_{m-1}, i_{m-1}+jr+2a), (v_{m-1}, i_{m-1}+jr+3a), \dots\] 
until we reach a vertex $(v_{m-1}, i_{m-1}+jr+ sa)$ such that $(v_{m-1}, i_{m-1}+jr+ (s+1)a)$ already belongs to a (not necessarily distinct) path $P' \in \mathcal P$. 
This extension is depicted in Figure \ref{fig:billiard}.
As $(v_{m-1}, i_{m-1}+jr + (s+1)a)$ is the first vertex encountered by $P$ that already belongs to a path $P' \in \mathcal P$, it follows from the way that we have extended $P$ that $(v_{m-1}, i_{m-1}+jr+(s+1)a)$ is the vertex at which $P'$ arrived at the fiber $v_{m-1}^{\sigma}$.
%; that is, $jr+(s+1)a = i_k$.
At this point, we stop adding vertices from $v_{m-1}^{\sigma}$ to $P$, with $(v_{m-1}, i_{m-1}+jr + sa)$ being the last vertex from $v_{m-1}^{\sigma}$ added.

We claim that after applying this technique at $v_{m-1}^{\sigma}$, the endpoints of the paths $P_0, \dots, P_{d-1}$ form 
 the set 
 \[S_{m-1}:=\{(v_{m-1}, i_{m-1}-a), (v_{m-1}, i_{m-1}-a+r), \dots, (v_{m-1}, i_{m-1}-a+(d-1)r)\}.\] (Note that we do not make any claims about the order in which these vertices appear as the endpoints of paths $P_0, \dots, P_{k-1}$.) 
 Indeed, if $a = 0$, then this claim clearly holds. 
 If $a\neq 0$, then we recall that
  the paths $P_0, \dots, P_{d-1}$ arrive at $v_{m-1}^{\sigma}$ at vertices of the set 
 $$\{(v_{m-1}, i_{m-1}), (v_{m-1}, i_{m-1}+r), \dots, (v_{m-1}, i_{m-1}+(d-1)r)\}.$$
 Then,
 in the process of extending our paths, each path $P \in \mathcal P$ is extended by adding vertices of the fiber $v_{m-1}^{\sigma}$ to $P$ until $P$ reaches a vertex $(v_{m-1}, i_{m-1} + tr - a)$, 
 where $t \in \{0, \dots, d-1\}$.
 Furthermore, after extending each path $P \in \mathcal P$, the endpoints of the paths $P_0, \dots, P_{d-1}$ are still distinct. Therefore, 
 it follows that after extending each path $P \in \mathcal P$, the endpoints of $P_0, \dots, P_{d-1}$ make up the set $S_{m-1}$. Thus, the claim holds.
 
 Finally, for each path $P \in \mathcal P$, we
 write $(v_{m-1}, i_{m-1} +jr - a)$ for the endpoint of $P$ in $v_{m-1}^{\sigma}$, and we add an edge $[(v_{m-1}, i_{m-1} +jr - a), (v_{m}, i_{m-1}+jr - a)]$ to extend $P$ to $v_m^\sigma$. This completes our construction of paths $P_0, \dots, P_{m-1}$. Observe that our family of paths is still vertex-disjoint.
 
 We check that the four properties of the lemma hold. The first property holds by the induction hypothesis. The second property holds for $2 \leq t \leq m-1$ by the induction hypothesis and holds for $t = m$ by the construction. The third property holds for $2 \leq t \leq m-2$, by the induction hypothesis. The statement also holds for $t = m-1$, as each component is a cycle and each path $P \in \mathcal P$ does not exit a component of $v_{m-1}^{\sigma}$ until $P \in \mathcal P$ cannot visit any more vertices in that component. The fourth statement also holds by the induction hypothesis and by construction. 
 Thus, induction is complete, and the theorem is proven.
\end{proof}
%%Figure-1%%
\begin{figure}
    \centering
    \includegraphics[scale=0.3333333]{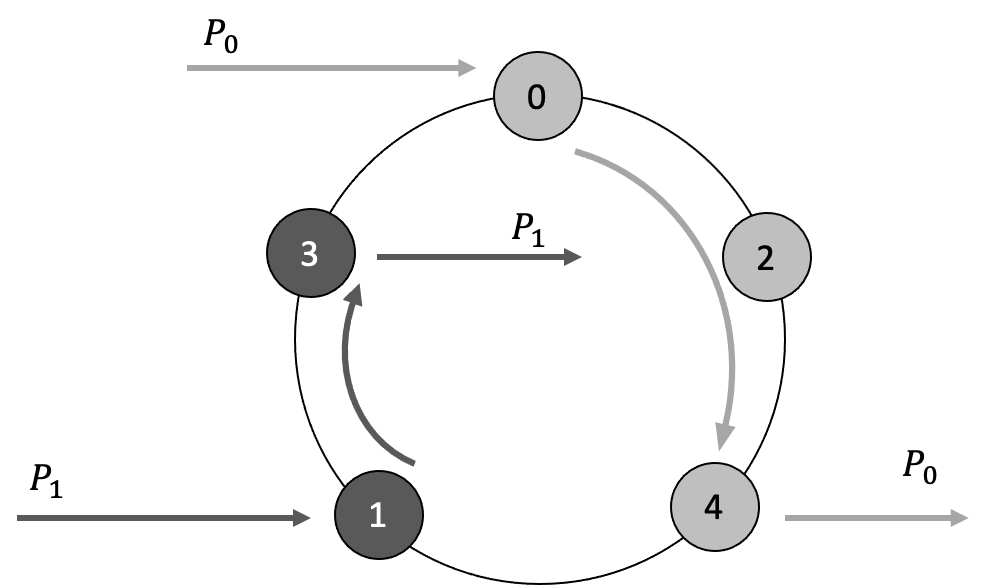}
    \caption{The figure shows an example of the path extension in Lemma \ref{lemma:billiardPaths}. The underlying group is $\mathbb{Z}_5$. The cycle is a fiber over some vertex $v_i$ with the voltage assignment $\sigma(v_i)=2$. We have depicted two paths $P_0$ and $P_1$. The path $P_0$ arrives at the fiber at 0, and $P_1$ arrives at the fiber at 1. We then extend $P_0$ to $0+2=2$ (mod $5$), and then to $2+2=4$ (mod $5$), and then stop because $4+2=1$ (mod $5$), and this vertex is already visited by $P_1$. }
    \label{fig:billiard}
\end{figure}

We use the name \emph{extended billiard strategy}
to refer to the method used in Lemma \ref{lemma:billiardPaths}
to generate our family $\mathcal P$ of paths.
 Lemma \ref{lemma:billiardPaths} tells us that given a path $\Gamma$, a voltage assignment $\sigma$, and a value $d$ as outlined in the lemma statement, 
 if we follow the extended billiard strategy as outlined to produce paths $P_0, \dots, P_{d-1}$, then for each value $2 \leq t \leq m-1$, 
 the paths arrive at the fiber $v_t^{\sigma}$ at a set of $d$ vertices $\{(v_t, i_t), (v_t, i_t+r), \dots, (v_t, i_t+(d-1)r)\}$, 
 for some value $i_t \in \mathbb{Z}_n$, where addition is calculated modulo $n$. 
 Furthermore, by following the proof of Lemma \ref{lemma:billiardPaths}, we see that this value $i_t$ is in fact $-(\sigma(v_2) + \sigma(v_3) + \dots + \sigma(v_{t-1}))$. 
 Furthermore, after applying our method at $v_t^{\sigma}$ so that the paths in $\mathcal P$ contain all vertices of $v_t^{\sigma}$, we see that the endpoints of the paths occupy the vertex set $\{(v_t, \alpha_t), (v_t, \alpha_t+r), \dots, (v_t, \alpha_t+(d-1)r)\}$, where $\alpha_t = i_t - \sigma(v_t,v_t)$. 
 Using this fact, we define the \emph{order} of the paths $P_0, \dots, P_{d-1}$
 at $v_t^{\sigma}$ as follows. 
 After applying our method at $v_t^{\sigma}$, the paths with endpoints at $(v_t, \alpha_t), (v_t, \alpha_{t}+r), \dots, (v_t, \alpha_t + (d - 1)r)$, respectively, are  $P_{a_1}, \dots, P_{a_d}$, where $(a_1, \dots, a_d)$ is some permutation of the set $\{0, \dots, d-1\}$. 
 We write $\pi(v_t) = (a_1, \dots, a_d)$, and we say that the permutation $\pi(v_t)$ gives the \emph{order} of the paths $P_0, \dots, P_{d-1}$ at $v_t$. Note that $\pi(v_t)$ depends on $t$, $d$, $n$, $\sigma$, and $r$.
It is convenient to define $\pi(v_1) = \id$. 
% We also define 
% \[ \varphi_t = \pi_t \circ \pi_{t-1}^{-1} \] 
% to be the individual permutation applied to the paths $P_0, \dots, P_{k-1}$ by the single fiber $v_t^{\sigma}$. In this way, for each $2 \leq t \leq m$, $\pi_t = \varphi_t \circ \dots \circ \varphi_2$.

Throughout this section, we use the following invariant. 

\begin{definition}
Let $\Gamma = (v_1, \dots, v_m)$ be a reflexive path with a voltage assignment $\sigma : A(\Gamma)\to \mathbb{Z}_n $. When $m \geq 3$, we define
$$\mc(\Gamma) =\left \lceil \frac{1}{2} \max\{\gcd(n,\sigma(v_t,v_t)): 2 \leq t \leq m-1\} \right \rceil.$$
When $m = 2$, we say $c(\Gamma) = 1$.
\end{definition}

For a reflexive path $\Gamma$, the quantity $\mc(\Gamma)$ is approximately half the maximum number of components that appear in the fiber of a single internal vertex of $\Gamma$.
The following corollary gives a simple condition for when a system of paths constructed in Lemma \ref{lemma:billiardPaths} includes all vertices in fibers over the internal vertices of $\Gamma$.
\begin{corollary}
\label{corInclude}
For an integer $m \geq 2$, let $\Gamma = (v_1, \dots, v_m)$ be a reflexive path, and let $\sigma:A(\Gamma) \rightarrow \mathbb{Z}_n$ be a voltage assignment. Let $\mathcal P = \{P_0, \dots, P_{d-1}\}$ be a family of paths on $\Gamma^{\sigma}$ constructed according to Lemma \ref{lemma:billiardPaths} with a value $r$ coprime to $n$.
If $ 2\mc (\Gamma) \leq d\leq n$, then the paths of $\mathcal P $ visit all vertices in each fiber $v_i^{\sigma}$, for $2 \leq i \leq m-1$.
\end{corollary}

\begin{proof}
Consider the fiber $v_t^{\sigma}$
for a value $2 \leq t \leq m-1$. 
This fiber contains a component $C_a$ for each coset $a + \langle \sigma(v_t) \rangle$ of $\mathbb{Z}_n$, and the number of such cosets is $\gcd(n, \sigma(v_t))$. 
By Lemma \ref{lemma:billiardPaths}, the paths $P_0, \dots, P_{d-1}$ arrive at $v^{\sigma}_t$ at a vertex set of the form 
$\{(v_i, t), (v_i, t+r), \dots, (v_i, t+r(d-1))\}$, so as $(r,n) = 1$, by the assumption $d \geq 2 \mc ( \Gamma ) \geq \gcd(n, \sigma(v_t))$, the vertices of the paths $P_0, \dots, P_{d-1}$ meet every 
component of $v_t^{\sigma}$.
Thus, by the third property of Lemma \ref{lemma:billiardPaths}, $V(v_t^{\sigma}) \subseteq V(P_0) \cup \dots \cup V(P_{d-1})$.
\end{proof}

Given a path $\Gamma$ satisfying the conditions of Corollary \ref{corInclude}, we would like to find conditions for when the paths $P_0, \dots, P_{d-1}$ can be joined together to form a Hamiltonian cycle. The next definition helps us achieve this goal.
Given a vertex $v$ whose loop has a nonzero voltage assigned by $\sigma$, we say that a set of edges $E \subseteq E( v^{\sigma}   )$ is in \emph{alternating consecutive order} if $E$ forms a color class of a proper $2$-coloring of the edges of some path in $v^{\sigma}$.
We sometimes call $E$ an \emph{alternating consecutive edge set}. For a graph $\Gamma$ with a voltage assignment $\sigma$, we often consider subgraphs of $\Gamma^{\sigma}$ that intersect some fiber $v^{\sigma}$ of $\Gamma^{\sigma}$ in all edges except for some alternating consecutive edge set. In other words, we may consider a subgraph $H \subseteq \Gamma^{\sigma}$ for which $E(H) \cap E(v^{\sigma}) = E(v^{\sigma}) \setminus E$, where $E$ is an alternating consecutive edge set in $v^{\sigma}$. We show an example of such a subgraph $H$ in Figure \ref{figAltCon}. 

% We will usually build a Hamiltonian cycle in $\Gamma^{\sigma}$ starting from special 2-factors which will intersect fibers of neighbours of a given vertex $v$ in $\Gamma$ in all but alternating consecutive edge sets, respectively. We then connect these 2-factors along these alternating consecutive edge sets as well as edges of the fiber $v^{\sigma}$ into a Hamiltonian cycle. The following two results guarantee such 2-factors in lifts of paths and spider graphs, two building blocks in our decomposition theorem that will be used in next section.

Next, we introduce a lemma giving a condition for the existence of a $2$-factor in the lift of a path satisfying certain properties. We apply this lemma later when finding
a Hamiltonian cycle in the lift of a path.
In this lemma and in later lemmas, we consider a path whose endpoints both have a label $r$ which is coprime to $n$ (and whose cut-edges are assumed by Lemma \ref{lemZero} all to have labels of $0$).
However, we note that given a path $\Gamma$ with such a voltage $r$ at its endpoints, we can give each vertex $v \in V(\Gamma)$ a new label $\phi(v) = \sigma(v) r^{-1} \in \mathbb Z_n$, and then $\Gamma^{\phi}$ is isomorphic to $\Gamma^{\sigma}$. Therefore, by relabelling group elements appropriately, we may assume in the proofs of this lemma and later lemmas that the endpoints of our path $\Gamma$ have a voltage of $1$.
%
%%Figure-2%%
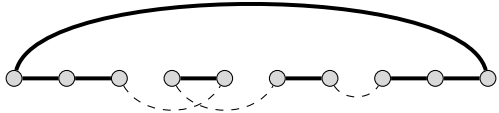
\begin{figure}

\begin{center}
\begin{tikzpicture}
    [scale=0.7,auto=left,every node/.style={circle,fill=gray!30},minimum size = 6pt,inner sep=0pt]

  \clip (7.5,1.5) rectangle + (12,-2.2);

%\node (z) at (19,0) [draw = white, fill = white] {$v^{\sigma}$};
%\node (z) at (19,-1) [draw = white, fill = white] {$u_0^{\sigma}$};
%\node (z) at (18.5,-2.5) [draw = white, fill = white] {$T_0^{\sigma}$};

%\node (a9) at (8,0) [draw = black] {};
\node (a10) at (9,0) [draw = black] {};
\node (a11) at (10,0) [draw = black] {};
\node (a12) at (11,0) [draw = black] {};
\node (a13) at (12,0) [draw = black] {};
\node (a14) at (13,0) [draw = black] {};
\node (a15) at (14,0) [draw = black] {};
\node (a16) at (15,0) [draw = black] {};
\node (a17) at (16,0) [draw = black] {};
\node (a18) at (17,0) [draw = black] {};
\node (a19) at (18,0) [draw = black] {};
%\node (a20) at (19,0) [draw = black] {};
%\node (a21) at (20,0) [draw = black] {};

\draw [line width = 1.5pt] (a10) to  [out=75,in=105,looseness = 0.5] (a19) ;
%\node (b12) at (11,0-1) [draw = black] {};
%\node (b13) at (12,0-1) [draw = black] {};
%\node (b14) at (13,0-1) [draw = black] {};
%\node (b15) at (14,0-1) [draw = black] {};
%\node (b16) at (15,0-1) %[draw = black] {};
%\node (b17) at (16,-1) [draw = black] {};
%\node (b18) at (17,-1) [draw = black] {};
%\node (b19) at (18,-1) [draw = black] {};

 %\foreach \from/\to in {a10/a11,a11/a12,a12/b12,a13/b13,a14/b14,a15/b15,a16/b16,a17/b17, a13/a14,a15/a16,a17/a18,a18/a19, a12/b12,b13/a13,b16/a16,b17/a17}
    %\draw (\from) -- (\to);
 
 %\draw [line width = 2pt, red] (b12) to   (b13) ;
 % \draw [line width = 2pt] (a12) to   (b12) ;
  %  \draw [line width = 2pt] (a13) to   (b13) ;
    
       % \draw [line width = 2pt] (a16) to   (b16) ;
        %\draw [line width = 2pt] (a17) to   (b17) ;
          \draw [line width = 1.5pt] (a11) to   (a10) ;
  \draw [line width = 1.5pt] (a11) to   (a12) ;
    \draw [line width = 1.5pt] (a13) to   (a14) ;
      \draw [line width = 1.5pt] (a15) to   (a16) ;
        \draw [line width = 1.5pt] (a17) to   (a18) ;
          \draw [line width = 1.5pt] (a18) to   (a19) ;
  % \draw [line width = 2pt, red] (a16) to   (a17) ;
     % \draw [line width = 2pt, red] (b16) to   (b17) ;
  \draw [dashed] (a12) to  [out=300,in=240] (a14) ;
  \draw [dashed] (a13) to  [out=300,in=240] (a15) ;
%   \draw [dashed] (b15) to  [out=270,in=270] (b18) ;
%      \draw [dashed] (b12) to  [out=270,in=270] (b19) ;
         \draw [dashed] (a16) to  [out=300,in=240] (a17) ;
%\draw [decorate,decoration={brace,amplitude=10pt,mirror,raise=4pt},yshift=0pt]
%(18.5,0) -- (18.5,1.3) node [white,midway,xshift=0.8cm] {};
           %\node (z) at (19.7,0.65) [draw = white, fill = white] {$v^{\sigma}$};
         
\end{tikzpicture} 
\end{center}
\caption{
The figure shows a fiber $v^{\sigma}$ in the lift $\Gamma^{\sigma}$ of a voltage graph $\Gamma$.
The edges in bold show the intersection of a subgraph $H \subseteq \Gamma^{\sigma}$ with $v^{\sigma}$, and the dashed edges represent the edges of $H$ outside of $v^{\sigma}$. 
The subgraph $H$ intersects $v^{\sigma}$ at all edges except for three edges, and these three missing edges
%(that are omitted in the figure)
form an alternating consecutive edge set in $v^{\sigma}$.}
\label{figAltCon}
\end{figure}

\begin{lemma}
\label{lemma2factor}
Let $\Gamma = (v_1, \dots, v_m)$ be a reflexive path with a voltage assignment $\sigma : A(\Gamma)\to \mathbb{Z}_n $. 
Suppose that $\sigma(v_1) = \sigma(v_m) =  r$, where $r$ is coprime with $n$.
Then, for each even integer $d$ satisfying
%$$d \geq max\{\gcd(n,\sigma(v_i,v_i)):   2\leq k\leq n-1\}$$
$$2 \mc (\Gamma) \leq d \leq n,$$
and for each integer $l$ satisfying $0\leq l\leq n-1$, there exists a $2$-factor $F$ of $\Gamma^{\sigma}$, as well as two edge subsets $E^L \subseteq E(v_1^{\sigma})$ and $E^R \subseteq E(v_m^{\sigma})$, such that
\begin{itemize}
    \item $|E^L| = |E^R| = d/2$;
    \item $E^L=\{[(v_1,l),(v_1,l+r)],[(v_1,l+2r),(v_1,l+3r)],\dots, 
    [(v_1,l+(d-2)r),(v_1,l+(d-1)r)]\}$; 
    \item $E^L$ and $E^R$ are alternating consecutive edge sets in $v_1^{\sigma}$ and $v_m^{\sigma}$, respectively;
    %\item $E(F) \cap E(u_0^{\sigma}) = E(u_0^{\sigma}) \setminus E^L$, and $E(F) \cap E(u_k^{\sigma}) = E(u_k^{\sigma}) \setminus E^R$.
    \item $F$ contains all edges of $v_1^{\sigma}$ except $E^L$ and all edges of $v_m^{\sigma}$ except $E^R$.
\end{itemize}
\end{lemma}

%Then $\Gamma^{\sigma}$ contains a 2-factor $F$
%with an edge subset $E \subseteq E(F)$ such that $E$ is in alternating consecutive order, and such that $E$ contains at least one edge from each component of $F$.
%Furthermore, the 2-consecutive set can start at any vertex over the fiber.

\begin{proof}
%By applying an appropriate automorphism to $\Gamma^{\sigma}$, we may assume that $\sigma(u_0,u_0) = \sigma(u_k,u_k) = 1$.
%{\PB This is technically incorrect; if we apply an automorphism to make this assumption, then when we undo the automorphism to get back the original labels, bullet $2$ will be messed up; that is, the $+1$ $+2$ $+3$ will turn into something else. I think it might be worth going way back and saying that we will usually assume that the voltages at path endpoints are all $1$, and then only generalize at the end of the section, maybe not until the proof of the main theorem.}
Let $d$ and $l$ be as in the statement of the lemma. By our previous discussion, we may assume that $r = 1$ by relabelling our group elements. 
We construct a family of $d$ paths $\mathcal P = \{P_0,...,P_{d-1}\}$ by the process of Lemma \ref{lemma:billiardPaths} (using our values $l$ and $r = 1$), each with one endpoint in the fiber $v_1^{\sigma}$ and other endpoint in the fiber $v_m^{\sigma}$. 
By Lemma \ref{lemma:billiardPaths}, we may assume that the endpoints of each path $P_i$ are $(v_1, l+i)$ and $(v_m,l+i + g)$, for a single value $g \in \mathbb{Z}_n$. Furthermore, by Corollary \ref{corInclude}, we may assume that the paths of $\mathcal P$ contain all vertices of the fibers $v_2^{\sigma}, v_3^{\sigma}, \dots,v_{m-1}^{\sigma}$. 
We construct the $2$-factor $F$ from the union $P_0 \cup P_1 \cup \dots \cup P_{d-1}$ by adding the following edges: 
\begin{itemize}
    \item all edges in the unique perfect matching on the vertex set $\{(v_1,l+i): 1 \leq i \leq d - 2\}$ in $v_1^{\sigma}$;
    \item all edges in the unique perfect matching on the vertex set $\{(v_m,l+i +g): 1 \leq i \leq d - 2\}$ in $v_m^{\sigma}$;
    \item all edges of the path $(v_1, l+d-1), (v_1, l+d), (v_1, l+d+1), \dots, (v_1,l-1), (v_1, l)$;
     \item all edges of the path $(v_m, l+d+g-1), (v_m, l+d+g), (v_m, l+d+g+1), \dots, (v_m,l+g-1), (v_m,l+g)$.
\end{itemize}

It is straightforward to check that $F$ is a $2$-factor of $\Gamma^{\sigma}$. We let $E^L$ consist of all edges in the unique perfect matching on $\{(v_1,l+i): 0 \leq i \leq d - 1\}$ in $v_1^{\sigma}$. Similarly, we let the set $E^R$ to consist of all edges in the unique perfect matching on $\{(v_m,l+i+g): 0 \leq i \leq d - 1\}$ in $v_m^{\sigma}$. By construction, all four properties of the lemma are satisfied for sets $E^L$ and $E^R$.
\end{proof}

For an example of a $2$-factor $F$ described in Lemma \ref{lemma2factor}, see Figure \ref{fig:example_factor1}.

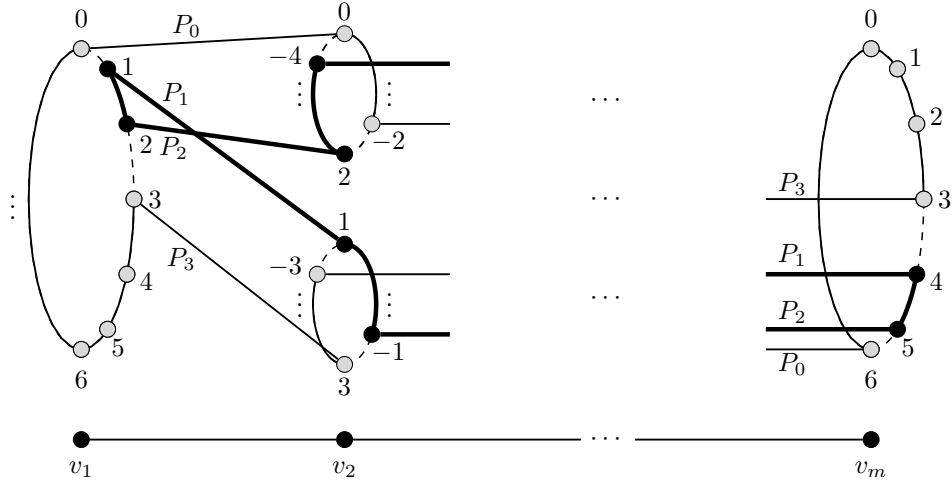
\begin{figure}
    \centering
    \begin{tikzpicture}
[scale=2,xscale = 0.35, yscale = 1]

     \node (z) at (2,1.15) [ draw = white, fill = white]  {$P_0$}; 
     \node (z) at (1.8,0.7) [ draw = white, fill = white]  {$P_1$}; 
     \node (z) at (1.75,0.34) [ draw = white, fill = white]  {$P_2$}; 
     \node (z) at (1.9,-0.4) [ draw = white, fill = white]  {$P_3$};

     \node (z) at (13.5,-1.1) [ draw = white, fill = white]  {$P_0$}; 
     \node (z) at (13.5,-0.375) [ draw = white, fill = white]  {$P_1$}; 
     \node (z) at (13.5,-0.75) [ draw = white, fill = white]  {$P_2$}; 
     \node (z) at (13.5,0.1) [ draw = white, fill = white]  {$P_3$};

     \node (z) at (-1.3,0) [ draw = white, fill = white]  {$\vdots$}; \node (z) at (5.85,0.75) [ draw = white, fill = white]  {$\vdots$}; 
     \node (z) at (4.15,0.75) [ draw = white, fill = white]  {$\vdots$}; 
     \node (z) at (5.85,-0.65) [ draw = white, fill = white]  {$\vdots$}; 
     \node (z) at (4.15,-0.65) [ draw = white, fill = white]  {$\vdots$}; 

 \draw [line width = 0.25mm,domain={90}:{360}] plot ({1*cos(\x)}, {1*sin(\x)});

 \draw [line width = 0.6mm,domain={30}:{60}] plot ({1*cos(\x)}, {1*sin(\x)});

 \node (z) at (0,1.2) [ draw =white, circle,fill=white,minimum size = 0pt,inner sep=0pt]  {$0$};
 
  \node (z) at (0,-1.2) [ draw =white, circle,fill=white,minimum size = 0pt,inner sep=0pt]  {$6$};

 \node (z) at ({1.2*cos(55)+0.2},{1.2*sin(55)-0.1}) [ draw = white, circle,fill=white,minimum size = 0pt,inner sep=0pt]  {$1$};
 
  \node (z) at ({1.2*cos(20)+0.1},{1.2*sin(20)-0.05}) [ draw = white, circle,fill=white,minimum size = 0pt,inner sep=0pt]  {$2$};
  
    \node (z) at ({1.2*cos(0)+0.2},{1.2*sin(0)-0.00}) [ draw = white, circle,fill=white,minimum size = 0pt,inner sep=0pt]  {$3$};

    \node (z) at ({1.2*cos(-30)+0.2},{1.2*sin(-30)+0.05}) [ draw = white, circle,fill=white,minimum size = 0pt,inner sep=0pt]  {$4$};

    \node (z) at ({1.2*cos(-60)+0.1},{1.2*sin(-60)+0.05}) [ draw = white, circle,fill=white,minimum size = 0pt,inner sep=0pt]  {$5$};

%%Here be the labels for the rightmost ellipse
 \node (z) at (15,1.2) [ draw =white, circle,fill=white,minimum size = 0pt,inner sep=0pt]  {$0$};
 
  \node (z) at (15,-1.2) [ draw =white, circle,fill=white,minimum size = 0pt,inner sep=0pt]  {$6$};

 \node (z) at ({15+1.2*cos(60)+0.25},{1.2*sin(60)-0.1}) [ draw = white, circle,fill=white,minimum size = 0pt,inner sep=0pt]  {$1$};
 
  \node (z) at ({15+1.2*cos(30)+0.2},{1.2*sin(30)-0.05}) [ draw = white, circle,fill=white,minimum size = 0pt,inner sep=0pt]  {$2$};
  
    \node (z) at ({15+1.2*cos(0)+0.2},{1.2*sin(0)-0.00}) [ draw = white, circle,fill=white,minimum size = 0pt,inner sep=0pt]  {$3$};

    \node (z) at ({15+1.2*cos(-30)+0.2},{1.2*sin(-30)+0.05}) [ draw = white, circle,fill=white,minimum size = 0pt,inner sep=0pt]  {$4$};

    \node (z) at ({15+1.2*cos(-60)+0.1},{1.2*sin(-60)+0.05}) [ draw = white, circle,fill=white,minimum size = 0pt,inner sep=0pt]  {$5$};

\foreach \shift in {5}
{

     \draw [line width = 0.25mm,domain={-30}:{90}] plot ({5+0.6*cos(\x)}, {0.7+0.4*sin(\x)});

    \draw [line width = 0.6mm,domain={150}:{270}] plot ({5+0.6*cos(\x)}, {0.7+0.4*sin(\x)});

     \draw [line width = 0.6mm,domain={-30}:{90}] plot ({5+0.6*cos(\x)}, {-0.7+0.4*sin(\x)});

    \draw [line width = 0.25mm,domain={150}:{270}] plot ({5+0.6*cos(\x)}, {-0.7+0.4*sin(\x)});

        \draw [line width = 0.25mm,domain={0}:{270}] plot ({15+cos(\x)}, {sin(\x)});
        
   \draw [line width = 0.6mm,domain={-60}:{-30}] plot ({15+cos(\x)}, {sin(\x)});

    \draw[line width = 0.25mm, dashed] (\shift,0.7) ellipse (0.6 and 0.4);

    \draw[line width = 0.2mm, dashed] (\shift,-0.7) ellipse (0.6 and 0.4);

    \node (a1) at (5,1.1) [ draw = black, circle,fill=gray!30,minimum size = 6pt,inner sep=0pt]  {};
    
    \node (b1) at (5,1.2) [ draw =white, circle,fill=white,minimum size = 0pt,inner sep=0pt]  {};
        \node (z) at (5,1.25) [ draw =white, circle,fill=white,minimum size = 0pt,inner sep=0pt]  {$0$};

    \node (a2) at (5,0.3) [ draw = black, circle,fill=black,minimum size = 6pt,inner sep=0pt]  {};
      
    \node (b2) at (5,0.2) [ draw = black, circle,fill=black,minimum size = 0pt,inner sep=0pt]  {};
    \node (z) at (5,0.15) [ draw = white, circle,fill=white,minimum size = 0pt,inner sep=0pt]  {$2$};
    
    \node (a4) at (5,-1.1) [ draw = black, circle,fill=gray!30,minimum size = 6pt,inner sep=0pt]  {};
    \node (b4) at (5,-1.2) [ draw =white, circle,fill=white,minimum size = 0pt,inner sep=0pt]  {};
    \node (z) at (5,-1.25) [ draw =white, circle,fill=white,minimum size = 0pt,inner sep=0pt]  {$3$};

    \node (a3) at (5,-0.3) [ draw = black, circle,fill=black,minimum size = 6pt,inner sep=0pt]  {};
     \node (b3) at (5,-0.2) [ draw =white, circle,fill=white,minimum size = 0pt,inner sep=0pt]  {};
     \node (z) at (5,-0.15) [ draw =white, circle,fill=white,minimum size = 0pt,inner sep=0pt]  {$1$};

    \node (z) at ({5 + 0.6*1.6*cos(-30)},{0.7 + 0.4*1.6*sin(-30)}) [ draw = white, circle,fill=white,minimum size = 0pt,inner sep=0pt]  {$-2$};
    
    \node (a5) at ({5 + 0.6*cos(-30)},{0.7 + 0.4*sin(-30)}) [ draw = black, circle,fill=gray!30,minimum size = 6pt,inner sep=0pt]  {};
    \node (b5) at ({5 + 0.6*1.4*cos(-30)},{0.7 + 0.4*1.4*sin(-30)}) [ draw = black, circle,fill=black,minimum size = 0pt,inner sep=0pt]  {};
    \node (c5) at (9,{0.7 + 0.4*sin(-30)}) [ draw = black, circle,fill=black,minimum size = 0pt,inner sep=0pt]  {};

    \node (z) at ({-0.35+5 + 0.6*1.5*cos(150)},{-0.05+ 0.7 + 0.4*1.5*sin(150)}) [ draw = white, circle,fill=white,minimum size = 0pt,inner sep=0pt]  {$-4$};
        \node (a6) at ({5 + 0.6*cos(150)},{0.7 + 0.4*sin(150)}) [ draw = black, circle,fill=black,minimum size = 6pt,inner sep=0pt]  {};
    
    \node (b6) at ({5 + 0.6*1.4*cos(150)},{0.7 + 0.4*1.4*sin(150)}) [ draw = black, circle,fill=black,minimum size = 0pt,inner sep=0pt]  {};
    
     \node (c6) at (9,{0.7 + 0.4*sin(150)}) [ draw = black, circle,fill=black,minimum size = 0pt,inner sep=0pt]  {};

    \node (z) at (2*\shift,0) [ draw = white, fill = white]  {$\cdots$}; 
    \node (z) at (2*\shift,0.66) [ draw = white, fill = white]  {$\cdots$};
     \node (z) at (2*\shift,-0.66) [ draw = white, fill = white]  {$\cdots$};
     
     \node (z) at ({5 + 0.6*1.6*cos(-30)},{-0.7 + 0.4*1.6*sin(-30)}) [ draw = white, circle,fill=white,minimum size = 0pt,inner sep=0pt]  {$-1$};
     
     \node (a7) at ({5 + 0.6*cos(-30)},{-0.7 + 0.4*sin(-30)}) [ draw = black, circle,fill=black,minimum size = 6pt,inner sep=0pt]  {};
     \node (b7) at ({5 + 0.6*1.4*cos(-30)},{-0.7 + 0.4*1.4*sin(-30)}) [ draw = black, circle,fill=black,minimum size = 0pt,inner sep=0pt]  {};
     \node (c7) at (9,{-0.7 + 0.4*sin(-30)}) [ draw = black, circle,fill=black,minimum size = 0pt,inner sep=0pt]  {};

     \node (z) at ({-0.35+5 + 0.6*1.5*cos(150)},{-0.05+ -0.7 + 0.4*1.5*sin(150)}) [ draw = white, circle,fill=white,minimum size = 0pt,inner sep=0pt]  {$-3$};
    \node (a8) at ({5 + 0.6*cos(150)},{-0.7 + 0.4*sin(150)}) [ draw = black, circle,fill=gray!30,minimum size = 6pt,inner sep=0pt]  {};
        \node (b8) at ({5 + 0.6*1.4*cos(150)},{-0.7 + 0.4*1.4*sin(150)}) [ draw = black, circle,fill=black,minimum size = 0pt,inner sep=0pt]  {};
     \node (c8) at (9,{-0.7 + 0.4*sin(150)}) [ draw = black, circle,fill=black,minimum size = 0pt,inner sep=0pt]  {};

\foreach \s in {0,3}
{

%\node (z) at (\s*\shift,-2)  [ draw = black, fill = black, circle,minimum size = 6pt,inner sep=0pt] {};

%\node (z) at (\s*\shift,-1.7) [ draw = white, fill = white]  {$v_{\s}$};

\draw[line width = 0.2mm, dashed] (\shift*\s,0) circle [thick, radius=1];

\foreach \i in {0,1,2,3,9,10,11}
    {

		\node (z) at ({\shift*\s + cos(30*\i)},{sin(30*\i)}) [ draw = black, circle,fill=black,minimum size = 6pt,inner sep=0pt]  {};

    }
    %\node (z) at ({\shift*\s +cos(90-30*\s)},{sin(90-30*\s)})  [ draw = black, circle,fill=gray!40!white,minimum size = 6pt,inner sep=0pt]  {};
    %\node (z) at ({\shift*\s +cos(60-30*\s)},{sin(60-30*\s)})  [ draw = black, circle,fill=gray!40!white,minimum size = 6pt,inner sep=0pt]  {};
    %\node (z) at ({\shift*\s +cos(30-30*\s)},{sin(30-30*\s)})  [ draw = black, circle,fill=gray!40!white,minimum size = 6pt,inner sep=0pt]  {};

}

}
		 \draw [line width = 0.25mm]   ( {cos(30*3)},{sin(30*3)}) --  (a1) ;
		 \draw [line width = 0.6mm]   ( {cos(30*2)},{sin(30*2)}) --  (a3) ;
		  \draw [line width = 0.6mm]   ( {cos(30)},{sin(30)}) --  (a2) ;		  \draw [line width = 0.25mm]   ( {cos(0)},{sin(0)}) to %[out=0,in=135] 
		  (a4) ;
		  
		   %\draw [line width = 0.2mm]   ( a2) --  (b2) ;
		    %\draw [line width = 0.6mm]   ( a1) --  (b1) ;	
            %\draw [line width = 0.6mm]   ( a3) --  (b3) ;
		   % \draw [line width = 0.2mm]   ( a4) --  (b4) ;	
		     %\draw [line width = 0.6mm]   ( a5) --  (b5) ;	
           % \draw [line width = 0.2mm]   ( a6) --  (b6) ;	
             % \draw [line width = 0.6mm]   ( a7) --  (b7) ;	
               % \draw [line width = 0.2mm]   ( a8) --  (b8) ;	
                
               % \draw [line width = 0.6mm]   ( a5) --  (c5) ;	
                % \draw [line width = 0.2mm]   ( a6) --  (c6) ;	
               % \draw [line width = 0.6mm]   ( a7) --  (c7) ;	
               % \draw [line width = 0.2mm]   ( a8) --  (c8) ;	

                %\draw [line width = 0.6mm]   ( b1) to [out = 0, in = 60,looseness = 1.1]  (b5) ;	
                  %\draw [line width = 0.2mm]   ( b2) to [out = 180, in = -120,looseness = 1.1]  (b6) ;
                  
                  %\draw [line width = 0.6mm]   ( b3) to [out = 0, in = 60,looseness = 1.1]  (b7) ;	
                 % \draw [line width = 0.2mm]   ( b4) to [out = 180, in = -120,looseness = 1.1]  (b8) ;

                \draw [line width = 0.6mm]   (a6) --  ( 7,{0.7 + 0.4*sin(150)}) ;
                \draw [line width = 0.25mm]   (a5) --  ( 7,{0.7 + 0.4*sin(-30)}) ;
                
                \draw [line width = 0.6mm]   (a7) --  ( 7,{-0.7 + 0.4*sin(-30)}) ;
                \draw [line width = 0.25mm]   (a8) --  ( 7,{-0.7 + 0.4*sin(150)}) ;

                \draw [line width = 0.25mm]   ( {15+cos(30*3)},{-sin(30*3)}) --  ( 13,{-sin(30*3)}) ;
                \draw [line width = 0.6mm]   ( {15+cos(30*2)},{-sin(30*2)}) --  ( 13,{-sin(30*2)}) ;
                \draw [line width = 0.6mm]   ( {15+cos(30*1)},{-sin(30*1)}) --  ( 13,{-sin(30*1)}) ;
                \draw [line width = 0.25mm]   ( {15+cos(30*0)},{-sin(30*0)}) --  ( 13,{-sin(30*0)}) ;
    
    \foreach \zShift in {0.4}
    {       
    \node (p1) at (0,-2+\zShift) [ draw = black, circle,fill=black,minimum size = 6pt,inner sep=0pt]  {};       \node (p2) at (5,-2+\zShift) [ draw = black, circle,fill=black,minimum size = 6pt,inner sep=0pt]  {};    
    \node (p3) at (15,-2+\zShift) [ draw = black, circle,fill=black,minimum size = 6pt,inner sep=0pt]  {};

    \draw [line width = 0.25mm]   (p1) --  (p2) ;
    \draw [line width = 0.25mm]   (p3) --  (p2) ;
      \node (z) at (10,-2+\zShift) [ draw = white, fill = white]  {$\cdots$}; 
      
        \node (z) at (0,-2.2+\zShift) [ draw = white, fill = white]  {$v_1$}; 
           \node (z) at (5,-2.2+\zShift) [ draw = white, fill = white]  {$v_2$}; 
           
              \node (z) at (15,-2.2+\zShift) [ draw = white, fill = white]  {$v_m$}; 
         }    
          \node (z) at (0,1) [ draw = black, fill = gray!30,circle,minimum size = 6pt,inner sep=0pt]  {}; 
          
          \node (z) at (0,-1) [ draw = black, fill = gray!30,circle,minimum size = 6pt,inner sep=0pt]  {}; 
          
          \node (z) at (1,0) [ draw = black, fill = gray!30,circle,minimum size = 6pt,inner sep=0pt]  {}; 
          
          \node (z) at ({cos(30)},{-sin(30)}) [ draw = black, fill = gray!30,circle,minimum size = 6pt,inner sep=0pt]  {}; 
       
          \node (z) at ({cos(60)},{-sin(60)}) [ draw = black, fill = gray!30,circle,minimum size = 6pt,inner sep=0pt]  {}; 
              
               \node (z) at (15,1) [ draw = black, fill = gray!30,circle,minimum size = 6pt,inner sep=0pt]  {}; 
          
          \node (z) at (15,-1) [ draw = black, fill = gray!30,circle,minimum size = 6pt,inner sep=0pt]  {}; 
          
          \node (z) at (16,0) [ draw = black, fill = gray!30,circle,minimum size = 6pt,inner sep=0pt]  {}; 
          
          \node (z) at ({15+cos(30)},{sin(30)}) [ draw = black, fill = gray!30,circle,minimum size = 6pt,inner sep=0pt]  {}; 
       
          \node (z) at ({15+cos(60)},{sin(60)}) [ draw = black, fill = gray!30,circle,minimum size = 6pt,inner sep=0pt]  {};

\end{tikzpicture}
    
    \caption{An example of construction of the $2$-factor $F$ described in Lemma~\ref{lemma2factor}
    constructed using a family $\mathcal P = \{P_0,P_1,P_2,P_3\}$ of paths. 
    The underlying group is $\mathbb{Z}_n$, and voltage of each vertex from $\mathbb Z_n$ is shown. Here, we suppose that $d=4\ge 2\mc (\Gamma)$ and $g=3$.
    The two different shades of vertices show the two components of the $2$-factor $F$. 
    Moreover, the component of $F$ containing $P_1$ and $P_2$ is highlighted in bold. The set $E^L$ is depicted by the two dashed edges in $v_1^{\sigma}$, and the set $E^R$ is depicted by the two dashed edges in $v_m^{\sigma}$. 
    }
    \label{fig:example_factor1}
\end{figure}

Next, we need to define some notation and terminology.
%\begin{definition} \label{def:orderpres}
Let $\Gamma = (v_1, \dots, v_m)$ be a reflexive path, and let $\sigma:A(\Gamma) \rightarrow \mathbb{Z}_n$ be a voltage assignment satisfying $\sigma(v_1,v_1) = \sigma(v_m, v_m)$. Let $\mathcal P = \{P_0,P_1,...,P_{d-1}\}$ 
be a family of $d = 2\mc(\Gamma)$ paths in $\Gamma^{\sigma}$ as defined in Lemma \ref{lemma:billiardPaths}.
For $2 \leq t \leq m-1$,
recall that $\pi(v_t)$ denotes the permutation on $\{0,\dots,d-1\}$ describing the order of the paths $P_0, \dots, P_{d-1}$ as they leave $v_t^{\sigma}$.
By our definition, the permutation $\pi(v_{m-1})$ describes the order of the paths in $\mathcal P$ as they enter $v_m^{\sigma}$.
%Recall that by the method of Lemma \ref{lemma:billiardPaths}, our paths in $\mathcal P$ will 
%arrive at $v_m^{\sigma}$ at a vertex set $\{(v_m, \alpha), (v_m, \alpha+r), \dots, (v_m, \alpha + (d - 1)r)\}$ for some element $m \in G$. If we say that the path $P_{a_i}$ contains the vertex $(v_m, \alpha + ir)$, then $(a_1, \dots, a_d)$ is a permutation of $\{0, \dots, d-1\}$; recall that we write $\pi(v_{m-1})$
%for this permutation and refer to it as the order of the paths in $\mathcal P$ as they exit $v_{m-1}^{\sigma}$.
 We say that $(\Gamma, \sigma)$ is an \emph{odd shifting path}
 %to describe a voltage path $(\Gamma, \sigma)$
 %such that 
 if 
 %whenever
 %Lemma \ref{lemma:billiardPaths} is
 %applied with $d = 2\mc(\Gamma)$ to construct a
 %path family $\mathcal P = \{P_0, \dots, P_{d-1}\}$ along the lift $\Gamma^{\sigma}$ of an odd shifting path $\Gamma$, 
 this order $\pi(v_{m-1})$ of the paths in $\mathcal P$ is of the form $(d-s, d-s+1, \dots, d-1, 0, 1, \dots, d-s-1)$ for some odd number $0 < s < d$. In other words, if $(\Gamma,\sigma)$ is an odd shifting path, then an odd number of the $d$ paths in $\mathcal P$
 that were initially at the beginning of the order $\pi(v_1)$ are shifted to the end of the order $\pi(v_{m-1})$.
 If $(\Gamma, \sigma)$ is an odd shifting path, then we say that the value $r = \sigma(v_1, v_1) = \sigma(v_m, v_m)$ is the \emph{endpoint voltage} of $(\Gamma,\sigma)$.
 For an example of the lift of 
 an odd shifting path, see Figure \ref{fig:odd_shifting_path}.

 We give a sufficient condition that allows for the construction of odd shifting paths. We need some more notation.
Again, let $\Gamma = (v_1, \dots, v_m)$ be a reflexive path, and let $\sigma:A(\Gamma) \rightarrow \mathbb{Z}_n$ be a voltage assignment. Let $\mathcal P = \{P_0,P_1,...,P_{d-1}\}$ 
be a family of $d = 2\mc(\Gamma)$ paths in $\Gamma^{\sigma}$ as defined in Lemma \ref{lemma:billiardPaths}.
For each value $2 \leq t \leq m-1$, we write $\varphi(v_t) = \pi(v_t) \circ \pi(v_{t-1})^{-1}$, so that $\varphi(v_t)$ describes the way that the paths $P_0,\dots,P_{d-1}$ are permuted by the extended billiard strategy at $v_t$. 
Note that since $d = 2\mc(\Gamma)$ is fixed, $\varphi(v_t)$ depends only on $n$ and $\sigma(v_t, v_t)$.
%that is, $\pi(v_i)=\varphi(v_i)\circ \pi(v_{i-1})$. 
Now, consider a pair $v_t, v_{t+1}$ of consecutive internal vertices on $\Gamma$.
If  $\varphi(v_{t+1})\circ\varphi(v_{t})=\id$, then we say that $v_i,v_{t+1}$ is an \emph{order preserving pair}. A special case of an order preserving pair is an \emph{inverse pair}, where $\sigma(v_t,v_t) = -\sigma(v_{t+1},v_{t+1})$ in $\mathbb{Z}_n$.

%Notice that when $d = 2$, being order preserving is not a strong condition. In fact if we identify permutations that are in opposite orders {\PB is it clear why we would want to do this?} , then every pair is order preserving for $d = 2$.  {\PB I think the way to do this is either to define a secondary order that's the reverse of the order or not to mention it at all.}

Now, let $v_t,v_{t+1}$ be an order preserving pair on the path $\Gamma$, and let $v_{t-1}$ and $v_{t+2}$ be the two other neighbors of $v_i$ and $v_{t+1}$ on $\Gamma$, respectively. Consider the operation by which we remove $v_t$ and $v_{t+1}$, along with all three edges and two loops incident to $v_t$ and $v_{t+1}$, and then join $v_{t-1}$ and $v_{t+2}$ by a new edge with voltage $0$. We call this operation \emph{smoothing out the pair} $v_t,v_{t+1}$.
%\end{definition}
%\begin{definition} 
%For a path $\Gamma$ defined as in Definition \ref{def:orderpres}
If $\Gamma$ has an even number of vertices $m \geq 2$, and if we can recursively smooth out order preserving pairs on $\Gamma$ until only the two endpoints are left, then we say the path $\Gamma$ is an \emph{order preserving path}. 
%\end{definition}
%\begin{definition}
%For a path $\Gamma$, defined as in Definition \ref{def:orderpres}, whose 
It is straightforward to show (see Figure \ref{fig:odd_shifting_path}) that if we form a single path $\Gamma$ by joining an even number of order preserving paths $\Gamma_1, \dots, \Gamma_{2k}$ all with a common voltage $r$ coprime to $n$ at their endpoints,
then $\Gamma$ is an odd shifting path.

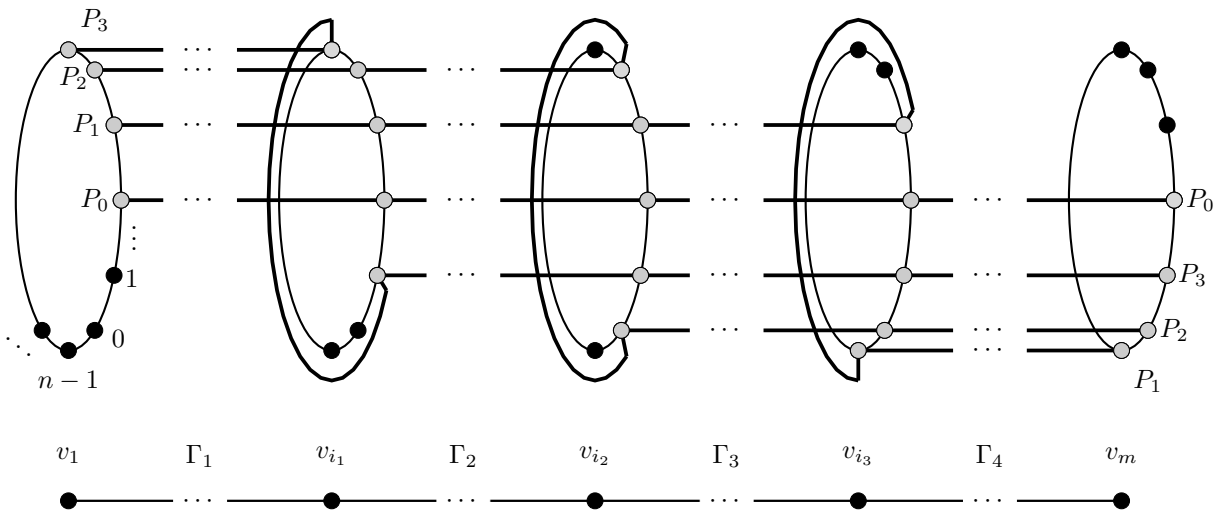
\begin{figure}

\begin{center}
\begin{tikzpicture}
[scale=2,xscale = 0.35, yscale = 1]

\node (z) at (0.5,1.2) [ draw = white, fill = white]  {$P_3$};
\node (z) at (0.1,0.816) [ draw = white, fill = white]  {$P_2$};
\node (z) at (0.35,0.5) [ draw = white, fill = white]  {$P_1$};
\node (z) at (0.5,0) [ draw = white, fill = white]  {$P_0$};

\node (z) at (20.5,-1.2) [ draw = white, fill = white]  {$P_1$};
\node (z) at (21,-0.866) [ draw = white, fill = white]  {$P_2$};
\node (z) at (21.35,-0.5) [ draw = white, fill = white]  {$P_3$};
\node (z) at (21.5,0) [ draw = white, fill = white]  {$P_0$};

\node (z) at (0,-1.2) [ draw = white, fill = white]  {$n-1$};
\node (z) at (0.95,-0.92) [ draw = white, fill = white]  {$0$};
\node (z) at (1.2,-0.52) [ draw = white, fill = white]  {$1$};
\node (z) at (1.25,-0.2) [ draw = white, fill = white]  {$\vdots$};
\node (z) at (-0.95,-0.92) [ draw = white, fill = white]  {$\ddots$};

\draw [line width=0.3mm] (0,-2)  -- (20,-2)   ;

\foreach \shift in {5}
{
\node (z) at (0*\shift,-1.7) [ draw = white, fill = white]  {$v_{1}$};
\node (z) at (4*\shift,-1.7) [ draw = white, fill = white]  {$v_{m}$};

\foreach \s in {0,1,2,3,4}
{

\node (z) at (\s*\shift,-2)  [ draw = black, fill = black, circle,minimum size = 6pt,inner sep=0pt] {};

\ifthenelse{\NOT 4 = \s}{
\node (z) at ({(\s+0.5)*\shift},-2) [ draw = white, fill = white]  {$\cdots$};

\draw [line width=0.5mm] ({\shift*\s +cos(90-30*\s)},{sin(90-30*\s)})  -- ({\shift*\s +cos(0)+0.8},{sin(90-30*\s)}) node[ sloped, scale = 1] {} ;
\draw [line width=0.5mm] ({\shift*\s +cos(60-30*\s)},{sin(60-30*\s)})  -- ({\shift*\s +cos(0)+0.8},{sin(60-30*\s)}) node[ sloped, scale = 1] {} ;
\draw [line width=0.5mm] ({\shift*\s +cos(30-30*\s)},{sin(30-30*\s)})  -- ({\shift*\s +cos(0)+0.8},{sin(30-30*\s)}) node[ sloped, scale = 1] {} ;
\draw [ line width=0.5mm] ({\shift*\s +cos(0-30*\s)},{sin(0-30*\s)})  --  ({\shift*\s +cos(0)+0.8},{sin(0-30*\s)}) node[ sloped, scale = 1] {} ;

\node (z) at ({\shift*\s +cos(0)+1.5},{sin(90-30*\s)}) [draw = white, fill = white]  {$\cdots$};
\node (z) at ({\shift*\s +cos(0)+1.5},{sin(60-30*\s)}) [draw = white, fill = white]  {$\cdots$};
\node (z) at ({\shift*\s +cos(0)+1.5},{sin(30-30*\s)}) [draw = white, fill = white]  {$\cdots$};
\node (z) at ({\shift*\s +cos(0)+1.5},{sin(0-30*\s)}) [draw = white, fill = white]  {$\cdots$};

\draw [line width=0.5mm] ({\shift*(\s+1) +cos(90-30*\s)},{sin(90-30*\s)})  -- ({\shift*\s +cos(0)+2.2},{sin(90-30*\s)}) node[ sloped, scale = 1] {} ;
\draw [line width=0.5mm] ({\shift*(\s+1) +cos(60-30*\s)},{sin(60-30*\s)})  -- ({\shift*\s +cos(0)+2.2},{sin(60-30*\s)}) node[ sloped, scale = 1] {} ;
\draw [line width=0.5mm] ({\shift*(\s+1) +cos(30-30*\s)},{sin(30-30*\s)})  -- ({\shift*\s +cos(0)+2.2},{sin(30-30*\s)}) node[ sloped, scale = 1] {} ;
\draw [ line width=0.5mm] ({\shift*(\s+1) +cos(0-30*\s)},{sin(0-30*\s)})  --  ({\shift*\s +cos(0)+2.2},{sin(0-30*\s)}) node[ sloped, scale = 1] {} ;

}{}

\ifthenelse{\NOT 4 = \s \AND \NOT 0 = \s}{

    \node (z) at (\s*\shift,-1.7) [ draw = white, fill = white]  {$v_{i_{\s}}$};

 \draw [line width = 0.5mm,domain={120-30*\s}:{360-30*(\s)}] plot ({\shift * \s + 1.2*cos(\x)}, {1.2*sin(\x)});

 \draw [line width = 0.5mm]   ({\shift * \s + 1.2*cos(120-30*\s)}, {1.2*sin(120-30*\s)}) --  ({\shift * \s + cos(120-30*\s)}, {sin(120-30*\s)}) ;

 \draw [line width = 0.5mm]   ({\shift * \s + 1.2*cos(360-30*(\s))}, {1.2*sin(360-30*(\s))}) --  ({\shift * \s + cos(360-30*(\s))}, {sin(360-30*(\s))}) ;

	}{};

\ifthenelse{\NOT 0 = \s}{

\node (z) at ({(\s-0.5)*\shift},-1.7) [ draw = white, fill = white]  {$\Gamma_{\s}$};

	}{};

\draw[thick] (\shift*\s,0) circle [thick, radius=1];
\foreach \i in {0,1,2,3,9,10,11}
    {
		%\ifthenelse{0 = \i \OR \i = 1 \OR \i = 2}{
		%\node (z) at ({\shift*\s + 0.7*cos(30*\i)},{0.82*sin(30*\i)}) [ draw = white, fill = white, minimum size = 6pt,inner sep=0pt]  {$\i$};
		%}{};

		\node (z) at ({\shift*\s + cos(30*\i)},{sin(30*\i)}) [ draw = black, circle,fill=black,minimum size = 6pt,inner sep=0pt]  {};

    }
    \node (z) at ({\shift*\s +cos(90-30*\s)},{sin(90-30*\s)})  [ draw = black, circle,fill=gray!40!white,minimum size = 6pt,inner sep=0pt]  {};
    \node (z) at ({\shift*\s +cos(60-30*\s)},{sin(60-30*\s)})  [ draw = black, circle,fill=gray!40!white,minimum size = 6pt,inner sep=0pt]  {};
    \node (z) at ({\shift*\s +cos(30-30*\s)},{sin(30-30*\s)})  [ draw = black, circle,fill=gray!40!white,minimum size = 6pt,inner sep=0pt]  {};
        	
    \ifthenelse{\NOT 4=\s}{
        \node (z) at ({\shift*\s +cos(0-30*\s)},{sin(0-30*\s)})  [ draw = black, circle,fill=gray!40!white,minimum size = 6pt,inner sep=0pt] {};
        	
        \node (z) at ({\shift*(\s+1) +cos(90-30*\s)},{sin(90-30*\s)})  [ draw =black, circle,fill=gray!40!white,minimum size = 6pt,inner sep=0pt] {};
        %This line should create the first gray nodes on circle 2,3,4,5 but the gray node does not show. Don't know why...
    }{};

\ifthenelse{\NOT 0 = \s}
{

        \node (z) at ({\shift*\s + cos(-30*\s+120)},{sin(-30*\s+120)}) [ draw = black, circle,fill=gray!30!white,minimum size = 6pt,inner sep=0pt]  {};
}
{};
}
}
		\node (z) at ({ cos(30*8)},{sin(30*8)}) [ draw = black, circle,fill=black,minimum size = 6pt,inner sep=0pt]  {};

\end{tikzpicture}
\caption{ The figure shows an example of an odd shifting path consisting of four order-preserving paths $\Gamma_1, \Gamma_2,\Gamma_3, \Gamma_4$ with endpoints 
$v_0, v_{i_1}, v_{i_2}, v_{i_3}, v_m$. 
In the figure, we show four paths $P_0, P_1, P_2, P_3$ in the lift that are constructed by the process of Lemma \ref{lemma:billiardPaths}. Since each path $\Gamma_i$ is order preserving, the final order of the paths $P_0, P_1,P_2, P_3$ arriving at $v_{i_4}^{\sigma}$ is ultimately only affected by three single shifts that occur in fibres $v_{i_1}^{\sigma}$, $v_{i_2}^{\sigma}$, and $v_{i_3}^{\sigma}$, and therefore the final order of these paths in $v_{m}^{\sigma}$ is $P_1,P_2,P_3,P_0$. Also note that while $\sigma(v_1) = \sigma(v_2) = \sigma(v_3) = 1$ in for simplicity, the vertices $v_{i_1}, v_{i_2}, v_{i_3}$ can have any common voltage $r$ coprime to $n$ in our construction.}

\label{fig:odd_shifting_path}
\end{center}
\end{figure}

% \begin{figure}
%     \centering
%     \includegraphics[scale=0.4]{Order_Shifting.png}
%     \caption{An example of odd shifting path described above. The underlying group is $\mathbb{Z}_n$, and we suppose $d=8\ge 2\mc (\Gamma)$.
%     In this example, we also suppose that $g=0$. The three different shades of vertices represent the three components of the $2$-factor $F$. Moreover, the component containing $P_1$ and $P_2$ is highlighted in bold. {\LG I suppose we can delete this figure since we already have figure 4}}
%     \label{fig:example_factor}
% \end{figure}

The following theorem tells us that under certain reasonable conditions, the lift of an odd shifting path is Hamiltonian.

\begin{theorem} 
\label{thm:base_case}
For an integer $m \geq 2$, let $\Gamma = (v_1, \dots, v_m)$ be an 
odd shifting path with a voltage assignment $\sigma:A(\Gamma) \rightarrow \mathbb{Z}_n$ and an endpoint voltage $r$ coprime to $n$.
Let $\mathcal P = \{P_0,\dots,P_{d-1}\}$ be the family of paths constructed by the process of Lemma \ref{lemma:billiardPaths} using $d = 2 \mc (\Gamma)$
and an element $l \in \mathbb Z_n$. Then, $\Gamma^\sigma$ contains a Hamiltonian cycle $C$ and two edge subsets $E^L\subseteq E(v_1^\sigma)$ and $E^R\subseteq E(v_m^\sigma)$ such that
\begin{itemize}
\item
$C$ contains all edges of the paths $P_0,\dots,P_{d-1}$;
\item 
$|E^L|=|E^R|=d/2$; 
\item $E^L=\{[(v_1,l),(v_1,l+ r)],[(v_1,l+2r),(v_1,l+3r)],\dots,[(v_1,l+(d-2)r),(v_1,l+(d-1)r)]\}$; 
\item
$E^L$ and $E^R$ are alternating consecutive edge sets;
\item $C$ contains all edges of $v_1^\sigma$ except $E^L$ and $C$ contains all edges of $v_m^\sigma$ except $E^R$.
\end{itemize}
\end{theorem}

\begin{proof}
As discussed before, we may assume without loss of generality that $r = 1$ and $l = 0$. 
Let $P_0,\dots,P_{d-1}$ be the $d$ paths produced by the extended billiard strategy described in Lemma \ref{lemma:billiardPaths}.

%Now we construct a Hamiltonian cycle $C$ that contains all edges of the paths $P_0,\dots,P_{d-1}$. 
%By Lemma \ref{lemma:billiardPaths} and applying an appropriate automorphism, we may assume $l=0$ and let the $d$ paths $P_0,P_1,\dots,P_{d-1}$ start at $(v_1,0),\dots,(v_1,d-1)$, respectively. 
%It follows from the discussion above that 
Since $(\Gamma, \sigma)$ is an odd shifting path,
the paths $P_0,\dots,P_{d-1}$ arrive in $v_m^{\sigma}$ respectively at the voltages
\[g + (d-s),g + (d-s+1), \dots, g + (d-1), g , g + 1, \dots, g + (d-s-1)\]
 for some element $g \in \mathbb Z_n$ and some odd number $0 < s < d$.
%\[(v_m,-2k+1),(v_m,-2k+2),\dots,(v_m,-2k+d).\]
%{\PB This isn't even true. We need to apply the odd shift.}
%{\PB for some value $k \in \mathbb Z_n$.}
We now obtain a Hamiltonian cycle $C$
on $\Gamma^{\sigma}$
from $P_0\cup P_1\cup\dots\cup P_{d-1}$ by adding the following edges:
\begin{itemize}
    \item all edges in the unique perfect matching on the vertex set $\{(v_1,i):1\leq i\leq d-2\}$ in $v_1^\sigma$;
    \item all edges in the unique perfect matching on the vertex set 
    $\{(v_m,g + i):1\leq i\leq d-2\}$ ;
    %in $v_m^\sigma$;
    \item all edges of the path $(v_1,d-1),(v_1,d),\dots,(v_1,n-1),(v_1,0)$;
    \item all edges of the path $(v_m,g  + d-1),(v_m,g + d),\dots,(v_m,g + n  -1),(v_m,g)$.
\end{itemize}
We note that for each odd value $1 \leq i \leq d-1$, the endpoints of $P_i$ and $P_{i+1}$ in $v_1^{\sigma}$ are joined by a path in $C \cap E(v_1^{\sigma})$ (with $d$ and $0$ identified). Furthermore, 
since $(\Gamma, \sigma)$ is an odd shifting path, 
for each even value $0 \leq i \leq d-2$, the endpoints of $P_i$ and $P_{i+1}$ in $v_m^{\sigma}$ are joined by a path in $C \cap E(v_m^{\sigma})$.
Hence, $C$ is a Hamiltonian cycle that, starting in $v_1^{\sigma}$, visits $P_0, P_1, \dots, P_{d-1}$ in order.

Finally, we let $E^L$ be the unique perfect matching on $\{(v_1,i):0\leq i \leq d-1\}$ in $v_1^\sigma$ which is not in $C$. Similarly we let $E^R$ be the unique perfect matching on $\{(v_m,g + i):0\leq i\leq d-1\}$ in $v_m^\sigma$. We observe that both $E^L$ and $E^R$ are alternating consecutive edge sets. This completes the proof.
\end{proof}

% \begin{figure}[h]
%     \centering
%     \includegraphics[scale=0.5]{Figure_Thm3_1_Gray.png}
%     \caption{A voltage graph that is a tree that satisfies the conditions in Theorem 3.1 and can be decomposed into one base path and two base spiders. The number beside each vertex denotes the voltage assignment of its loop. If a vertex has no number besides it, it means that its loop can be any voltage within $\mathbb{Z}_n$, where $3$ and $11$ are coprime to $n$. }
%     \label{fig:decompositionNeededNewLabelToMakeTheErrorMessageGoAway}
% \end{figure}

%\begin{definition}
%We define certain vertices of a tree $T$ as \emph{joints}, and each joint will receive a \emph{weight}. We denote the weight of a vertex $v \in V(T)$ as $\omega_T(v)$. 
%\end{definition}

%\begin{definition}

Next, we show that if multiple odd shifting paths are joined together to form a tree, then the lift of this tree is often Hamiltonian.
Let $T$ be a reflexive tree with a voltage assignment $\sigma:A(T) \rightarrow \mathbb{Z}_n$, and $L$ be its set of loops. If there exists a system of odd shifting paths $\{Q_1,\dots,Q_k\}$ such that $\{E(Q_1) \setminus L,\dots E(Q_k) \setminus L\}$ is a partition of $E(T)\setminus L$, and if the paths $Q_i$ and $Q_j$ are strictly internally vertex disjoint\footnote{Two paths P and Q are strictly internally vertex disjoint if each vertex in $V(P) \cap V(Q)$ is an endpoint of both $P$ and $Q$.} for any $i\neq j$, then we say that $T$ can be \emph{odd shifting decomposed}, and we say that $\{Q_1,\dots,Q_k\}$ is the \emph{odd shifting decomposition} of $T$. Observe that in an odd shifting decomposition $\{Q_1, \dots, Q_k\}$, since the odd shifting paths only intersect at their endpoints, the paths $Q_1, \dots, Q_k$ all have the same endpoint voltage.

Let $Q_i$ 
be a path with endpoints $u$ and $v$. We assign \emph{weight} from $Q_i$ to $u$ and $v$, and we write $\omega_{Q_i}(u)=\omega_{Q_i}(u)=2\mc(Q_i)$ for this weight.
If a vertex $v\in V(T)$ is an endpoint of a path in $\{Q_1,\dots,Q_k\}$, we call $v$ a \emph{joint} of $T$. We define the \emph{tree weight} of $v$ in $T$ as $\Omega_T(v)=\sum \omega_{Q_i}(v)$, where the sum is over all paths in $\{Q_1,\dots,Q_k\}$ for which $v$ is an endpoint.
%\end{definition}
We say $(T,\sigma)$ is \emph{properly weighted} if $\Omega_T(v)\leq n$ (the order of the group $\mathbb Z_n$) holds for each joint $v$ of $T$.

Finally, we have the following sufficient condition for the existence of a Hamiltonian cycle in the lift of a tree.

\begin{theorem}
\label{thm:decomp_thm}
Let $T$ be a reflexive tree and let $\sigma:A(T) \rightarrow \mathbb{Z}_n$ be a voltage assignment on $T$. Suppose the voltage graph $(T,\sigma)$ can be odd shifting decomposed into paths $\{Q_1,\dots,Q_k\}$ whose common endpoint voltage is coprime to $n$.
%and that for all $i$, the endpoints of $Q_i$ have the same voltage assignment that is co-prime to $n$. 
If $(T,\sigma)$ is properly weighted, then $T^{\sigma}$ is Hamiltonian.
\end{theorem}

\begin{proof}
As discussed before, we may relabel the elements of our group $G$ so that our paths $Q_i$ all have endpoint voltage $1$.
Our proof is by induction on $k$, the number of odd shifting paths in the decomposition of $T$. We prove the stronger statement that there exists a Hamiltonian cycle $C$ in $T^{\sigma}$ that satisfies the following property:
\begin{enumerate}[(1)]
\item \label{IH} For each joint vertex $v\in V(T)$, there exists a path $P$ with $E(P) \subseteq E(C)\cap E(v^{\sigma})$
such that
 \[|E(P)|\geq n-\Omega_T(v).\]
\end{enumerate}

For the base case, when $k=1$, $T$ is a path with two joints $u$ and $v$. Theorem \ref{thm:base_case} implies that there exists a Hamiltonian cycle $C$ in $\Gamma^\sigma$ along with two alternating consecutive edge sets $E^R\subseteq E(v^\sigma)$ and $E^L\subseteq E(u^\sigma)$ 
of size at most $c(\Gamma)$
such that $C$ contains all edges of $u^\sigma$ and $v^\sigma$ except $E^R$ and $E^L$. Therefore, 
$E(C)\cap E(v^{\sigma})$
contains a path of at least $n- 2\mc(T)=n-\Omega_T(v)$ edges; that is, \ref{IH} holds for $v$. By a similar argument, \ref{IH} holds for $u$.

For the inductive step, suppose that \ref{IH} holds for all values up to $k-1$.
Let $T$ be a tree with a corresponding system $\{Q_1, \dots, Q_k\}$ of odd shifting paths that satisfies the conditions of the theorem. 
Since $\{Q_1, \dots, Q_k\}$ is an edge-partition of the tree $T$, there exists at least one odd shifting path, without loss of generality $Q_k$, that intersects $\bigcup_{i=1}^{k-1}V(Q_i)$ at only one of its endpoints. Observe that
$$T_0 := \bigcup_{i =1}^{k-1} Q_i$$
is connected and hence a tree.
Moreover, $\{Q_1,\dots,Q_{k-1}\}$ is an odd shifting decomposition of $T_0$, and $T_0$ is properly weighted.
%Note that 
%$|V(Q_k) \cap V(T_0)| = 1$, and since 
%removing $S_k$ can break the connectivity of $T$.
Let $u$ be the vertex where $Q_k$ is joined to $T_0$---that is, $\{u\} = V(Q_k) \cap V(T_0)$.
%be an endpoint of $Q_k$
Furthermore, let
$$\Omega_{T_0}(u) = \sum_{\substack{\{i \in [1,k-1]: \\ u\in V(Q_i)\}}} \omega_{Q_i}(u).$$
%and $\Omega_{T_0}(u)$ be the total weight $u$ received in $T_0$. 
By the induction hypothesis, $T_0^\sigma$ has a Hamiltonian cycle $C_0$, and for each joint vertex of $T_0$, there exists a path which satisfies \ref{IH}. In particular, we can find a path 
\[P=(e_1,e_2,\dots,e_s)\subseteq E(C_0)\cap E(u^{\sigma}),\]
where $s\geq n-\Omega_{T_0}(u)$. %Moreover, for each $i=1,\dots,s-1$, $e_i$ and $e_{i+1}$ are incident edges in $C_0$.

Now, recall that $u$ is the endpoint of $Q_k$ in $T_0$, and let $v\neq u$ be the other endpoint of $Q_k$. According to Theorem \ref{thm:base_case}, $Q_k^\sigma$ contains a Hamiltonian cycle $C_k$  along with two alternating consecutive edge sets $E_k^L \subseteq E(u^{\sigma})$ and $E_k^R \subseteq E(v^{\sigma})$ of size $c(Q_k)$, for which 
\[E(C_k) \cap u^{\sigma} = u^{\sigma} \setminus E_k^L \textrm{ and } E(C_k) \cap v^{\sigma} = u^{\sigma} \setminus E_k^R.\] 
Now, we show that we can combine $C_0$ and $C_k$ to a Hamiltonian cycle $C$ in $T^\sigma$, such that $C$ also satisfies the statement \ref{IH}.
%Let $E_k^L\subseteq E(u^\sigma)$ and $E_k^R\subseteq E(v^\sigma)$ be the two alternating consecutive edge subsets that are not in $C_k$. 
By our assumption,
$\omega_{Q_k}(u)=2\mc(Q_k)$,
and $n-\Omega_{T_0}(u)\geq \omega_{Q_k}(u)$.
Since 
\[s \geq n-\Omega_{T_0}(u)\geq 2\mc(Q_k) = |E_k^L|,\]
 by choosing an appropriate value for $l$ in Theorem \ref{thm:base_case} such that $(u,l)$ is an endpoint of $e_1$ but not of $e_2$, we observe that $E_k^L = \{e_1,e_3,\dots,e_{2\mc(Q_k)-1}\}$ in $C_k$. Without loss of generality, we assume that $l = 0 $, so that the endpoints of each edge $e_i \in P$ are $(u,i-1)$ and $(u,i)$.
 
 We define $A_0= C_0\setminus P$, and we see that $A_0$ is a Hamiltonian path on $T_0^{\sigma} \setminus P$ with endpoints $(u,0)$ and $(u,s)$. Similarly, we define $A_k = C_k \setminus  (E(u^{\sigma}) \setminus P)$,
 and we see that $A_k$ is a Hamiltonian path 
 on the graph 
 $Q_k^{\sigma} \setminus (E(u^{\sigma}) \setminus P)$
 with endpoints $(u,0)$ and $(u,s)$. 
 Therefore, $V(A_0) \cup V(A_k) = V(T^{\sigma})$, and since $A_0$ and $A_k$ only intersect at their endpoints, it follows that $C = A_0 \cup A_k$ is a Hamiltonian cycle on $T^{\sigma}$. See Figure \ref{fig:example_factor} for an example of this construction.
 %By connecting the endpoints of $P_0$ and $P_k$ respectively, we obtain a Hamiltonian cycle $C$.

Now, we show that $C$ satisfies \ref{IH}.
After constructing $C$, we observe that 
$E(C) \cap v^{\sigma} = v^{\sigma} \setminus E_k^R$.
Since $E_k^R$ is an alternating consecutive edge set of size $\mc(Q_k)$, $E(v^\sigma)\cap E(C)$ contains a path of at least $n - 2\mc(Q_k) = n - \Omega_T(v)$ edges.
Next, we
consider $E(u^\sigma)\cap E(C)$. When constructing $C$ from $C_k$, only the edge set $\{e_1,e_3,\dots,e_{2\mc(Q_k)-1}\}$ was deleted from $C_0$,
so by the induction hypothesis, $u^{\sigma} \cap E(C)$ still has a path $(e_{2\mc(Q_k)}, \dots, e_s)$ with
%$\{e_1,e_2,\dots,e_{2\mc(Q_k)}\}$ are not consecutive edge set. Thus we have at least 
$$r-2\mc(Q_k) + 1 >  (n-\Omega_{T_0}(u))-\omega_{Q_k}(u)=n-\Omega_T(u)$$
edges.
Hence, both joint vertices satisfy the condition \ref{IH}, and the condition \ref{IH} has not changed for the remaining joint vertices of $T$. Therefore, \ref{IH} holds for all vertices of $T$, and induction is complete. 
\end{proof}

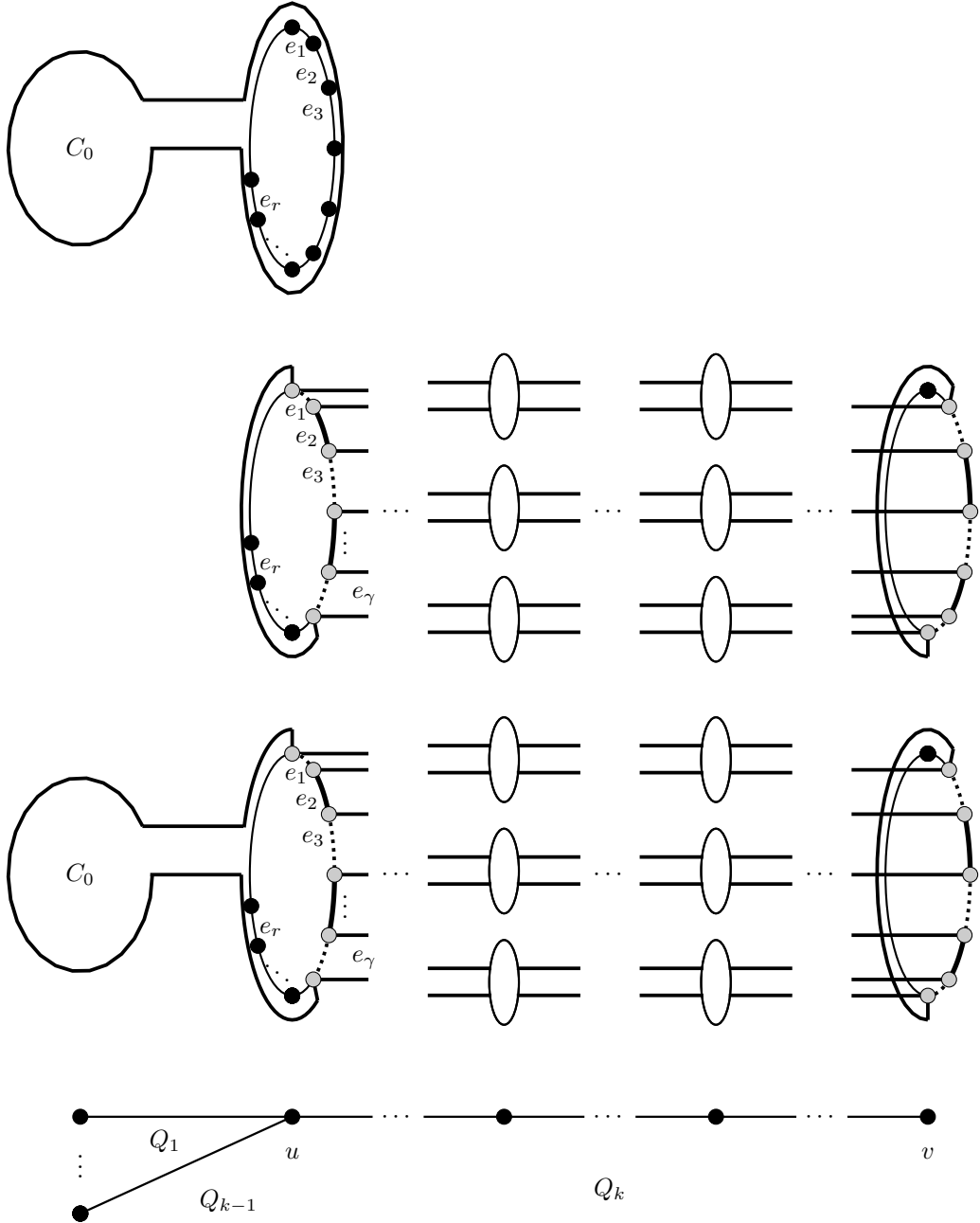
\begin{figure}[] 

    \begin{center}
        \begin{tikzpicture}
        [scale=1.7,xscale = 0.35, yscale = 1]
        
        %labels for paths
        % \node (z) at (5.5,1.2) [ draw = white, fill = white]  {$P_3$};
        \node (z) at (5.1,0.816) [ draw = white, fill = white]  {$e_1$};
        \node (z) at (5.35,0.6) [ draw = white, fill = white]  {$e_2$};
        \node (z) at (5.5,0.3) [ draw = white, fill = white]  {$e_3$};
        
        \node (z) at (4.5,-0.45) [ draw = white, fill = white]  {$e_r$};
        \node (z) at (6.7,-0.7) [ draw = white, fill = white]  {$e_{\gamma}$};
        \node (z) at (6.25,-0.2) [ draw = white, fill = white]  {$\vdots$};
        \node (z) at (4.65,-0.75) [ draw = white, fill = white]  {$\ddots$};
        \node (z) at (0,0) [ draw = white, fill = white]  {$C_0$};
        
        %Line 2
        \node (z) at (5.1,3.816) [ draw = white, fill = white]  {$e_1$};
        \node (z) at (5.35,3.6) [ draw = white, fill = white]  {$e_2$};
        \node (z) at (5.5,3.3) [ draw = white, fill = white]  {$e_3$};
        
        \node (z) at (4.5,-0.45+3) [ draw = white, fill = white]  {$e_r$};
        \node (z) at (6.7,-0.7+3) [ draw = white, fill = white]  {$e_{\gamma}$};
        \node (z) at (6.25,-0.2+3) [ draw = white, fill = white]  {$\vdots$};
        \node (z) at (4.65,-0.75+3) [ draw = white, fill = white]  {$\ddots$};
        
        %line 3
        \node (z) at (5.1,6.816) [ draw = white, fill = white]  {$e_1$};
        \node (z) at (5.35,6.6) [ draw = white, fill = white]  {$e_2$};
        \node (z) at (5.5,6.3) [ draw = white, fill = white]  {$e_3$};
        
        \node (z) at (4.5,-0.45+6) [ draw = white, fill = white]  {$e_r$};
        
        \node (z) at (4.65,-0.75+6) [ draw = white, fill = white]  {$\ddots$};
        \node (z) at (0,6) [ draw = white, fill = white]  {$C_0$};
        
        %edges on top
        \draw [line width=0.3mm] (0,-2)  -- (20,-2)   ;
        \draw [line width=0.3mm] (0,-2.8)  -- (5,-2)   ;

        \foreach \shift in {5}
        {
        \foreach \s in {0,1,2,3,4}
        {
        
        %vertices on top path
        \node (z) at (\s*\shift,-2)  [ draw = black, circle,fill=black,minimum size = 6pt,inner sep=0pt] {}; %I changed the color here, sorry if that was wrong
        \node (z) at (0,-2.8)  [ draw = black, circle,fill=black,minimum size = 6pt,inner sep=0pt] {}; %Here too
        %label of vertex on top path
        \node (z) at (\shift,-2.3) [ draw = white, fill = white]  {$u$};
        \node (z) at (4*\shift,-2.3) [ draw = white, fill = white]  {$v$};

        \ifthenelse{\NOT 4 = \s \AND \NOT 0=\s}{
        %dots on top edges
        \node (z) at ({(\s+0.5)*\shift},-2) [ draw = white, fill = white]  {$\cdots$};
        
        %outgoing directed line
        \foreach \k in {0,1}{
        \ifthenelse{1=\s}{

		%Here be the arcs outgoing from the fibre above $u$ 
        \draw [line width=0.5mm] ({\shift*\s +cos(90-30*(\s-1))},{sin(90-30*(\s-1))+3*\k})  -- ({\shift*\s +cos(0)+0.8},{sin(90-30*(\s-1))+3*\k}) node[ sloped, scale = 1] {} ;
        \draw [line width=0.5mm] ({\shift*\s +cos(60-30*(\s-1))},{sin(60-30*(\s-1))+3*\k})  -- ({\shift*\s +cos(0)+0.8},{sin(60-30*(\s-1))+3*\k}) node[sloped, scale = 1] {} ;
        \draw [line width=0.5mm] ({\shift*\s +cos(30-30*(\s-1))},{sin(30-30*(\s-1))+3*\k})  -- ({\shift*\s +cos(0)+0.8},{sin(30-30*(\s-1))+3*\k}) node[ sloped, scale = 1] {} ;
        \draw [ line width=0.5mm] ({\shift*\s +cos(0-30*(\s-1))},{sin(0-30*(\s-1))+3*\k})  --  ({\shift*\s +cos(0)+0.8},{sin(0-30*(\s-1))+3*\k}) node[sloped, scale = 1] {} ;
        \draw [ line width=0.5mm] ({\shift*\s +cos(0-30*(\s))},{sin(0-30*(\s))+3*\k})  --  ({\shift*\s +cos(0)+0.8},{sin(0-30*(\s))+3*\k}) node[sloped, scale = 1] {} ;
        \draw [ line width=0.5mm] ({\shift*\s +cos(0-30*(\s+1))},{sin(0-30*(\s+1))+3*\k})  --  ({\shift*\s +cos(0)+0.8},{sin(0-30*(\s+1))+3*\k}) node[ sloped, scale = 1] {} ;
        }
        {

        %\draw [line width=0.5mm] ({\shift*\s +0.3},{0.2+3*\k})  -- ({\shift*\s +cos(0)+0.8},{0.2+3*\k}) node[currarrow, pos=0.5, sloped, scale = 1] {} ;
        %\draw [ line width=0.5mm] ({\shift*\s +0.3)},{0+3*\k})  --  ({\shift*\s +cos(0)+0.8},{0+3*\k}) node[currarrow, pos=0.5, sloped, scale = 1] {} ;
        
        %\draw [line width=0.5mm] ({\shift*\s +0.3},{-1+3*\k})  -- ({\shift*\s +cos(0)+0.8},{-1+3*\k}) node[currarrow, pos=0.5, sloped, scale = 1] {} ;
        %\draw [ line width=0.5mm] ({\shift*\s +0.3},{-1.2+3*\k})  --  ({\shift*\s +cos(0)+0.8},{-1.2+3*\k}) node[currarrow, pos=0.5, sloped, scale = 1] {} ;
        };
        
        %dots between circles

        \node (z) at ({\shift*\s +cos(0)+1.5},{3*\k}) [draw = white, fill = white]  {$\cdots$};

        %incoming directed line
        \ifthenelse{ 3 =\s}{
        
		%Here be the arcs incoming to the fibre above $v$
        \draw [line width=0.5mm] ({\shift*(\s+1) +cos(90-30*(\s-2))},{sin(90-30*(\s-2))+3*\k})  -- ({\shift*\s +cos(0)+2.2},{sin(90-30*(\s-2))+3*\k}) node[ sloped, scale = 1] {} ;
        \draw [line width=0.5mm] ({\shift*(\s+1) +cos(90-30*\s)},{sin(90-30*(\s-1))+3*\k})  -- ({\shift*\s +cos(0)+2.2},{sin(90-30*(\s-1))+3*\k}) node[ sloped, scale = 1] {} ;
        \draw [line width=0.5mm] ({\shift*(\s+1) +cos(90-30*\s)},{sin(90-30*\s)+3*\k})  -- ({\shift*\s +cos(0)+2.2},{sin(90-30*\s)+3*\k}) node[ sloped, scale = 1] {} ;
        \draw [line width=0.5mm] ({\shift*(\s+1) +cos(60-30*\s)},{sin(60-30*\s)+3*\k})  -- ({\shift*\s +cos(0)+2.2},{sin(60-30*\s)+3*\k}) node[ sloped, scale = 1] {} ;
        \draw [line width=0.5mm] ({\shift*(\s+1) +cos(30-30*\s)},{sin(30-30*\s)+3*\k})  -- ({\shift*\s +cos(0)+2.2},{sin(30-30*\s)+3*\k}) node[ sloped, scale = 1] {} ;
        \draw [ line width=0.5mm] ({\shift*(\s+1) +cos(0-30*\s)},{sin(0-30*\s)+3*\k})  --  ({\shift*\s +cos(0)+2.2},{sin(0-30*\s)+3*\k}) node[ sloped, scale = 1] {} ;
        
        }
        {
        
        %\draw [line width=0.5mm] ({\shift*(\s+1)         -2},{sin(60)+0.3+2.95*\k})  -- ({\shift*(\s+1)-0.3},{sin(60)+0.3+2.95*\k}) node[currarrow, pos=0.5, sloped, scale = 1] {} ;
        %\draw [line width=0.5mm] ({\shift*(\s+1) -2},{sin(40)+0.3+2.95*\k})  -- ({\shift*(\s+1) -0.3},{sin(40)+0.3+2.95*\k}) node[currarrow, pos=0.5, sloped, scale = 1] {} ;
        
        %\draw [line width=0.5mm] ({\shift*(\s+1) -2},{0.2+3*\k})  -- ({\shift*(\s+1) -0.3},{0.2+3*\k}) node[currarrow, pos=0.5, sloped, scale = 1] {} ;
        %\draw [ line width=0.5mm] ({\shift*(\s+1) -2)},{0+3*\k})  --  ({\shift*(\s+1) -0.3},{0+3*\k}) node[currarrow, pos=0.5, sloped, scale = 1] {} ;
        
        %\draw [line width=0.5mm] ({\shift*(\s+1) -2},{-1+3*\k})  -- ({\shift*(\s+1)-0.3},{-1+3*\k}) node[currarrow, pos=0.5, sloped, scale = 1] {} ;
        %\draw [ line width=0.5mm] ({\shift*(\s+1)-2},{-1.2+3*\k})  --  ({\shift*(\s+1)-0.3},{-1.2+3*\k}) node[currarrow, pos=0.5, sloped, scale = 1] {} ;
        };
        
        }
        
        }{}
        
        \ifthenelse{  1 = \s}{
        
        %circle bold line around each fibre
         \draw [line width = 0.5mm,domain={120-30*\s}:{191-30*(\s)}] plot ({\shift * \s + 1.2*cos(\x)}, {1.2*sin(\x)});
         \draw [line width = 0.5mm,domain={210-30*\s}:{330-30*(\s)}] plot ({\shift * \s + 1.2*cos(\x)}, {1.2*sin(\x)});
         
        %line 2
         \draw [line width = 0.5mm,domain={120-30*\s}:{330-30*(\s)}] plot ({\shift * \s + 1.2*cos(\x)}, {1.2*sin(\x)+3});
        
        %line 3
         \draw [line width = 0.5mm,domain={191-30*\s}:{-150-30*(\s)}] plot ({\shift * \s + 1.2*cos(\x)}, {1.2*sin(\x)+6});
        
        %connecting part
         \draw [line width = 0.5mm]   ({\shift * \s + 1.2*cos(120-30*\s)}, {1.2*sin(120-30*\s)}) --  ({\shift * \s + cos(120-30*\s)}, {sin(120-30*\s)}) ;
        
         \draw [line width = 0.5mm]   ({\shift * \s + 1.2*cos(360-30*(\s+1))}, {1.2*sin(360-30*(\s+1))}) --  ({\shift * \s + cos(360-30*(\s+1))}, {sin(360-30*(\s+1))}) ;
         
         %line 2
          \draw [line width = 0.5mm]   ({\shift * \s + 1.2*cos(120-30*\s)}, {1.2*sin(120-30*\s)+3}) --  ({\shift * \s + cos(120-30*\s)}, {sin(120-30*\s)+3}) ;
        
         \draw [line width = 0.5mm]   ({\shift * \s + 1.2*cos(360-30*(\s+1))}, {1.2*sin(360-30*(\s+1))+3}) --  ({\shift * \s + cos(360-30*(\s+1))}, {sin(360-30*(\s+1))+3}) ;
        
        	}{};
        
        \ifthenelse{4=\s}{
         \draw [line width = 0.5mm,domain={60}:{270}] plot ({\shift * \s + 1.2*cos(\x)}, {1.2*sin(\x)});
        
         \draw [line width = 0.5mm]   ({\shift * \s + 1.2*cos(120-30*(\s-2))}, {1.2*sin(120-30*(\s-2))}) --  ({\shift * \s + cos(120-30*(\s-2))}, {sin(120-30*(\s-2))}) ;
        
         \draw [line width = 0.5mm]   ({\shift * \s + 1.2*cos(360-30*(\s-1))}, {1.2*sin(360-30*(\s-1))}) --  ({\shift * \s + cos(360-30*(\s-1))}, {sin(360-30*(\s-1))}) ;
         
         %line 2
         
          \draw [line width = 0.5mm,domain={60}:{270}] plot ({\shift * \s + 1.2*cos(\x)}, {1.2*sin(\x)+3});
        
         \draw [line width = 0.5mm]   ({\shift * \s + 1.2*cos(120-30*(\s-2))}, {1.2*sin(120-30*(\s-2))+3}) --  ({\shift * \s + cos(120-30*(\s-2))}, {sin(120-30*(\s-2))+3}) ;
        
         \draw [line width = 0.5mm]   ({\shift * \s + 1.2*cos(360-30*(\s-1))}, {1.2*sin(360-30*(\s-1))+3}) --  ({\shift * \s + cos(360-30*(\s-1))}, {sin(360-30*(\s-1))+3}) ;
         }{};

        \ifthenelse{\NOT 0 = \s}{
        
        %Q label above edge
        \node (z) at ({(3-0.5)*\shift},-2.6) [ draw = white, fill = white]  {$Q_{k}$};
        \node (z) at ({(1-0.6)*\shift},-2.2) [ draw = white, fill = white]  {$Q_{1}$};
        \node (z) at ({(0.7)*\shift},-2.7) [ draw = white, fill = white]  {$Q_{k-1}$};
        \node (z) at (0,-2.35) [ draw = white, fill = white]  {$\vdots$};
        
        }
        {
        %bold line around C_0 fibre
        \draw [line width=0.5mm] ({\shift*\s +cos(30-30*(\s))+0.6},{sin(30-30*(\s))-0.1})  -- ({\shift*\s +cos(0)+2.85},{sin(30-30*(\s))-0.1}) node[ pos=0.5, sloped, scale = 1] {} ;
        \draw [ line width=0.5mm] ({\shift*\s +cos(0-30*(\s))+0.65},{sin(0-30*(\s))})  --  ({\shift*\s +cos(0)+2.8},{sin(0-30*(\s))}) node[ pos=0.5, sloped, scale = 1] {} ;

        \draw [line width = 0.5mm,domain={30-30*\s}:{360-30*(\s)}] plot ({\shift * \s + 1.7*cos(\x)}, {0.8*sin(\x)});
        
        %line 3
        \draw [line width=0.5mm] ({\shift*\s +cos(30-30*(\s))+0.6},{sin(30-30*(\s))-0.1+6})  -- ({\shift*\s +cos(0)+2.85},{sin(30-30*(\s))-0.1+6}) node[ pos=0.5, sloped, scale = 1] {} ;
        \draw [ line width=0.5mm] ({\shift*\s +cos(0-30*(\s))+0.65},{sin(0-30*(\s))+6})  --  ({\shift*\s +cos(0)+2.8},{sin(0-30*(\s))+6}) node[ pos=0.5, sloped, scale = 1] {} ;

        \draw [line width = 0.5mm,domain={30-30*\s}:{360-30*(\s)}] plot ({\shift * \s + 1.7*cos(\x)}, {0.8*sin(\x)+6});
        
        };
        
        %circles
        \ifthenelse{\NOT 0=\s \AND \NOT 2=\s \AND \NOT 3=\s}{
        \foreach \k in {0,1}{
        \draw[thick] (\shift*\s,3*\k) circle [thick, radius=1];
           %removed e edges
            \ifthenelse{1=\s }{
            \draw [white,line width = 0.5mm,domain={120-30*\s}:{90-30*(\s)}] plot ({\shift * \s + cos(\x)}, {sin(\x)+3*\k});
             \draw [dotted,line width = 0.5mm,domain={120-30*\s}:{90-30*(\s)}] plot ({\shift * \s + cos(\x)}, {sin(\x)+3*\k});
            \draw [line width = 0.7mm,domain={90-30*\s}:{60-30*(\s)}] plot ({\shift * \s + cos(\x)}, {sin(\x)+3*\k});
             
             \draw [white,line width = 0.5mm,domain={60-30*\s}:{30-30*(\s)}] plot ({\shift * \s + cos(\x)}, {sin(\x)+3*\k});
             \draw [dotted,line width = 0.5mm,domain={60-30*\s}:{30-30*(\s)}] plot ({\shift * \s + cos(\x)}, {sin(\x)+3*\k});
             \draw [line width = 0.7mm,domain={30-30*\s}:{-30*(\s)}] plot ({\shift * \s + cos(\x)}, {sin(\x)+3*\k});
             
             \draw [white,line width = 0.5mm,domain={-30*\s}:{-30-30*(\s)}] plot ({\shift * \s + cos(\x)}, {sin(\x)+3*\k});
             \draw [dotted,line width = 0.5mm,domain={-30*\s}:{-30-30*(\s)}] plot ({\shift * \s + cos(\x)}, {sin(\x)+3*\k});
             
             }{
             \ifthenelse{4= \s}{
               \draw [white,line width = 0.5mm,domain={120-30*\s/2}:{90-30*(\s/2)}] plot ({\shift * \s + cos(\x)}, {sin(\x)+3*\k});
             \draw [dotted,line width = 0.5mm,domain={120-30*\s/2}:{90-30*(\s/2)}] plot ({\shift * \s + cos(\x)}, {sin(\x)+3*\k});
            \draw [line width = 0.7mm,domain={90-30*\s/2}:{60-30*(\s/2)}] plot ({\shift * \s + cos(\x)}, {sin(\x)+3*\k});
             
             \draw [white,line width = 0.5mm,domain={60-30*\s/2}:{30-30*(\s/2)}] plot ({\shift * \s + cos(\x)}, {sin(\x)+3*\k});
             \draw [dotted,line width = 0.5mm,domain={60-30*\s/2}:{30-30*(\s/2)}] plot ({\shift * \s + cos(\x)}, {sin(\x)+3*\k});
             \draw [line width = 0.7mm,domain={30-30*\s/2}:{-30*(\s/2)}] plot ({\shift * \s + cos(\x)}, {sin(\x)+3*\k});
             
             \draw [white,line width = 0.5mm,domain={-30*\s/2}:{-30-30*(\s/2)}] plot ({\shift * \s + cos(\x)}, {sin(\x)+3*\k});
             \draw [dotted,line width = 0.5mm,domain={-30*\s/2}:{-30-30*(\s/2)}] plot ({\shift * \s + cos(\x)}, {sin(\x)+3*\k});
             
             }{};};}

             \ifthenelse{1=\s}{
             \draw[thick] (\shift*\s,6) circle [thick, radius=1];
             \foreach \i in {0,1,2,3,9,10,11}
            {	\node (z) at ({\shift*\s + cos(30*\i)},{sin(30*\i)+6}) [ draw = black, circle,fill=black,minimum size = 6pt,inner sep=0pt]  {};}
             }{};
        \foreach \i in {0,1,2,3,9,10,11}
            {
        
        		\foreach \k in {0,1}{
        
                %gray vertices on circles
        		\node (z) at ({\shift*\s + cos(30*\i)},{sin(30*\i)+3*\k}) [ draw = black, circle,fill=gray!40!white,minimum size = 6pt,inner sep=0pt]  {};
        		%black vertices that are not used in billiard strategy
        		\ifthenelse{1=\s}{
            		\node (z) at ({\shift*\s + cos(270)},{sin(270)+3*\k}) [ draw = black, circle,fill=black,minimum size = 6pt,inner sep=0pt]  {};
        		}
        		{
        			\node (z) at ({\shift*\s + cos(90)},{sin(90)+3*\k}) [ draw = black, circle,fill=black,minimum size = 6pt,inner sep=0pt]  {};
        		}
        		
        		}
        		
            }

        }
        {
        \ifthenelse{\NOT 0=\s }{
        \foreach \k in {0,1}{
        
        %\draw[thick] (\shift*\s,3*\k) circle [thick, radius=0.35];
        %\draw[thick] (\shift*\s,-1+3*\k) circle [thick, radius=0.35];
        }}
        {

        }
        }
        
        }
        %end for loop s=0,1,2,3,4
        }
        %end for loop of 5

        %this code is for the extra two vertices on the left sides of the cycles above $u$
        \foreach \k in {0,1,2}{
        		\node (z) at ({ 5+cos(30*5.5)},{sin(30*6.5)+3*\k}) [ draw = black, circle,fill=black,minimum size = 6pt,inner sep=0pt]  {};
        		\node (z) at ({5+cos(30*7.2)},{sin(30*7.2)+3*\k}) [ draw = black, circle,fill=black,minimum size = 6pt,inner sep=0pt]  {};
        }

\foreach \level in {0,1} {
	\foreach \sublevel in {0,1,2} { %Three sublevels in one level
		\foreach \fibre in {0,1} { %Which fibre
		\foreach \subshift in {0.92} { %This changes the distance between sublevels
		\foreach \vertshift in {0.25} { %This one can be adjusted to move everything up or down
		\foreach \extrashift in {0,1} { %This one puts the arcs on both sides of the circles
			\draw[thick] (5*\fibre+10,0.7 +  \vertshift  +3- \subshift * \sublevel - 3*\level) circle [thick, radius=0.35];
        	\draw [line width=0.5mm] ({5*\fibre+10.3 -\extrashift*2.1  },{sin(60)+  \vertshift  +2.95 - \subshift * \sublevel - 3*\level})  -- ({5*\fibre+10+cos(0)+0.8-\extrashift*2.1},{sin(60)+ \vertshift +2.95 -  \subshift * \sublevel - 3*\level}) node[sloped, scale = 1] {} ;
        	\draw [line width=0.5mm] ({5*\fibre+10.3-\extrashift*2.1},{sin(40)+ \vertshift +2.95- \subshift * \sublevel -3*\level})  -- ({5*\fibre+10 +cos(0)+0.8-\extrashift*2.1},{sin(40) + \vertshift +2.95- \subshift * \sublevel -3*\level })
        	node[ sloped, scale = 1] {} ;
        	%node[currarrow, pos=0.5, sloped, scale = 1] {} ;
		} } } }
	}
}

        \end{tikzpicture}
\end{center}
    \caption{The figure shows four graphs. The first (top) graph represents $T_0^{\sigma}$, and a Hamiltonian cycle $C_0$ is shown. The second graph shows the Hamiltonian cycle $C_k$ constructed on $Q_k^{\sigma}$ using Theorem \ref{thm:base_case}. The third graph shows the combination of $C_0$ and $C_k$ that gives a Hamiltonian cycle $C$ on $T^{\sigma}$ as in the proof of Theorem \ref{thm:decomp_thm}.
    (In the second and third graphs, we write $\gamma = 2\mc(Q_k)-1$.)
    Finally, the last figure shows the base graph $T$.} \label{fig:example_factor}
\end{figure}

\section{Large prime cyclic groups}
\label{sec:large}
Let $T$ be a reflexive tree whose edges are labelled with integers by some function $\phi$. In this section, we consider the following question: Under which conditions is the lift of $T$ Hamiltonian, when the labels from $\phi$ are taken as elements of a large prime-ordered cyclic group $\mathbb{Z}_p$? As we always consider this specific question, it is convenient to fix some notation as follows.

Given a reflexive tree $T$, we let $\phi:E(T) \rightarrow \mathbb{Z}$ be an assignment of integers to the edges of $T$, and we let $\phi$ be positive at each loop of $T$.
We often write $\phi(v)$ instead of $\phi(v,v)$ for a vertex $v \in V(T)$.
We let $p$ be a prime number, and we  typically let $p$ be large. We let $\psi:\mathbb{Z} \rightarrow \mathbb{Z}_p$ be the natural homomorphism $x \mapsto x + \langle p \rangle$.
Then, for each edge $e \in E(T)$ with corresponding arcs $e^+, e^- \in A(T)$, we let $\sigma$ assign the values $\psi \circ \phi(e)$ and $-\psi \circ \phi(e)$ to $e^+$ and $e^-$, respectively.
Therefore, given $T$ and $\phi$, $\sigma$ gives a voltage assignment to $A(T)$ over $\mathbb{Z}_p$ by interpreting the integers given by $\phi$ as group elements in $\mathbb{Z}_p$. 
For a vertex $v \in V(T)$, we often say that $\sigma(v)$ is the voltage of $v$.
Note that since $p$ is prime and $\phi$ is nonzero, for each vertex $v \in V(T)$, $\sigma(v)$ is a generator for $\mathbb Z_p$, and therefore each fiber of $T^{\sigma}$ contains a single cycle.
With this notation in place, we may rephrase our question more precisely.
\begin{question}
\label{question:bigPrime}
Let $T$ be a reflexive tree, and let $\phi$ be an assignment of integers to $E(T)$ as described above. Does there exist a number $N \in \mathbb{N}$ such that $T^{\sigma}$ is Hamiltonian whenever $p \geq N$?
\end{question}

While we do not have an answer for Question \ref{question:bigPrime}, we are able to show the following results. 

\begin{theorem}
\label{thmLeaf}
Let $T$ be a reflexive tree, and let $\phi$ be an assignment of integers to $E(T)$ as described above. Suppose that there exists a vertex $v\in V(T)$ such that for every neighbor $u$ of $v$, $\phi(u)= \phi(v)$. Then, when $p$ is a sufficiently large prime, $T^{\sigma}$ is Hamiltonian.
\end{theorem}

%%Figure-6%%
\begin{figure}
    \centering
    \begin{tikzpicture}
    [scale=1,auto=left,every node/.style={circle,fill=gray!30},minimum size = 6pt,inner sep=0pt]
%\clip (1,-3.35) rectangle + (21.5,4.9);

%%%%TREEEEEEEE
\node (z) at (4,-0.5) [draw = white, fill = white] {$3$};
\node (z) at (5,-0.5) [draw = white, fill = white] {$3$};
\node (z) at (6,-0.5) [draw = white, fill = white] {$3$};
\node (z) at (5,1) [draw = white, fill = white] {$3$};

%\node (z) at (1.5,-1.5) [draw = white, fill = white] {};
%\node (a1) at (0,0) [draw = black] {};
%\node (a2) at (1,0) [draw = black] {};
\node (f3) at (2,0) [draw = black] {};
\node (f4) at (3,0) [draw = black] {};
\node (f5) at (4,0) [draw = black,fill=black] {};
\node (f6) at (3.7,1) [draw = black] {};
\node (f7) at (3,1) [draw = black] {};
%\node (f9) at (4,0.5) [draw = black] {};
\node (f10) at (5,0) [draw = black,fill=black] {};
\node (h1) at (6,0) [draw = black,fill=black] {};
\node (h2) at (6,0.5) [draw = black] {};
\node (a) at (4.3,1) [draw = black] {};
\node (a1) at (5,0.5) [draw = black,fill=black] {};
\node (b) at (2,1) [draw = black] {};

\foreach \from/\to in {
%a1/a2,a2/a3,
f3/f4,f4/f5,f5/f6,f4/f7,f5/f10,f10/h1,h1/h2,a/f5,f10/a1,b/f4}
    \draw (\from) -- (\to);
    \end{tikzpicture}
    \caption{This figure shows a reflexive voltage tree (with loops omitted) that satisfies the conditions in Theorem \ref{thmLeaf}. The number beside each vertex denotes the voltage assignment of its loop. If a vertex has no number beside it, then its loop can have any nonzero voltage within $\mathbb{Z}_p$.}
    \label{fig:thmTree}
\end{figure}
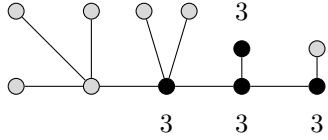

\begin{theorem}
\label{thm1}
Let $T$ be a reflexive tree, and let $\phi$ be an assignment of integers to $E(T)$ as described above. Suppose that there exist two adjacent vertices $u,v \in V(T)$ for which $\phi(u) = \phi(v) = 1$. Then, when $p$ is a sufficiently large prime, $T^{\sigma}$ is Hamiltonian. 
\end{theorem}

%%Figure-6%%
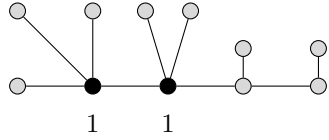
\begin{figure}
    \centering
    \begin{tikzpicture}
    [scale=1,auto=left,every node/.style={circle,fill=gray!30},minimum size = 6pt,inner sep=0pt]
%\clip (1,-3.35) rectangle + (21.5,4.9);

%%%%TREEEEEEEE
\node (z) at (4,-0.5) [draw = white, fill = white] {$1$};
\node (z) at (3,-0.5) [draw = white, fill = white] {$1$};

%\node (z) at (1.5,-1.5) [draw = white, fill = white] {};
%\node (a1) at (0,0) [draw = black] {};
%\node (a2) at (1,0) [draw = black] {};
\node (f3) at (2,0) [draw = black] {};
\node (f4) at (3,0) [draw = black,fill=black] {};
\node (f5) at (4,0) [draw = black,fill=black] {};
\node (f6) at (3.7,1) [draw = black] {};
\node (f7) at (3,1) [draw = black] {};
%\node (f9) at (4,0.5) [draw = black] {};
\node (f10) at (5,0) [draw = black] {};
\node (h1) at (6,0) [draw = black] {};
\node (h2) at (6,0.5) [draw = black] {};
\node (a) at (4.3,1) [draw = black] {};
\node (a1) at (5,0.5) [draw = black] {};
\node (b) at (2,1) [draw = black] {};

\foreach \from/\to in {
%a1/a2,a2/a3,
f3/f4,f4/f5,f5/f6,f4/f7,f5/f10,f10/h1,h1/h2,a/f5,f10/a1,b/f4}
    \draw (\from) -- (\to);
    \end{tikzpicture}
    \caption{This figure shows a reflexive voltage tree (with loops omitted) that satisfies the conditions of Theorem \ref{thm1}. The number beside each vertex denotes the voltage assignment of its loop. If a vertex has no number beside it, then its loop can have 
    any nonzero voltage within $\mathbb{Z}_p$.}
    \label{fig:thm1}
\end{figure}

Trees that satisfy the conditions of Theorem \ref{thmLeaf} and Theorem \ref{thm1} are sketched in Figures \ref{fig:thmTree} and \ref{fig:thm1}.
For both of these theorems, we use the following general strategy. First, we describe how to build a $2$-factor in $T^{\sigma}$ satisfying certain conditions. Then we use the vertex $u$ in the case of Theorem \ref{thmLeaf} and the pair $u,v$ in the case of Theorem \ref{thm1} to connect all components of our $2$-factor into a single Hamiltonian cycle.
In the proofs of both of these theorems, we assume that $p$ is a sufficiently large prime number. We also use Lemma \ref{lemZero} to assume without loss of generality that $\phi(e) = 0$ for every cut-edge $e \in E(T)$.

We establish some definitions and lemmas that help us prove these results. For a vertex $v \in V(T)$, and for two integers $0 \leq a,b \leq p-1$, we define $v^{\sigma}[a,b]$ to be the graph induced by the vertex set $\{(v,i): a \leq i \leq b\}$. We also have the following definition.

\begin{definition}
\label{defMatching}
Let $v \in V(T)$. 
%be a vertex with a loop $e$, and suppose that there is a fixed voltage assignment $\sigma(v): \{e\} \rightarrow \mathbb{Z}_p$ at $e$. 
For each positive multiple $N$ of $2 \varphi(v)$, we let $M_v(N)$ denote the matching in $v^{\sigma}$ containing all edges of the form
$$\{(v, 2a\sigma(v) + i), (v, (2a+1)\sigma(v) + i)\}$$
for $a \in \{0, 1, \dots, \frac{N}{2 \varphi(v)} - 1\}$, $i \in \{0,1,\dots,\varphi(v) - 1\}$. 
%{\LS When I first looked at this definition, I was a bit confused but then one can nicely see that this will be matching! At this point the choice of intervals for $a$ or $i$ are not clear. Shall we say something how this will be used? Why it is enough to take that many matching edges?} 
Furthermore, for an element $g \in G$, we say that $M_v(N)+g$ denotes the matching obtained from $M_v(N)$ by applying the automorphism $(v,t) \mapsto (v,t+g)$.
\end{definition}

Figure \ref{figMatching} shows a matching $M_v(N)$ for $\varphi(v) = 3$ and $N = 12$, as well as a matching $M_v(N) + 1$.
The matchings described in Definition \ref{defMatching} are useful when we construct our $2$-factor on $T^{\sigma}$ as described earlier. Informally, for two adjacent vertices $w, x\in V(T)$, if we have a $2$-factor $F$ of the lift of the component of $T-x$ containing $w$ that contains many edges of $w^{\sigma}$ and a second $2$-factor $F_0$ of the lift of the component of $T-w$ containing $x$ that contains many edges of $x^{\sigma}$, then we may ``join" $F$ and $F_0$ by removing two matchings of this form of equal size from $w^{\sigma}$ and $x^{\sigma}$ and replacing them with a matching consisting of edges of the form $[(w,g), (x,g)]$. 
Figure \ref{figGrowingTree} shows two $2$-factors being joined in a similar way to what we have described here. The removed edges are depicted in gray, and the added edges are shown as vertical edges.
The following observations are useful for us when we use these matchings.

%For example, if $\sigma(v) = 3$ and $N = 12$, the matching $M_v(N)+1$ contains the edges whose endpoints correspond to the following $\mathbb{Z}_p$ element pairs: $ (1,4), (2,5), (3,6),  (7,10), (8,11), (9,12)$.

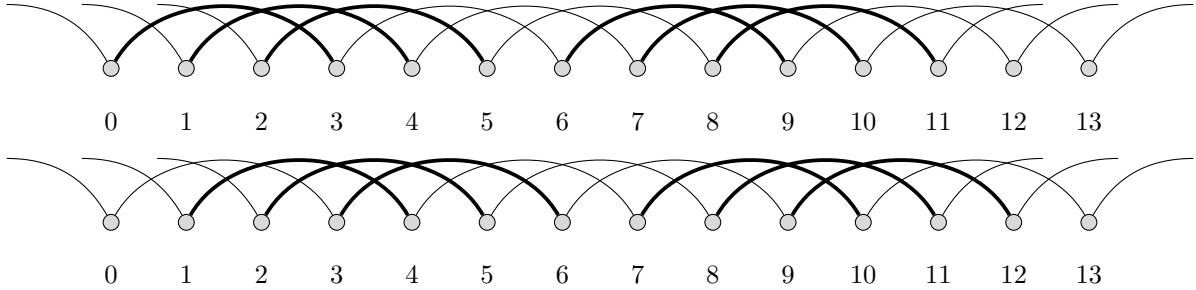
\begin{figure}
\begin{tikzpicture}
  [scale=1,auto=left,every node/.style={circle,fill=gray!30},minimum size = 6pt,inner sep=0pt]
 %[scale=2.2,auto=left,every node/.style={circle,fill=gray!30,minimum size = 6pt,inner sep=0pt}]
%\node (z) at (0,-0.5) [draw = white, fill = white] {$T$};
%\node (z) at (1.5,-1.5) [draw = white, fill = white] {};
\node (a0) at (-1,-0.7) [draw = white, fill = white] {$0$};
\node (a0) at (0,-0.7) [draw = white, fill = white] {$1$};
\node (a0) at (1,-0.7) [draw = white, fill = white] {$2$};
\node (a0) at (2,-0.7) [draw = white, fill = white] {$3$};
\node (a0) at (3,-0.7) [draw = white, fill = white] {$4$};
\node (a0) at (4,-0.7) [draw = white, fill = white] {$5$};
\node (a0) at (5,-0.7) [draw = white, fill = white] {$6$};
\node (a0) at (6,-0.7) [draw = white, fill = white] {$7$};
\node (a0) at (7,-0.7) [draw = white, fill = white] {$8$};
\node (a0) at (8,-0.7) [draw = white, fill = white] {$9$};
\node (a0) at (9,-0.7) [draw = white, fill = white] {$10$};
\node (a0) at (10,-0.7) [draw = white, fill = white] {$11$};
\node (a0) at (11,-0.7) [draw = white, fill = white] {$12$};
\node (a0) at (12,-0.7) [draw = white, fill = white] {$13$};

\node (n3) at (-2.5,0.85) [draw = white, fill = white] {};
\node (n2) at (-1.5,0.85) [draw = white, fill = white] {};
\node (n1) at (-0.5,0.85) [draw = white, fill = white] {};
\node (a0) at (-1,0) [draw = black] {};
\node (a1) at (0,0) [draw = black] {};
\node (a2) at (1,0) [draw = black] {};
\node (a3) at (2,0) [draw = black] {};
\node (a4) at (3,0) [draw = black] {};
\node (a5) at (4,0) [draw = black] {};
\node (a6) at (5,0) [draw = black] {};
\node (a7) at (6,0) [draw = black] {};
\node (a8) at (7,0) [draw = black] {};
\node (a9) at (8,0) [draw = black] {};
\node (a10) at (9,0) [draw = black] {};
\node (a11) at (10,0) [draw = black] {};
\node (a12) at (11,0) [draw = black] {};
\node (a13) at (12,0) [draw = black] {};
\node (a14) at (11.5,0.85) [draw = white, fill = white] {};
\node (a15) at (12.5,0.85) [draw = white, fill = white] {};
\node (a16) at (13.5,0.850) [draw = white, fill = white] {};
%\node (a17) at (16,0) [draw = black] {};
%\node (a18) at (17,0) [draw = black] {};
%\node (a19) at (18,0) [draw = black] {};
%\node (a20) at (19,0) [draw = black] {};
%\node (a21) at (20,0) [draw = black] {};
%\node (a22) at (21,0) [draw = black] {};
%\node (a23) at (22,0) [draw = black] {};
%\draw (a1) to  [out=80,in=100] (a23) ;

\draw  [line width = 0.5mm] (a0) to  [out=60,in=120] (a3) ;
\draw [line width = 0.5mm] (a1) to  [out=60,in=120] (a4) ;
\draw [line width = 0.5mm] (a2) to  [out=60,in=120] (a5) ;
\draw  (a6) to  [out=60,in=120] (a9) ;
\draw [line width = 0.5mm] (a7) to  [out=60,in=120] (a10) ;
\draw [line width = 0.5mm] (a8) to  [out=60,in=120] (a11) ;

\draw  (a9) to  [out=60,in=120] (a12) ;

\draw (a3) to  [out=60,in=120] (a6) ;
\draw  (a4) to  [out=60,in=120] (a7) ;
\draw  (a5) to  [out=60,in=120] (a8) ;
\draw  [line width = 0.5mm] (a6) to  [out=60,in=120] (a9) ;
\draw  (a10) to  [out=60,in=120] (a13) ;
\draw  (a11) to  [out=60,in=180] (a14) ;
\draw  (a12) to  [out=60,in=180] (a15) ;
\draw  (a13) to  [out=60,in=180] (a16) ;
\draw  (n1) to  [out=0,in=120] (a2) ;
\draw  (n2) to  [out=0,in=120] (a1) ;
\draw  (n3) to  [out=0,in=120] (a0) ;

%\foreach \from/\to in {a0/a1,a1/a2,a2/a3,
%a3/a4,a4/a5,a5/a6,a6/a7,a7/a8,a8/a9,a9/a10,a10/a11,a11/a12,a12/a13,a13/a14,a14/a15,a15/a16,a16/a17,a17/a18,a18/a19,a19/a20,a21/a20,a21/a22,a22/a23}
%    \draw (\from) -- (\to);
\end{tikzpicture} 
\begin{tikzpicture}
  [scale=1,auto=left,every node/.style={circle,fill=gray!30},minimum size = 6pt,inner sep=0pt]
 %[scale=2.2,auto=left,every node/.style={circle,fill=gray!30,minimum size = 6pt,inner sep=0pt}]
%\node (z) at (0,-0.5) [draw = white, fill = white] {$T$};
%\node (z) at (1.5,-1.5) [draw = white, fill = white] {};
\node (a0) at (-1,-0.7) [draw = white, fill = white] {$0$};
\node (a0) at (0,-0.7) [draw = white, fill = white] {$1$};
\node (a0) at (1,-0.7) [draw = white, fill = white] {$2$};
\node (a0) at (2,-0.7) [draw = white, fill = white] {$3$};
\node (a0) at (3,-0.7) [draw = white, fill = white] {$4$};
\node (a0) at (4,-0.7) [draw = white, fill = white] {$5$};
\node (a0) at (5,-0.7) [draw = white, fill = white] {$6$};
\node (a0) at (6,-0.7) [draw = white, fill = white] {$7$};
\node (a0) at (7,-0.7) [draw = white, fill = white] {$8$};
\node (a0) at (8,-0.7) [draw = white, fill = white] {$9$};
\node (a0) at (9,-0.7) [draw = white, fill = white] {$10$};
\node (a0) at (10,-0.7) [draw = white, fill = white] {$11$};
\node (a0) at (11,-0.7) [draw = white, fill = white] {$12$};
\node (a0) at (12,-0.7) [draw = white, fill = white] {$13$};

\node (n3) at (-2.5,0.85) [draw = white, fill = white] {};
\node (n2) at (-1.5,0.85) [draw = white, fill = white] {};
\node (n1) at (-0.5,0.85) [draw = white, fill = white] {};
\node (a0) at (-1,0) [draw = black] {};
\node (a1) at (0,0) [draw = black] {};
\node (a2) at (1,0) [draw = black] {};
\node (a3) at (2,0) [draw = black] {};
\node (a4) at (3,0) [draw = black] {};
\node (a5) at (4,0) [draw = black] {};
\node (a6) at (5,0) [draw = black] {};
\node (a7) at (6,0) [draw = black] {};
\node (a8) at (7,0) [draw = black] {};
\node (a9) at (8,0) [draw = black] {};
\node (a10) at (9,0) [draw = black] {};
\node (a11) at (10,0) [draw = black] {};
\node (a12) at (11,0) [draw = black] {};
\node (a13) at (12,0) [draw = black] {};
\node (a14) at (11.5,0.85) [draw = white, fill = white] {};
\node (a15) at (12.5,0.85) [draw = white, fill = white] {};
\node (a16) at (13.5,0.850) [draw = white, fill = white] {};
%\node (a17) at (16,0) [draw = black] {};
%\node (a18) at (17,0) [draw = black] {};
%\node (a19) at (18,0) [draw = black] {};
%\node (a20) at (19,0) [draw = black] {};
%\node (a21) at (20,0) [draw = black] {};
%\node (a22) at (21,0) [draw = black] {};
%\node (a23) at (22,0) [draw = black] {};
%\draw (a1) to  [out=80,in=100] (a23) ;

\draw  (a0) to  [out=60,in=120] (a3) ;
\draw [line width = 0.5mm] (a1) to  [out=60,in=120] (a4) ;
\draw [line width = 0.5mm] (a2) to  [out=60,in=120] (a5) ;
\draw  (a6) to  [out=60,in=120] (a9) ;
\draw [line width = 0.5mm] (a7) to  [out=60,in=120] (a10) ;
\draw [line width = 0.5mm] (a8) to  [out=60,in=120] (a11) ;

\draw [line width = 0.5mm] (a3) to  [out=60,in=120] (a6) ;
\draw  (a4) to  [out=60,in=120] (a7) ;
\draw  (a5) to  [out=60,in=120] (a8) ;
\draw  [line width = 0.5mm] (a9) to  [out=60,in=120] (a12) ;
\draw  (a10) to  [out=60,in=120] (a13) ;
\draw  (a11) to  [out=60,in=180] (a14) ;
\draw  (a12) to  [out=60,in=180] (a15) ;
\draw  (a13) to  [out=60,in=180] (a16) ;
\draw  (n1) to  [out=0,in=120] (a2) ;
\draw  (n2) to  [out=0,in=120] (a1) ;
\draw  (n3) to  [out=0,in=120] (a0) ;

%\foreach \from/\to in {a0/a1,a1/a2,a2/a3,
%a3/a4,a4/a5,a5/a6,a6/a7,a7/a8,a8/a9,a9/a10,a10/a11,a11/a12,a12/a13,a13/a14,a14/a15,a15/a16,a16/a17,a17/a18,a18/a19,a19/a20,a21/a20,a21/a22,a22/a23}
%    \draw (\from) -- (\to);
\end{tikzpicture} 
\caption{The figures show part of a fiber $v^{\sigma}$ of a vertex $v$ with $\varphi(v) = 3$ which belongs to the lift of a voltage graph over a large cyclic group. In the upper figure, the bolded edges make up the edges in the matching $M_v(12)$. In the lower figure, the bolded edges make up the edges in the matching $M_v(12)+1$.}
\label{figMatching}
\end{figure}

\begin{observation}
\label{obsDJ}
For each positive integer $N$, when $p$ is sufficiently large, the matchings $M_v(N)$ and $M_v(N) + \sigma(v)$ are edge-disjoint.
\end{observation}
\begin{observation}
$M_v(2N) = M_v(N) \cup( M_v(N) + N)$.
\end{observation}

%{\LS These are nice observations, indeed!}

%In our proof of Theorems \ref{thmLeaf} and \ref{thm1}, w

In the following lemma, we show that for any vertex $v \in V(T)$, $T^{\sigma}$ has a $2$-factor $H$ 
each of whose components contains at least one edge in a certain local structure of the fiber $v^{\sigma}$. 
This will be useful when we attempt to ``attach" components of a $2$-factor together into a Hamiltonian cycle, as we will know that an edge of every component can be found in a specific part of $T^{\sigma}$.

\begin{lemma}
\label{lemma:twoFactor}
%Let $T$ be a reflexive tree, and let $\phi$ be an assignment of integers to $E(T)$. Then 
Suppose that $p$ is a sufficiently large prime in terms of $T$ and $\phi$. Then, 
there exists a positive integer $N=N(\phi,T)$ such that
for each vertex $v \in V(T)$,
%and for all sufficiently large $p$, 
$T^{\sigma}$ contains a $2$-factor $H$
satisfying $M_v(N) \subseteq E(H)$, and such that each component of $H$ has an edge in the matching $M_v(N)$. Furthermore, $H$ may be chosen such that $H$ contains every edge of $v^{\sigma}[N,p-1]$. 
\end{lemma}
\begin{proof}
We prove the lemma by induction on $|V(T)|$. When $|V(T)| = 1$, $T^{\sigma}$ is a cycle, so by letting $H = T^{\sigma}$, the lemma holds for $N = 2 \varphi(v)$.

Now, suppose $|V(T)| > 1$. Let $v \in V(T)$, and 
%As we may apply an automorphism to any $2$-factor in $T^{\sigma}$, it will be enough to prove the lemma for a ``shift" of the matching described in the lemma's statement. 
%Recall that by Lemma \ref{lemZero}, we assume that $\phi(e) = 0$ for each cut-edge $e \in E(T)$. 
let $v$ have neighbors $u_0, \dots, u_{r-1}$ in $T$. We let $T \setminus \{v\}$ have components $T_0, \dots, T_{r-1}$, so that $u_i \in V(T_i)$ for each $i \in \{ 0, \dots, r-1 \}$. 
By the induction hypothesis, for each $i \in \{0, \dots, r-1\}$,
we may choose an integer $N_i = N_i(\phi_{|T_i}, T_i)$ and a $2$-factor $F_i$ on $T_i$ such that the matching $M_{u_i}(N_i)$ is a subset of $E(F_i)$ and contains an edge of each component of $F_i$, 
and such that $F_i$ contains every edge of $u_i^{\sigma}[N_i, p-1]$.
Then, we define
\[N' = 2\varphi(v)N_1 N_2 \dots N_r,\] 
and by applying automorphisms to each $F_i$, it follows that for each $i \in \{0, \dots, r-1\}$, we may find a $2$-factor $H_i$ on $T_i$ such that the matching $M_{u_i}(N_i) + iN' + \sigma(v)$ is a subset of $E(H_i)$ and contains an edge of each component of $H_i$. Furthermore, since $F_i$ contains every edge of $u_i^{\sigma}[N_i, p-1]$, and since $p$ is sufficiently large, we can also assume that $M_{u_i}(N') + iN' + \sigma(v) \subseteq E(H_i)$.

%we may find
%integers $N_
%$2$-factors $F_0, \dots, F_{r-1}$ on $T_0, \dots, T_{r-1}$, respectively, such that for each $i \in \{0, \dots, r-1\}$, the matching $M_{u_i}(N_i) + iN' + \sigma(v)$ contains an edge of each component of $F_i$, for some large enough $N_i$.

%By the induction hypothesis and by applying automorphisms, we may find $2$-factors $F_0, \dots, F_{r-1}$ on $T_0, \dots, T_{r-1}$, respectively, such that for each $i \in \{0, \dots, r-1\}$, the matching $M_{u_i}(N_i) + iN' + \sigma(v)$ contains an edge of each component of $F_i$, for some large enough $N_i$. 
%By applying automorphisms on $T_0, \dots, T_{r-1}$, we may assume that each component of each $F_i$ has an edge in the matching $M_{u_i}(N) + iN$.

Now, we consider the graph $H \subseteq  T^{\sigma} $ consisting of the edges of the cycle $v^{\sigma}$ as well as $E(H_0) \cup \dots \cup E(H_{r-1})$. As $H$ is a union of the $2$-factors $H_0, \dots, H_{r-1}$ as well as the cycle $v^{\sigma}$, $H$ is $2$-regular. We
modify the graph $H$ as follows.
For each $i \in \{0, \dots, r-1\}$, we add the edges 
\[[(v,\sigma(v) + iN' + a), (u_i, \sigma(v) + iN' + a)]\] 
for each value $a \in \{ 0, \dots, N'-1\}$. Then, for each $i \in \{0, \dots, r-1\}$, we remove the matching $M_{u_i}(N') + iN' + \sigma(v)$ from $E(H)$, and we 
also remove the matching $M_v(rN') + \sigma(v)$ from $E(H)$. This leaves us with a $2$-factor $H$ in $T^{\sigma}$.
We show an example of the construction of $H$ in Figure \ref{figGrowingTree}. In the figure, $r = 2$, $\varphi(u_1) = 1$, $\varphi(u_2) = \varphi(v) = 2$, and $N' = 8$.

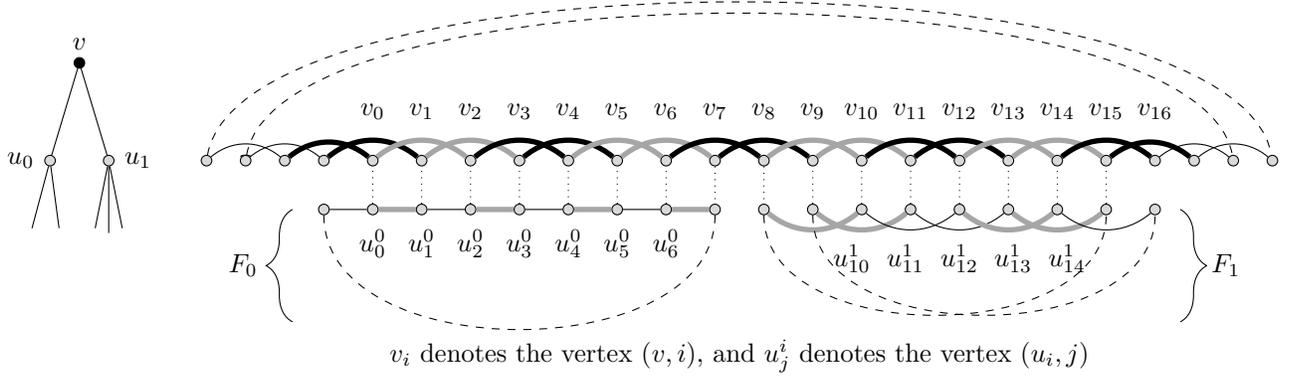
\begin{figure}
\begin{center}
\begin{tikzpicture}
  [scale=0.65,auto=left,every node/.style={circle,fill=gray!30},minimum size = 4pt,inner sep=0pt]
  \clip (-4.5,-4.5) rectangle + (33,7.8);
 %[scale=2.2,auto=left,every node/.style={circle,fill=gray!30,minimum size = 6pt,inner sep=0pt}]
%\node (z) at (0,-0.5) [draw = white, fill = white] {$T$};
%\node (z) at (1.5,-1.5) [draw = white, fill = white] {};

%%%%%TREEEEEEEE
\node (z) at (-3,2.4) [draw = white, fill = white] {$v$};

\node (z) at (-4.2,0) [draw = white, fill = white] {$u_0$};
\node (z) at (-1.8,0) [draw = white, fill = white] {$u_1$};

%\node (z) at (1.5,-1.5) [draw = white, fill = white] {};
%\node (a1) at (0,0) [draw = black] {};
%\node (a2) at (1,0) [draw = black] {};
\node (a3) at (-3.4,-1.5) [draw = white, fill = white] {};
\node (a4) at (-3.6,0) [draw = black] {};
\node (a5) at (-3,2) [draw = black, fill = black] {};
\node (a6) at (-4,-1.5) [draw = white, fill = white] {};
\node (a8) at (-2.4,0) [draw = black] {};
%\node (a9) at (5,0.5) [draw = black] {};
\node (a10) at (-2.4,-1.6) [draw = white,fill=white] {};
\node (c1) at (-2.7,-1.5) [draw = white,fill=white] {};
\node (c2) at (-2.1,-1.5) [draw = white,fill=white] {};

%\node (b1) at (3.2,-1.15) %[draw = black] {};
%\node (b2) at (3,-0.8) [draw = black] {};
%\node (b3) at (3.4,-1.5) [draw = black] {};
    %b1/b3,b1/b4,b3/b4,b3/b2,b4/b2
    %,d2/b1,b2/f2
\foreach \from/\to in {
%a1/a2,a2/a3,
a3/a4,a4/a5,a4/a6,a5/a8,a8/a10,a8/c1,a8/c2}
    \draw (\from) -- (\to);
%%%%END TREEEEE

\node (z) at (10.5,-4) [draw = white, fill = white] {$v_i$ denotes the vertex $(v,i)$, and $u^i_j$ denotes the vertex $(u_i, j)$};

\node (z) at (3,1) [draw = white, fill = white] {$v_0$};
\node (z) at (4,1) [draw = white, fill = white] {$v_1$};
\node (z) at (5,1) [draw = white, fill = white] {$v_2$};
\node (z) at (6,1) [draw = white, fill = white] {$v_3$};
\node (z) at (7,1) [draw = white, fill = white] {$v_4$};
\node (z) at (8,1) [draw = white, fill = white] {$v_5$};
\node (z) at (9,1) [draw = white, fill = white] {$v_6$};
\node (z) at (10,1) [draw = white, fill = white] {$v_7$};
\node (z) at (11,1) [draw = white, fill = white] {$v_8$};
\node (z) at (12,1) [draw = white, fill = white] {$v_9$};
\node (z) at (13,1) [draw = white, fill = white] {$v_{10}$};
\node (z) at (14,1) [draw = white, fill = white] {$v_{11}$};
\node (z) at (15,1) [draw = white, fill = white] {$v_{12}$};
\node (z) at (16,1) [draw = white, fill = white] {$v_{13}$};
\node (z) at (17,1) [draw = white, fill = white] {$v_{14}$};
\node (z) at (18,1) [draw = white, fill = white] {$v_{15}$};
\node (z) at (19,1) [draw = white, fill = white] {$v_{16}$};

\node (z) at (3,-1.7) [draw = white, fill = white] {$u^0_{0}$};
\node (z) at (4,-1.7) [draw = white, fill = white] {$u^0_{1}$};
\node (z) at (5,-1.7) [draw = white, fill = white] {$u^0_{2}$};
\node (z) at (6,-1.7) [draw = white, fill = white] {$u^0_{3}$};
\node (z) at (7,-1.7) [draw = white, fill = white] {$u^0_{4}$};
\node (z) at (8,-1.7) [draw = white, fill = white] {$u^0_{5}$};
\node (z) at (9,-1.7) [draw = white, fill = white] {$u^0_{6}$};

\node (z) at (12.8,-2) [draw = white, fill = white] {$u^1_{10}$};
\node (z) at (13.9,-2) [draw = white, fill = white] {$u^1_{11}$};
\node (z) at (15,-2) [draw = white, fill = white] {$u^1_{12}$};
\node (z) at (16.1,-2) [draw = white, fill = white] {$u^1_{13}$};
\node (z) at (17.2,-2) [draw = white, fill = white] {$u^1_{14}$};

%\node (am1) at (-1.2,0) [draw = black] {};
\node (a0) at (-0.4,0) [draw = black] {};
\node (a1) at (0.4,0) [draw = black] {};
\node (a2) at (1.2,0) [draw = black] {};
\node (a3) at (2,0) [draw = black] {};
\node (a4) at (3,0) [draw = black] {};
\node (a5) at (4,0) [draw = black] {};
\node (a6) at (5,0) [draw = black] {};
\node (a7) at (6,0) [draw = black] {};
\node (a8) at (7,0) [draw = black] {};
\node (a9) at (8,0) [draw = black] {};
\node (a10) at (9,0) [draw = black] {};
\node (a11) at (10,0) [draw = black] {};
\node (a12) at (11,0) [draw = black] {};
\node (a13) at (12,0) [draw = black] {};
\node (a14) at (13,0) [draw = black] {};
\node (a15) at (14,0) [draw = black] {};
\node (a16) at (15,0) [draw = black] {};
\node (a17) at (16,0) [draw = black] {};
\node (a18) at (17,0) [draw = black] {};
\node (a19) at (18,0) [draw = black] {};
\node (a20) at (19,0) [draw = black] {};
\node (a21) at (19.8,0) [draw = black] {};
\node (a22) at (20.6,0) [draw = black] {};
\node (a23) at (21.4,0) [draw = black] {};

\draw [dashed] (a1) to  [out=80,in=100,looseness = 0.5] (a22) ;
\draw [dashed] (a0) to  [out=80,in=100,looseness = 0.5] (a23) ;
\draw (a1) to  [out=40,in=140] (a1) ;
\draw (a0) to  [out=40,in=140] (a2) ;
\draw (a1) to  [out=40,in=140] (a3) ;
\draw [line width = 2pt, black](a2) to  [out=40,in=140] (a4) ;
\draw [line width = 2pt, black](a3) to  [out=40,in=140] (a5) ;
\draw [line width = 2pt, gray!70!white](a4) to  [out=40,in=140] (a6) ;
\draw [line width = 2pt, gray!70!white](a5) to  [out=40,in=140] (a7) ;
\draw [line width = 2pt, black](a6) to  [out=40,in=140] (a8) ;
\draw [line width = 2pt, black](a7) to  [out=40,in=140] (a9) ;
\draw [line width = 2pt, gray!70!white](a8) to  [out=40,in=140] (a10) ;
\draw [line width = 2pt, gray!70!white](a9) to  [out=40,in=140] (a11) ;
\draw [line width = 2pt, black](a10) to  [out=40,in=140] (a12) ;
\draw [line width = 2pt, black](a11) to  [out=40,in=140] (a13) ;
\draw [line width = 2pt, gray!70!white](a12) to  [out=40,in=140] (a14) ;
\draw [line width = 2pt, gray!70!white](a13) to  [out=40,in=140] (a15) ;
\draw [line width = 2pt, black](a14) to  [out=40,in=140] (a16) ;
\draw [line width = 2pt, black](a15) to  [out=40,in=140] (a17) ;
\draw [line width = 2pt, gray!70!white](a16) to  [out=40,in=140] (a18) ;
\draw [line width = 2pt, gray!70!white](a17) to  [out=40,in=140] (a19) ;
\draw [line width = 2pt, black](a18) to  [out=40,in=140] (a20) ;
\draw [line width = 2pt, black](a19) to  [out=40,in=140] (a21) ;
\draw (a20) to  [out=40,in=140] (a22) ;
 \draw (a21) to  [out=40,in=140] (a23) ;

\node (b3) at (2,-1) [draw = black] {};
\node (b4) at (3,-1) [draw = black] {};
\node (b5) at (4,-1) [draw = black] {};
\node (b6) at (5,0-1) [draw = black] {};
\node (b7) at (6,0-1) [draw = black] {};
\node (b8) at (7,0-1) [draw = black] {};
\node (b9) at (8,0-1) [draw = black] {};
\node (b10) at (9,0-1) [draw = black] {};
\node (b11) at (10,0-1) [draw = black] {};
\node (b12) at (11,0-1) [draw = black] {};
\node (b13) at (12,0-1) [draw = black] {};
\node (b14) at (13,0-1) [draw = black] {};
\node (b15) at (14,0-1) [draw = black] {};
\node (b16) at (15,0-1) [draw = black] {};
\node (b17) at (16,-1) [draw = black] {};
\node (b18) at (17,-1) [draw = black] {};
\node (b19) at (18,-1) [draw = black] {};
 \node (b20) at (19,-1) [draw = black] {};
 \foreach \from/\to in {b4/a4,b5/a5,b6/a6,b7/a7,b8/a8,b9/a9,b10/a10,a11/b11,a12/b12,b13/a13,b14/a14,b15/a15,b16/a16,b17/a17,b18/a18,b19/a19}
    \draw[dotted] (\from) -- (\to);

     \foreach \from/\to in {b3/b4,b4/b5,b5/b6,b6/b7,b7/b8,b8/b9,b9/b10,b10/b11}
    \draw (\from) -- (\to);
 
 \draw [line width = 2pt, gray!70!white] (b12) to  [out=320,in=220] (b14) ;
 \draw  [line width = 2pt, gray!70!white] (b13) to  [out=320,in=220] (b15) ;
 \draw  (b14) to  [out=320,in=220] (b16) ;
 \draw  (b15) to  [out=320,in=220] (b17) ;
 \draw [line width = 2pt, gray!70!white](b16) to  [out=320,in=220] (b18) ;
 \draw [line width = 2pt, gray!70!white](b17) to  [out=320,in=220] (b19) ;
 \draw  (b18) to  [out=320,in=220] (b20) ;
  \draw [line width = 2pt, gray!70!white] (b4) to  (b5) ;
  \draw [line width = 2pt, gray!70!white] (b6) to  (b7) ;
  \draw [line width = 2pt, gray!70!white] (b8) to  (b9) ;
  \draw [line width = 2pt, gray!70!white] (b10) to  (b11) ;
 % \draw [line width = 2pt, black] (b5) to  (b6) ;
  %\draw [line width = 2pt, black] (b7) to  (b8) ;
  %\draw [line width = 2pt, black] (b9) to  (b10) ;
  %\draw [line width = 2pt, black] (b3) to  (b4) ;
   \draw [dashed] (b3) to  [out=270,in=270] (b11) ;
    \draw [dashed] (b12) to  [out=270,in=270] (b19) ;
        \draw [dashed] (b13) to  [out=270,in=270] (b20) ;

        \draw[decorate,decoration={brace,amplitude=10pt},xshift=-4pt,yshift=0pt]
(1.5,-3.3) -- (1.5,-1) node [white,midway,xshift=-0.6cm] 
{};
   \node (z) at (0.35,-2.15) [draw = white, fill = white] {$F_0$};
\draw [decorate,decoration={brace,amplitude=10pt,mirror,raise=4pt},yshift=0pt]
(19.3,-3.3) -- (19.3,-1) node [white,midway,xshift=0.8cm] {};
           \node (z) at (20.45,-2.15) [draw = white, fill = white] {$F_1$};
        
\end{tikzpicture}
\end{center}
\caption{The first figure shows part of a reflexive tree $T$ with a voltage assignment $\sigma$ and a specified vertex $v$ which has two neighbors $u_0$ and $u_1$. 
The second figure shows parts of the fibers $v^{\sigma}$, $u_0^{\sigma}$, and $u_1^{\sigma}$, with the dashed lines representing the undrawn parts of the fibers. 
We may seek a $2$-factor $F$ each of whose components has an edge in a certain matching in $v^{\sigma}$ as follows.
We find $2$-factors $F_0$ and $F_1$ in the two components of $T \setminus \{v\}$, which are shown as cycles. 
We then may remove a matching from each of $F_0$ and $F_1$, as well as from $v^{\sigma}$ and add the dotted edges between $v^{\sigma}$ and the other fibers. The edges that we remove are shown in gray. 
We then obtain a larger $2$-factor $H$, each of whose components contain at least one of the edges in the bolded matching shown in the second figure.}
\label{figGrowingTree}
\end{figure}

We wish to show that there exists some integer $N$ for which $N$ and $H$ satisfy the conditions of the lemma.
%some matching $M$ in $v^{\sigma}$ of the form $M_v(N)$ for some integer $N$ that belongs to $E(H)$ and contains an edge of every component of $H$. 
We will show that $N = rN' + 2\varphi(v)$ is a suitable choice in the next two claims, and then the proof will be complete.
%Although $M$ is not exactly of the form $M_v(N')$ for an integer $N'$, we may adjust $M$ to be in this form by applying the automorphism $(v,g) \mapsto (v,g+\sigma(v))$ to $T^{\sigma}$. Therefore, 
%The next two claims will be devoted to showing that $M$ belongs to $E(H)$ and contains an edge of every component of $H$.

\begin{claim}
\label{claim:np1}
$ M_v(N) \subseteq E(H)$. Furthermore, $H$ contains every edge of $v^{\sigma}[N,p-1]$.
\end{claim}
\begin{proof}[Proof of claim:]
By construction, $E(H)$ contains every edge of $v^{\sigma}$ except for those contained in $M_v(rN') + \sigma(v)$, which immediately proves the second statement. For the first statement, it suffices to show that $ M_v(N)$ is disjoint from $M_v(rN') + \sigma(v)$. As 
$$M_v(rN')+ \sigma(v) \subseteq  M_v(rN' + 2\varphi(v))+ \sigma(v) = M_v(N) + \sigma(v) $$
it suffices to show that $ M_v(N)$ is disjoint from $ M_v(N)+ \sigma(v)$. However, this statement follows directly from Observation \ref{obsDJ}. Hence, we see that when we remove $M_v(rN')+ \sigma(v)$ from $v^{\sigma}$, we do not remove any edges from $ M_v(N) $. Therefore, each edge of $ M_v(N)  $ belongs to $E(H)$. \end{proof}

\begin{claim}
\label{claim:mvn}
Each component of $H$ has an edge in $ M_v(N)$. 
\end{claim}
\begin{proof}[Proof of claim:]
By construction, $E(H) \cap E(v^{\sigma}) = E(v^{\sigma}) \setminus (M_v(rN') + \sigma(v)). $
We first observe that the group values of the endpoints of $ M_v(rN') + \sigma(v)  $ make up the subset of $\mathbb Z_p$ 
$$S = \{\sigma(v), \dots, \sigma(v) +  rN' - 1 \}.$$ 
Furthermore, $S$ is a subset of the group values of the vertex set $V(M_v(N))$. Therefore, as $H$ is $2$-regular, any component in $H$ that contains a vertex $(v,b)$ for $b \in S$ also has an edge in $M_v(N) $.

Now, we show that each component of $H$ has at least one edge in the 
%fiber $v^{\sigma}$. 
matching $M_v(N)$.
We note that $H \cap v^{\sigma}$ is a forest of paths, as is $H \cap H_i$ for each $2$-factor $H_i$ obtained from the induction hypothesis. Therefore, since $H$ is $2$-regular, every component $C$ of $H$ contains both a vertex from $v^{\sigma}$ and a vertex from a $2$-factor $H_i$. Hence, it follows that $C$ contains an edge of the form 
\[[(v,\sigma(v) + iN' + a), (u_i, \sigma(v) + iN' + a)]\] 
for some value $i \in \{0, \dots, r-1\}$ and some value $a \in \{0,\dots,N'-1\}$. Since $b = \sigma(v) + iN' + a \in S$, it then follows from the argument above that $C$ contains an edge in $M_v(N)$. This completes the proof of the claim. 
%Indeed, suppose some component $C$ of $H$ does not have an edge in $v^{\sigma}$. 
%Since $H$ is $2$-regular, $C$ must be a cycle, and it must follow that $E(C) \subseteq E(H_i)$ for some $i \in \{0, \dots, r-1\}$,  as $H \cap v^{\sigma}$ is a forest of paths. However, by construction, every component of $H_i$ contains an edge in the matching $M_{u_i}(N') + iN' + \sigma(v)$, which was removed while constructing $H$. Therefore, $H \cap H_i$ is a forest of paths, which contradicts the assumption that $E(C) \subseteq E(H_i)$. Hence, we conclude that each component of $H$ contains at least one edge in $v^{\sigma}$. This completes the proof of Lemma \ref{lemma:twoFactor}.
\end{proof}
Claims \ref{claim:np1} and \ref{claim:mvn} imply the lemma statement, completing the proof of the lemma.
\end{proof}

Now we are ready to prove Theorem \ref{thmLeaf}. 
\begin{proof}[Proof of Theorem \ref{thmLeaf}]
Let $T$ be a reflexive tree, and let $v$ be a vertex of $T$ with $r$ neighbors $u_0, \dots, u_{r-1}$. Furthermore, for each neighbor $u_i$ of $v$, let $\phi(u_i) = \phi(v)$. For each $i \in \{0, \dots, r-1\}$, let $T_i$ be the component of $T \setminus \{v\}$ containing $u_i$. By Lemma \ref{lemma:twoFactor}, for each $i \in \{0, \dots, r-1\}$, we may find a $2$-factor $F_i$ of $T_i^{\sigma}$, along with an integer $N_i = N_i(\phi_{|T_i}, T_i)$, such that each component of $F_i$ has at least one edge in $M_{u_i}(N_i)$ and such that each edge induced by $u_i^{\sigma}[N_i, p-1]$ belongs to $E(F_i)$. 
We let 
\[N = 2\varphi(v)N_1 N_2 \dots N_r.\]
By applying automorphisms on $T_0, \dots, T_{r-1}$, we may then obtain for each $i \in \{0, \dots, r-1\}$ a $2$-factor $H_i$, each of whose components contains at least one edge in the matching $M_{u_i}(N) + iN$.
%In particular, the group elements of the endpoints of these matchings are disjoint.

Now, for each $i \in \{0, \dots, r-1\}$ and each component $C$ of $F_i$, we choose an edge $$[(u_i, a),(u_i,a + \sigma(v))] \in M_{u_i}(N) + iN$$ of $C$, by choosing an appropriate $a$.
(Note that this is possible, since $\sigma(u_i) = \sigma(v)$.)
Then, we remove the edges $[(u_i, a),(u_i,a + \sigma(v))]$ and $[(v, a),(v,a + \sigma(v))]$ from $C$ and $v^{\sigma}$, respectively, and add the edges $[(u_i,a),(v_i,a)]$ and $[(u_i, a+ \sigma(v)),(v,a + \sigma(v))]$. %We illustrate this process in Figure \ref{figExtend}. 
By following this process for each $i$, we begin with a cycle $A = v^{\sigma}$, and each time we replace edges for a component $C$ as described above, we extend $A$ to include every vertex of $C$. By repeating this process for every component of every $2$-factor $H_i$, $A$ is extended to a cycle that visits every vertex of $T^{\sigma}$. Therefore, $T^{\sigma}$ is Hamiltonian. Finally, note that a sufficient bound on $p$ is  $p > 2^{|V(T)|} \prod_{v \in V(T)} \phi(v)$.
 \end{proof}
 
 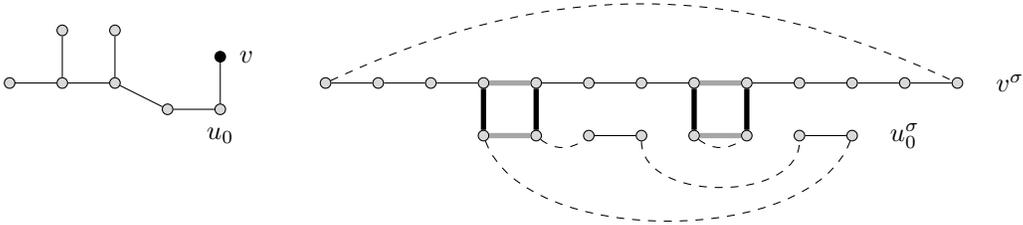
\begin{figure}
 \begin{center}
\begin{tikzpicture}
    [scale=0.7,auto=left,every node/.style={circle,fill=gray!30},minimum size = 4pt,inner sep=0pt]
 \clip (1,-2.65) rectangle + (21.5,4.2);

%%%%TREEEEEEEE
\node (z) at (6.5,0.5) [draw = white, fill = white] {$v$};
\node (z) at (6,-1) [draw = white, fill = white] {$u_0$};

%\node (z) at (1.5,-1.5) [draw = white, fill = white] {};
%\node (a1) at (0,0) [draw = black] {};
%\node (a2) at (1,0) [draw = black] {};
\node (f3) at (2,0) [draw = black] {};
\node (f4) at (3,0) [draw = black] {};
\node (f5) at (4,0) [draw = black] {};
\node (f6) at (4,1) [draw = black] {};
\node (f7) at (3,1) [draw = black] {};
%\node (f9) at (4,0.5) [draw = black] {};
\node (f10) at (5,-0.5) [draw = black] {};
\node (h1) at (6,-0.5) [draw = black] {};
\node (h2) at (6,0.5) [draw = black, fill = black] {};

\foreach \from/\to in {
%a1/a2,a2/a3,
f3/f4,f4/f5,f5/f6,f4/f7,f5/f10,f10/h1,h1/h2}
    \draw (\from) -- (\to);
%%%%END TREE

\node (z) at (21,0) [draw = white, fill = white] {$v^{\sigma}$};
\node (z) at (19,-1) [draw = white, fill = white] {$u_0^{\sigma}$};
\node (z) at (21,0) [draw = white, fill = white] {$v^{\sigma}$};
%\node (z) at (18.5,-2.4) [draw = white, fill = white] {$T_0^{\sigma}$};

\node (a9) at (8,0) [draw = black] {};
\node (a10) at (9,0) [draw = black] {};
\node (a11) at (10,0) [draw = black] {};
\node (a12) at (11,0) [draw = black] {};
\node (a13) at (12,0) [draw = black] {};
\node (a14) at (13,0) [draw = black] {};
\node (a15) at (14,0) [draw = black] {};
\node (a16) at (15,0) [draw = black] {};
\node (a17) at (16,0) [draw = black] {};
\node (a18) at (17,0) [draw = black] {};
\node (a19) at (18,0) [draw = black] {};
\node (a20) at (19,0) [draw = black] {};
\node (a21) at (20,0) [draw = black] {};

\draw [dashed] (a9) to  [out=25,in=155] (a21) ;

\node (b12) at (11,0-1) [draw = black] {};
\node (b13) at (12,0-1) [draw = black] {};
\node (b14) at (13,0-1) [draw = black] {};
\node (b15) at (14,0-1) [draw = black] {};
\node (b16) at (15,0-1) [draw = black] {};
\node (b17) at (16,-1) [draw = black] {};
\node (b18) at (17,-1) [draw = black] {};
\node (b19) at (18,-1) [draw = black] {};

 \foreach \from/\to in {a9/a10,a10/a11,a11/a12,a12/a13,a13/a14,a14/a15,a15/a16,a16/a17,a17/a18,a18/a19,a19/a20,a20/a21,b12/b13,b14/b15,b16/b17,b18/b19, a12/b12,b13/a13,b16/a16,b17/a17}
    \draw (\from) -- (\to);
 
 \draw [line width = 2pt, gray!70!white] (b12) to   (b13) ;
  \draw [line width = 2pt] (a12) to   (b12) ;
    \draw [line width = 2pt] (a13) to   (b13) ;
    
        \draw [line width = 2pt] (a16) to   (b16) ;
        \draw [line width = 2pt] (a17) to   (b17) ;
        
  \draw [line width = 2pt, gray!70!white] (a12) to   (a13) ;
   \draw [line width = 2pt, gray!70!white] (a16) to   (a17) ;
      \draw [line width = 2pt, gray!70!white] (b16) to   (b17) ;
  \draw [dashed] (b13) to  [out=320,in=220] (b14) ;
   \draw [dashed] (b15) to  [out=270,in=270] (b18) ;
      \draw [dashed] (b12) to  [out=290,in=250,looseness=0.8] (b19) ;
         \draw [dashed] (b16) to  [out=320,in=220] (b17) ;
\end{tikzpicture} 
\end{center}
\caption{
The first figure shows a tree $T$ with a voltage assignment $\sigma$ and a vertex $v$, all of whose neighbors (here, just $u_0$) have the same voltage as $v$. The second figure shows parts of the fibers $v^{\sigma}$ and $u^{\sigma}$ (with their vertices depicted horizontally) and the graph $T_0^{\sigma}$, with the dashed lines representing the undrawn parts of the graphs. 
$T_0^{\sigma}$ has a $2$-factor $F_0$ each of whose components has an edge in a certain local part of the fiber $u_0^{\sigma}$.
Here, we may create a Hamiltonian cycle on $T^{\sigma}$ by removing one edge in $u_0^{\sigma}$ from each cycle of $F_0$ and then replacing these edges with a matching between $u_0^{\sigma}$ and $v^{\sigma}$. Here, the removed edges are shown in gray, and the added edges are shown vertically in bold.
}
\label{fig42}
 \end{figure}

Using Lemma \ref{lemma:twoFactor}, we can also prove Theorem \ref{thm1}.
\begin{proof}[Proof of Theorem \ref{thm1}]
Let $T$ be a reflexive tree, and let $u,v \in V(T)$ be an adjacent pair of vertices for which $\phi(u) = \phi(v) = 1$. We define $T_u \subseteq T$ as the subtree obtained by removing the edge $uv$ and taking the component with $u$. We define $T_v \subseteq T$ in the same way, by taking the component containing $v$. By Lemma \ref{lemma:twoFactor}, there exists a $2$-factor $F_u$ in $T_u^{\sigma}$ in which each component of $F_u$ has an edge in the matching $M_u(N_u)$, for some integer $N_u = N_u(\phi_{|T_u},T_u)$. Similarly, there exists a $2$-factor $F_v$ in $T_v^{\sigma}$ in which each component has an edge in the matching $M_v(N_v)+N_v$, for some integer $ N_v = N_v(\phi_{|T_v},T_v) $. Note that as $\phi(u) = \phi(v) = 1$, the edges induced by $u^{\sigma}[N_u, p-1]$ all belong to a single component $A_u$ of $F_u$, and the edges induced by $v^{\sigma}[N_v + N_u, p-1]$
% \cup v^{\sigma}[0, N_u-1]$ 
all belong to a single component $A_v$ of $F_v$.

%in $u^{\sigma}$ outside of 
%$M_u(N_u)$ 
%belongs to a single cycle $C_u$, and every edge in $v^{\sigma}$ outside of $M_v(N_v)$ belongs to a single cycle $C_v$. 
We now construct a Hamiltonian cycle $H$ on $T^{\sigma}$. We first let $H = F_u \cup F_v$, and then we modify $H$ as follows.
For each component $C$ in $F_u$, we remove from $H$ an edge $[(u,i), (u,i+1)] \in E(C) \cap M_u(N_u)$, as well as the edge $[(v,i), (v,i+1)]$. Then, we add to $H$ the edges $[(u,i), (v,i)]$ and $[(u,i+1), (v,i+1)]$.
%and replace it with the edges $(u_i, v_i)$ and $(u_{i+1},v_{i+1})$. 
In doing so, we effectively ``attach" the component $C$ to $A_v$ as in Figure \ref{fig:my_label_3}. We repeat this process for every component $C$ of $F_u$, and we attach every component of $F_u$, apart from $A_u$, to $A_v$. Similarly, through this process, we attach every component of $F_v$, apart from $A_v$, to $A_u$. This leaves us with two cycles, which we may attach by choosing some value $i$ satisfying $N_u + N_v < i < p-1$, removing from $H$ the edges $[(u,i), (u,{i+1})]$ and $[(v,i), (v,{i+1})]$, and adding to $H$ the edges $[(u,i) , (v,i)]$ and $[ (u,{i+1}), (v,{i+1})]$. This is possible because we can choose $p>N_u+N_v+1$. Through this process, we obtain a single Hamiltonian cycle $H$ on $T^{\sigma}$. This completes the proof. 
\end{proof}
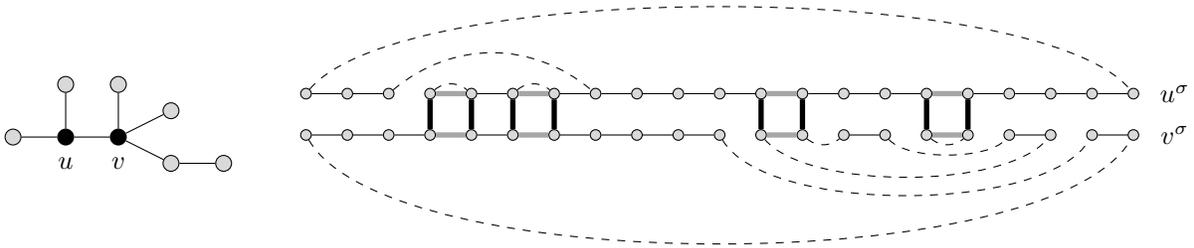
\begin{figure}
\begin{center}
\begin{tikzpicture}
  [scale=0.7,auto=left,every node/.style={circle,fill=gray!30},minimum size = 6pt,inner sep=0pt]
\node (z) at (12,1.5) [draw = white, fill = white] {$v$};
\node (z) at (11,1.5) [draw = white, fill = white] {$u$};

\node (z) at (10,0) [draw = white, fill = white] {};
%\node (a1) at (0,0) [draw = black] {};
%\node (a2) at (1,0) [draw = black] {};
\node (f3) at (10,2) [draw = black] {};
\node (f4) at (11,2) [draw = black, fill = black] {};
\node (f5) at (12,2) [draw = black, fill = black] {};
\node (f6) at (12,3) [draw = black] {};
\node (f7) at (11,3) [draw = black] {};
\node (f9) at (13,2.5) [draw = black] {};
\node (f10) at (13,1.5) [draw = black] {};
\node (h1) at (14,1.5) [draw = black] {};
\node (h2) at (15,1.5) [draw = white, fill = white] {};

\foreach \from/\to in {
%a1/a2,a2/a3,
f3/f4,f4/f5,f5/f6,f4/f7,f5/f9,f5/f10,f10/h1}%,h1/h2}
    \draw (\from) -- (\to);
    
\end{tikzpicture} 
\begin{tikzpicture}
  [scale=0.55,auto=left,every node/.style={circle,fill=gray!30},minimum size = 4pt,inner sep=0pt]
  \clip (-0.3,-3.8) rectangle + (21.8,5.95);

\node (a1) at (0,0) [draw = black] {};
\node (a2) at (1,0) [draw = black] {};
\node (a3) at (2,0) [draw = black] {};
\node (a4) at (3,0) [draw = black] {};
\node (a5) at (4,0) [draw = black] {};
\node (a6) at (5,0) [draw = black] {};
\node (a7) at (6,0) [draw = black] {};
\node (a8) at (7,0) [draw = black] {};
\node (a9) at (8,0) [draw = black] {};
\node (a10) at (9,0) [draw = black] {};
\node (a11) at (10,0) [draw = black] {};
\node (a12) at (11,0) [draw = black] {};
\node (a13) at (12,0) [draw = black] {};
\node (a14) at (13,0) [draw = black] {};
\node (a15) at (14,0) [draw = black] {};
\node (a16) at (15,0) [draw = black] {};
\node (a17) at (16,0) [draw = black] {};
\node (a18) at (17,0) [draw = black] {};
\node (a19) at (18,0) [draw = black] {};
\node (a20) at (19,0) [draw = black] {};
\node (a21) at (20,0) [draw = black] {};

\node (z) at (21,0) [draw = white, fill = white] {$u^{\sigma}$};
\node (z) at (21,-1) [draw = white, fill = white] {$v^{\sigma}$};

\draw [dashed] (a1) to  [out=50,in=130,looseness=0.45] (a21) ;

\node (b1) at (0,0-1) [draw = black] {};
\node (b2) at (1,0-1) [draw = black] {};
\node (b3) at (2,0-1) [draw = black] {};
\node (b4) at (3,0-1) [draw = black] {};
\node (b5) at (4,0-1) [draw = black] {};
\node (b6) at (5,0-1) [draw = black] {};
\node (b7) at (6,-1) [draw = black] {};
\node (b8) at (7,-1) [draw = black] {};
\node (b9) at (8,-1) [draw = black] {};
\node (b10) at (9,0-1) [draw = black] {};
\node (b11) at (10,0-1) [draw = black] {};
\node (b12) at (11,0-1) [draw = black] {};
\node (b13) at (12,0-1) [draw = black] {};
\node (b14) at (13,0-1) [draw = black] {};
\node (b15) at (14,0-1) [draw = black] {};
\node (b16) at (15,0-1) [draw = black] {};
\node (b17) at (16,-1) [draw = black] {};
\node (b18) at (17,-1) [draw = black] {};
\node (b19) at (18,-1) [draw = black] {};
\node (b20) at (19,0-1) [draw = black] {};
\node (b21) at (20,0-1) [draw = black] {};
\draw [dashed] (b1) to  [out=-60,in=-120,looseness = 0.5] (b21) ;

 \foreach \from/\to in {b1/b2,b2/b3,b3/b4,b4/b5,b5/b6,b6/b7,b7/b8,b8/b9,b9/b10,b10/b11,b20/b21,
 a1/a2,a3/a2,a5/a6,a7/a8,a8/a9,a9/a10,a10/a11,a11/a12,a12/a13,a13/a14,a14/a15,a15/a16,a16/a17,a17/a18,a18/a19,a19/a20,a20/a21,b12/b13,b14/b15,b16/b17,b18/b19, a12/b12,b13/a13,b16/a16,b17/a17}
    \draw (\from) -- (\to);
 
   \draw [line width = 2pt] (a4) to   (b4) ;
    \draw [line width = 2pt] (a5) to   (b5) ;
    
    \draw [dashed] (a3) to  [out=40,in=140] (a8) ;
    \draw [dashed] (a4) to  [out=40,in=140] (a5) ;
      % \draw [dashed] (a8) to  [out=40,in=140] (a9) ;
  %  \draw [dashed] (a4) to  [out=40,in=140] (a7) ;
    
    \draw [dashed] (a6) to  [out=40,in=140] (a7) ;
        \draw [line width = 2pt] (a6) to   (b6) ;
        \draw [line width = 2pt] (a7) to   (b7) ;
  \draw [line width = 2pt, gray!70!white] (a5) to   (a4) ;
    \draw [line width = 2pt, gray!70!white] (a6) to   (a7) ;
   \draw [line width = 2pt, gray!70!white] (b5) to   (b4) ;
    \draw [line width = 2pt, gray!70!white] (b6) to   (b7) ;
 
 \draw [line width = 2pt, gray!70!white] (b12) to   (b13) ;
  \draw [line width = 2pt] (a12) to   (b12) ;
    \draw [line width = 2pt] (a13) to   (b13) ;
    
        \draw [line width = 2pt] (a16) to   (b16) ;
        \draw [line width = 2pt] (a17) to   (b17) ;
        
  \draw [line width = 2pt, gray!70!white] (a12) to   (a13) ;
   \draw [line width = 2pt, gray!70!white] (a16) to   (a17) ;
      \draw [line width = 2pt, gray!70!white] (b16) to   (b17) ;
  \draw [dashed] (b13) to  [out=320,in=220] (b14) ;
   \draw [dashed] (b15) to  [out=310,in=230,looseness=0.6] (b18) ;
      \draw [dashed] (b12) to  [out=310,in=230,looseness=0.6] (b19) ;
       \draw [dashed] (b11) to  [out=300,in=240,looseness=0.6] (b20) ;
         \draw [dashed] (b16) to  [out=320,in=220] (b17) ;
\end{tikzpicture} 
\end{center}   
\caption{The first figure shows a tree $T$ with a voltage assignment $\sigma$ and two vertices $u, v$, both of which have voltage $1$. The second figure shows parts of the fibers $u^{\sigma}$, $v^{\sigma}$.%, and the graph $T^{\sigma}$, with the dashed lines representing the undrawn parts of the graph. 
We may seek a $2$-factor $F$ each of whose components has an edge in a certain matching in $v^{\sigma}$ as follows. We find $2$-factors $F_u$ and $F_v$ in the two components of $T \setminus \{uv\}$, which are depicted using cycles in the second figure, with the dashed lines representing the undrawn parts of the graph. We then may remove a matching from each of $F_u$ and $F_v$, as well as from $v^{\sigma}$, and add edges between $u^{\sigma}$ and $v^{\sigma}$. The edges that we remove are shown in gray, and the edges that we add are shown in bold. We then obtain a Hamiltonian cycle on $T^{\sigma}$.}
    \label{fig:my_label_3}
\end{figure}

One
consequence of Theorem \ref{thmLeaf} is that while we cannot answer Question \ref{question:bigPrime} for every tree and every labelling of its edges, we can answer Question \ref{question:bigPrime} when labels are assigned to the loops of a tree uniformly at random from a finite subset $S$ of the positive integers.

\begin{theorem}
\label{thm:prob}
Let $S$ be a finite set of positive integers, and let $k = |S|$.
For each sufficiently large $n$, there exists a value $p_0 = p_0(n,S)$ for which the following holds.
Let $T$ be a
reflexive tree on $n$ vertices, and let $\phi$ assign a value from $S$ to each edge $e \in E(T)$ uniformly at random.
%Let each cut-edge of $T$ is assigned $0$. 
Then, with probability at least $1 - \exp(- \frac {1}{4k} \sqrt{n} )$, $T^{\sigma}$ is Hamiltonian.
\end{theorem}

\begin{proof}
As usual, we use Lemma \ref{lemZero} to consider the equivalent problem in which $\phi(e) = 0$ for every cut-edge $e \in E(T)$.
We show that with probability at least $1 - \exp(- \frac 1{4k} \sqrt n)$, $T$ contains some vertex $v$ such that $\phi(v) = \phi(u)$ for each $u \in N(v)$. Then, Theorem \ref{thmLeaf} will imply that $T^{\sigma}$ is Hamiltonian.

We consider two cases. First, suppose that $T$ has at least $\frac{1}{4}\sqrt{ n }$ leaves. As $n$ is 
sufficiently large, we assume that $T$ has no adjacent pair of leaves.
After assigning labels to each $u \in V(T)$ of degree at least $2$, the conditional probability that a given leaf $\ell$ receives the same label from $\phi$ as its neighbor is $\frac 1k$. Given that we assign labels to leaves independently, the probability that no leaf $\ell$ receives the same label as its neighbor is 
$(1-\frac 1k)^{\frac 14 \sqrt{n}} < 
\exp(- \frac{1}{4k} \sqrt{n})$.

Next, suppose that $T$ has at most  $\frac 14 \sqrt{ n}$ leaves. 
We write $T'$ for the tree obtained from $T$ by using an edge to replace each maximal path with internal vertices of degree $2$. 
In other words, $T'$ is obtained by ``smoothing out" the degree-$2$ vertices of $T$. As $T'$ has no vertex of degree $2$ and $\sum_{v \in V(T')} (\deg_{T'}(v) - 2) = -2$, it follows that the number of vertices of degree at least $3$ in $T'$ is fewer than the number of leaves of $T'$. As $T'$ has at most $\frac 14 \sqrt{k n}$ leaves, it follows that $E(T') < |V(T')| < \frac 12  \sqrt{ n}$. Hence, by the pigeonhole principle, some edge of $T'$ replaced a path of $T$ with at least 
$\frac{n-1 -  \frac 12 \sqrt{ n}}{ \frac 12 \sqrt{ n}} > 3  \left \lceil \frac 13 \sqrt { n } \right  \rceil  + 2$ vertices. 
Therefore, $T$ has a path $P$ with $3  \left \lceil \frac 13 \sqrt { n } \right  \rceil $ vertices such that $\deg_T(v) = 2$ for each $v \in V(P)$. We choose a set $S$ of $\lceil \frac 13 \sqrt n \rceil$ vertices $v \in V(P)$ so that the closed neighborhoods $N[v]$ for $v \in S$ partition $V(P)$. As the closed neighborhoods $N[v]$ are distinct for $v \in S$ and are all size $3$, the probability that no $v \in S$ receives the same label from $\phi$ as both of its neighbors is at most $(1- \frac 1k)^{\left \lceil \frac 13 \sqrt n \right  \rceil } < \exp(-\frac {1}{4k} \sqrt{n})$.

Thus, in both cases, it holds with probability at least $1 - \exp(-\frac{1}{4k} \sqrt n)$ that for some vertex $v \in V(T)$, $\phi(v) = \phi(u)$ for each $u \in N(v)$. Then, Theorem \ref{thmLeaf} implies that for some $p_0 = p_0(n,S)$, $T^{\sigma}$ is Hamiltonian for each prime $p \geq p_0$.
\end{proof}

\section{Covering graphs with large circumference}
\label{sec:circumference}
In the previous section, we show that when a reflexive tree 
with nonzero voltages
is lifted over a large prime cyclic group, the covering graph is Hamiltonian whenever a certain local condition on voltages is met somewhere in the tree. It is natural to ask, however, what can we say when these local conditions are not met? In this section, we show that when any reflexive tree is lifted over a large prime cyclic group, the covering graph has a cycle containing almost all of the vertices, regardless of the voltage assignments of the base graph. In fact, we will allow the order $n$ of our group not to be prime as long as all labels on $T$ are coprime to $n$.

%We define $\phi$ and $\sigma$ as in the previous section.

\begin{theorem}
\label{thm:almostAll}
Let $\Delta \geq 0$ be an integer, let $0 < \epsilon \leq \frac{1}{2}$, and let $n \geq \frac{5 \Delta}{\epsilon^2} $ be an integer. If $T$ is a reflexive tree of maximum degree at most $\Delta$, and if $\sigma: A(T) \rightarrow \mathbb Z_n$ is a mapping that assigns group elements coprime to $n$ to each loop of $T$, then there exists a cycle $C$ on $T^{\sigma}$ that contains at least $ ( 1  - \epsilon)|V(T^{\sigma})|$ vertices.
\end{theorem}

Rather than proving Theorem \ref{thm:almostAll} directly, we instead prove the following stronger but more opaque theorem.

\begin{theorem}
\label{thm:ugly}
Let $n$ and $\omega$ be positive integers satisfying $n \left ( 1 - \frac{1}{ \lceil \sqrt{\omega} \rceil } - \frac{\lceil \sqrt{\omega} \rceil }{\omega} \right ) \geq \lceil \sqrt{\omega} \rceil$. If $T$ is a reflexive tree of maximum degree at most $\frac{n}{\omega}$, and if $\sigma: A(T) \rightarrow \mathbb Z_n$ is a mapping that assigns group elements coprime to $n$ to each loop of $T$, then there exists a cycle $C$ on $T^{\sigma}$ that contains at least $\left ( 1 - \frac{1}{ \lceil \sqrt{\omega} \rceil } - \frac{\lceil \sqrt{\omega} \rceil }{\omega} \right )|V(T^{\sigma})|$ vertices.
\end{theorem}

\begin{claim}
Theorem \ref{thm:ugly} implies Theorem \ref{thm:almostAll}.
\end{claim}
\begin{proof}
We assume that Theorem \ref{thm:ugly} holds. If $\Delta = 0$ in Theorem \ref{thm:almostAll} then the theorem is trivial, so we assume that $\Delta \geq 1$.
We show that the assumptions of Theorem \ref{thm:almostAll} satisfy the assumptions of Theorem \ref{thm:ugly}, and we also show that the conclusion of Theorem \ref{thm:ugly} implies the conclusion of Theorem \ref{thm:almostAll}.

Let $n$, $\Delta$, and $\epsilon$ be chosen as in Theorem \ref{thm:almostAll}.
We choose $\omega$ to be the smallest integer so that $\frac{1}{ \lceil \sqrt{\omega} \rceil } + \frac{\lceil \sqrt{\omega} \rceil}{\omega} < \epsilon$, and we observe that $\epsilon \leq \frac{1}{ \lceil \sqrt{\omega-1} \rceil } + \frac{\lceil \sqrt{\omega-1} \rceil}{\omega-1}$. We must show that if $n \geq \frac{5 \Delta}{\epsilon^2}$ holds, then so do both inequalities in the stronger statement. We note that since $\frac{1}{ \lceil \sqrt{\omega} \rceil } + \frac{\lceil \sqrt{\omega} \rceil}{\omega} \geq \frac{1}{2}$ for $1 \leq \omega \leq 16$, it holds that $\omega \geq 17$. 

First, we would like to show that $n \left ( 1 - \frac{1}{ \lceil \sqrt{\omega} \rceil } - \frac{\lceil \sqrt{\omega} \rceil }{\omega} \right ) \geq \lceil \sqrt{\omega} \rceil$. Since $n \geq \frac{5}{\epsilon^2}$, and since $\epsilon \leq \frac{1}{ \lceil \sqrt{\omega-1} \rceil } + \frac{\lceil \sqrt{\omega-1} \rceil}{\omega-1}$, it is enough to show that 
\[\frac{5}{ \left (\frac{1}{ \lceil \sqrt{\omega-1} \rceil } + \frac{\lceil \sqrt{\omega-1} \rceil}{\omega-1} \right )^2} \geq \frac{\lceil \omega \rceil }{1 - \frac{1}{ \lceil \sqrt{\omega} \rceil } - \frac{\lceil \sqrt{\omega} \rceil }{\omega}} \]
holds for $\omega \geq 17$.
However, we have 
\[\frac{\lceil \omega \rceil }{1 - \frac{1}{ \lceil \sqrt{\omega} \rceil } - \frac{\lceil \sqrt{\omega} \rceil }{\omega}}  \left ( \frac{1}{ \lceil \sqrt{\omega-1} \rceil } + \frac{\lceil \sqrt{\omega-1} \rceil}{\omega-1} \right ) ^2 \rightarrow 4\]
as $\omega$ increases, and it is easy to check for small values of $\omega$ that the inequality holds. Therefore, it holds that $n \left ( 1 - \frac{1}{ \lceil \sqrt{\omega} \rceil } - \frac{\lceil \sqrt{\omega} \rceil }{\omega} \right ) \geq \lceil \sqrt{\omega} \rceil$.

Second, we would like to show that $\Delta \leq \frac{n}{\omega}$ holds. Since $n \geq \frac{5 \Delta}{\epsilon^2}$, and since $\epsilon \leq \frac{1}{ \lceil \sqrt{\omega-1} \rceil } + \frac{\lceil \sqrt{\omega-1} \rceil}{\omega-1}$, it is sufficient to show that 
\[5 \geq \omega \left ( \frac{1}{ \lceil \sqrt{\omega-1} \rceil } + \frac{\lceil \sqrt{\omega-1} \rceil}{\omega-1} \right )^2.\]
However, the right hand side expression has a limit of $4$ as $\omega$ increases, and it is easy to check the inequality for small $\omega$. Therefore, it holds that $\Delta \leq \frac{n}{\omega}$. Hence, we have shown that the hypotheses of Theorem \ref{thm:ugly} hold, so by assumption, the conclusions of Theorem \ref{thm:ugly} hold as well.

Now, since $\frac{1}{ \lceil \sqrt{\omega} \rceil } + \frac{\lceil \sqrt{\omega} \rceil}{\omega} < \epsilon$, it is clear that the conclusion of Theorem \ref{thm:ugly} implies that the conclusion of Theorem \ref{thm:almostAll} also holds. This completes the proof of the claim.
\end{proof}

For the proof of Theorem \ref{thm:ugly}, we need the following lemma. For an integer $n \geq 1$ and a pair $g,h \in \mathbb Z_n$, we say that the \emph{distance} between $g$ and $h$ is the minimum number of terms $1$ or $-1$ that must be added to $g$ to obtain $h$.
\begin{lemma}
\label{lem:sqrtn}
Let $1 \leq m \leq n$ be integers. For each generator $g \in \mathbb Z_n$, there exists an integer $1 \leq k \leq m$ and an element $h \in \mathbb Z_n$ at a distance of at most $\lfloor \frac{n}{m} \rfloor$ from $0$
for which $kg = h$.
\end{lemma}
\begin{proof}
Consider the set $K = \{k g:1 \leq k \leq m\} \subseteq \mathbb Z_n$. 
For each element $a \in K$, we define the set $R_a \subseteq \mathbb Z_n$ as the set $\{a, a+1, \dots, a+ \lfloor \frac{n}{m} \rfloor \}$. For two distinct elements $a,b \in K$, if $R_a \cap R_b \neq \emptyset$, then $a$ and $b$ are at a distance of at most $\lfloor \frac{n}{m} \rfloor$. Now, since $g$ is a generator of $\mathbb Z_n$, all elements of $K$ are distinct. Therefore, since $m (\lfloor \frac{n}{m} \rfloor +1) > n$, there exist two elements $k_1g, k_2g \in K$ for which $R_{k_1g}$ and $R_{k_2 g}$ intersect. Without loss of generality, we assume that $k_2 g \in R_{k_1 g}$, which implies that modulo $n$, $k_2 g - k_1 g$ is at most $\lfloor \frac{n}{m} \rfloor$. However, this implies that $|k_2 - k_1| g$ is at a distance of at most $\lfloor \frac{n}{m} \rfloor$ from $0$, so letting $k = |k_2 - k_1|$ and letting $h = |k_2 - k_1|g$ gives us our result.
\end{proof}

Now, we are ready to prove Theorem \ref{thm:ugly}, which implies Theorem \ref{thm:almostAll}.

\begin{proof}[Proof of Theorem \ref{thm:ugly}]
By Lemma \ref{lemZero}, we may assume that every cut-edge $e \in E(T)$ satisfies $\phi(e) = 0$.
We give an orientation to the cut-edges of $T$ so that each vertex of $T$ has out-degree at most $1$.
We prove the stronger statement that we may choose $C$ so that $C$ contains all but at most $\deg^+(v) \frac{n}{\lceil  \sqrt{\omega} \rceil}  + \deg^-(v)  \lceil \sqrt{\omega} \rceil $ edges from each fiber $v^{\sigma}$ in $T^{\sigma}$ and such that $E(C) \cap E(v^{\sigma})$ forms a path.

We prove the statement by induction on $|V(T)|$. When $|V(T)| = 1$, $T^{\sigma}$ is a single cycle, so the statement holds. Now, suppose $|V(T)| > 1$.
Let $\ell$ be a leaf of $T$ with out-degree $1$ and a neighbor $v \in V(T)$. We assume by the induction hypothesis that $(T - \ell)^{\sigma}$ contains a cycle $C'$ that satisfies our stronger condition, and in particular, that contains all but at most $\deg^+(v) \frac{n} {\lceil  \sqrt{\omega} \rceil}  + ( \deg^-(v)  -1 )   \lceil \sqrt{\omega} \rceil $ edges of $v^{\sigma}$ and such that $E(C') \cap E(v^{\sigma})$ is a path.

Since $\sigma(v)$ is coprime to $n$, we may assume by relabelling our group elements that $\sigma(v) = 1$.
Now, we extend $C'$ to $\ell^{\sigma}$ as follows. Using Lemma \ref{lem:sqrtn}, we choose an integer $1 \leq k \leq \frac{n}{\lceil \sqrt{\omega}\rceil }$ and an element $h \in \mathbb Z_n$ at a distance $d \leq \lceil \sqrt{\omega} \rceil$ from $0$ for which $k  \sigma(\ell) = h$. Now, since $C'$ intersects $v^{\sigma}$ in at least 
\[n -  \frac{n}{\lceil  \sqrt{\omega} \rceil} - \frac{n}{\omega} \cdot \lceil  \sqrt{\omega} \rceil \geq \lceil \sqrt{\omega} \rceil \] edges, and since these edges form a path in $v^{\sigma}$, we may find some path $P \subseteq v^{\sigma} \cap C'$ of length at least $ \lceil \sqrt{\omega} \rceil $, and by applying an automorphism to $T^{\sigma}$, we may assume that $P$ is of the form $(v_0, v_1, \dots, v_{|E(P)|})$ and that $v_0$ is an endpoint of the path $E(C') \cap E(v^{\sigma})$. We remove from $C'$ all edges on the subpath $(v_0, v_1, \dots, v_d)$, and we add to $C'$ the edges $[v_0, \ell_0]$ and $[v_d, \ell_d]$. Now, since $d \in \{h, -h\}$ modulo $n$, and since $k \sigma(\ell) = h$, there exists a path of length $k$ from $\ell_0$ to $\ell_d$ in $\ell^{\sigma}$. Then, since $\ell^{\sigma}$ is a cycle, there also exists a path $P'$ of length $n - k$ in $\ell^{\sigma}$ from $\ell_0$ to $\ell_d$. We add this path $P'$ to $C'$, which gives us our final cycle $C$.

Finally, we check that $C$ satisfies all conditions of the induction hypothesis. First, we note that $C$ contains all but $k \leq \frac{n}{\lceil \sqrt{\omega}\rceil }$ edges of $\ell^{\sigma}$. Additionally, when $T - \ell$ was extended to $T$ and $C'$ was extended to $C$, the in-degree of $v$ increased by one, and $C$ lost $d \leq \lceil \sqrt{\omega}\rceil$ edges from $v^{\sigma}$ compared to $C'$.  Therefore, $C$ contains all but at most $  \frac{n}{\lceil \sqrt{\omega} \rceil} \deg^+(w) + \lceil \sqrt{\omega}\rceil \deg^-(w)$ edges from each fiber $w^{\sigma}$ in $T^{\sigma}$. Finally, $E(C) \cap E(\ell^{\sigma})$ is a path, as is $E(v^{\sigma}) \cap E(C)$,
since $E(v^{\sigma}) \cap E(C)$ was obtained from the path $E(v^{\sigma}) \cap E(C')$ by removing a subpath containing an endpoint.
Therefore, the induction hypothesis holds for $T$ and $C$, and the proof is complete.
\end{proof}

\section{Acknowledgment}
We are grateful to an anonymous referee for carefully reading the manuscript, and in particular, for pointing out mistakes in previous versions of Lemma \ref{lemma:billiardPaths} and Theorem \ref{thm:prob}.

\raggedright
\bibliographystyle{plain}
\bibliography{bib}

\end{document}